\documentclass[11pt]{article}
\usepackage[noadjust]{cite}
\usepackage[margin=1in]{geometry}
\usepackage{setspace}
\usepackage{color}

\usepackage{amsmath}    
\usepackage{amssymb}    
\usepackage{amsthm}     
\usepackage{mathtools}  
\usepackage{dsfont}     
\usepackage{accents}    

\usepackage{graphicx}
\usepackage{booktabs}   
\usepackage{adjustbox}
\usepackage{multirow}   
\usepackage{caption}
\usepackage{subcaption}

\usepackage{tikz}
\usetikzlibrary{shapes,arrows}
\usepackage{tikz-cd}

\usepackage{enumerate}
\usepackage[ruled,vlined]{algorithm2e}

\usepackage{authblk}    
\usepackage[normalem]{ulem}  
\usepackage{textcomp}   
\usepackage{appendix}   

\newtheorem{theorem}{Theorem}
\newtheorem{lemma}{Lemma}

\usepackage[colorlinks=true, allcolors=blue]{hyperref}
\usepackage{float}
\makeatletter
\newcases{centercases}{\quad}
  {\hfil\m@th\displaystyle{##}\m@th\displaystyle{##}\hfil}
  {\m@th\displaystyle{##}\m@th\displaystyle{##}\hfil}{\lbrace}{.}
\makeatother

\usepackage{grffile}
\usetikzlibrary{decorations.pathreplacing}
\usepackage{wrapfig}
\usepackage{csquotes}

\newcommand{\commentout}[1]{}

\DeclareMathOperator*{\argmin}{arg\,min}

\newtheorem{remark}{Remark}[section]
\theoremstyle{plain}

\theoremstyle{definition}

\newcommand{\bc}{\mathbf{c}}

\newcommand{\bx}{\mathbf{x}}

\newlength{\dhatheight}

\begin{document}
	
	\title{WG-IDENT: Weak Group Identification of PDEs with Varying Coefficients}
	\author{
 Cheng Tang\thanks{Department of Mathematics, Hong Kong Baptist University, Kowloon Tong, Hong Kong. Email: 22481184@life.hkbu.edu.hk. }, 
 Roy Y. He\thanks{Department of Mathematics, City University of Hong Kong, Tat Chee Avenue, Kowloon, Hong Kong.
	Email: royhe2@cityu.edu.hk.},  
 Hao Liu\thanks{Department of Mathematics, Hong Kong Baptist University, Kowloon Tong, Hong Kong.
		Email: haoliu@hkbu.edu.hk.}
  }
	
	\date{}
	\maketitle

	\begin{abstract}
The identification of Partial Differential Equations (PDEs) has emerged as a prominent data-driven approach for mathematical modeling and has attracted considerable attention in recent years. The stability and precision in identifying PDE from heavily noisy spatiotemporal data present significant difficulties. This problem becomes even more complex when the coefficients of the PDEs are subject to spatial variation. In this paper, we propose a \textbf{W}eak formulation of \textbf{G}roup-sparsity-based framework for \textbf{IDENT}ifying PDEs with varying coefficients, called \textbf{WG-IDENT}, to tackle this challenge. Our approach utilizes the weak formulation of PDEs to reduce the impact of noise.
We represent test functions and unknown PDE coefficients using B-splines, where the knot vectors of test functions are optimally selected based on spectral analysis of the noisy data. To facilitate feature selection, we propose to integrate group sparse regression with a newly designed group feature trimming technique, called GF-Trim, to eliminate unimportant features. Extensive and comparative ablation studies are conducted to validate our proposed method.
The proposed method not only demonstrates greater robustness to high noise levels compared to state-of-the-art algorithms but also achieves superior performance while exhibiting reduced sensitivity to hyperparameter selection.

	\end{abstract}

\section{Introduction}
The use of differential equations to model empirical observations provides a fundamental framework for understanding the governing principles of complex systems across diverse fields, including physics, environmental science, and biomedical engineering. Advances in data acquisition technologies have led to the availability of vast amounts of data. This has facilitated the emergence of data-driven methods for model discovery, which complement classical modeling approaches. Various algorithms based on symbolic regression~\cite{JBongard_symbolic,Schaeffer_sparse_reg} and sparse-regression~\cite{Discoveringgoverning_sparsereg,Schaeffer2017LearningPD_sparsereg,Data_drivenPDE_sparsereg,IDENT_sparsereg,he2022robust, Messenger_2021_sparse_reg, Rudy_sparse_reg, tang2022weakident,  WU2019200,  Messenger_2021} have been proposed  under diverse settings. More recent work also involves deep learning techniques~\cite{long2018pdenet,long2019pde}. Most of them~\cite{Qin_2019, chen2025duedeeplearningframework, NEURIPS2023_70518ea4,chen2022deep,churchill2025principal,wu2020data} focus on learning the underlying evolution operators (flow maps) directly from data, providing powerful surrogate models but without yielding explicit PDE structures. For instance, regression-based approaches for identifying partial differential equations (PDEs) often start with a generic form
\begin{equation} \label{pde_form}
\partial_t u = F(\bx, \partial_\bx u, \partial^2_\bx u,\dots),
\end{equation}
where the right-hand side represents a combination of terms from a predefined feature dictionary. The dictionary contains various spatial derivatives and nonlinear combinations that are hypothesized to be relevant to the system. The function $F$ represents the governing equation of the system and is assumed to be a linear combination of a small subset of the dictionary. This reflects the general principle that physically meaningful PDEs are typically sparse and involve only smooth functional forms, which both enhances interpretability and reduces the risk of overfitting. Consequently, the objective of PDE identification is to identify a sparse vector of coefficients that determine the contribution of each feature in the dictionary to the overall equation.

One of the primary difficulties in PDE identification arises from \textbf{noise amplification} during numerical differentiation, as illustrated in Figure~\ref{fig:noise_differentiation}. Numerical differentiation magnifies the noise in the original signal, rendering the derivative of the noisy signal considerably more erratic and less reliable than that of the clean signal. In many regression-based methods~\cite{he2022robust,he2023group,IDENT_sparsereg,Li_Sun_Zhao_Lehman_2020,Data_drivenPDE_sparsereg, Rudy_sparse_reg,Schaeffer2017LearningPD_sparsereg}, partial derivatives of the observed trajectories must be approximated to construct a dictionary of candidate features.  Since the identification is highly sensitive to approximation errors in these features, various techniques—such as  total variation~\cite{Schaeffer2017LearningPD_sparsereg},  local polynomial fitting~\cite{Data_drivenPDE_sparsereg, Rudy_sparse_reg}, least-squares moving average~\cite{IDENT_sparsereg}, and Successively Denoised Differentiation (SDD)~\cite{he2022robust}—have been proposed to enhance robustness. In general, achieving accurate estimates of partial derivatives requires striking a careful balance between the sampling frequency and the shapes of the smoothing kernels~\cite{he2022asymptotic}. 

\begin{figure}
    \centering
    \includegraphics[width=0.6\textwidth]{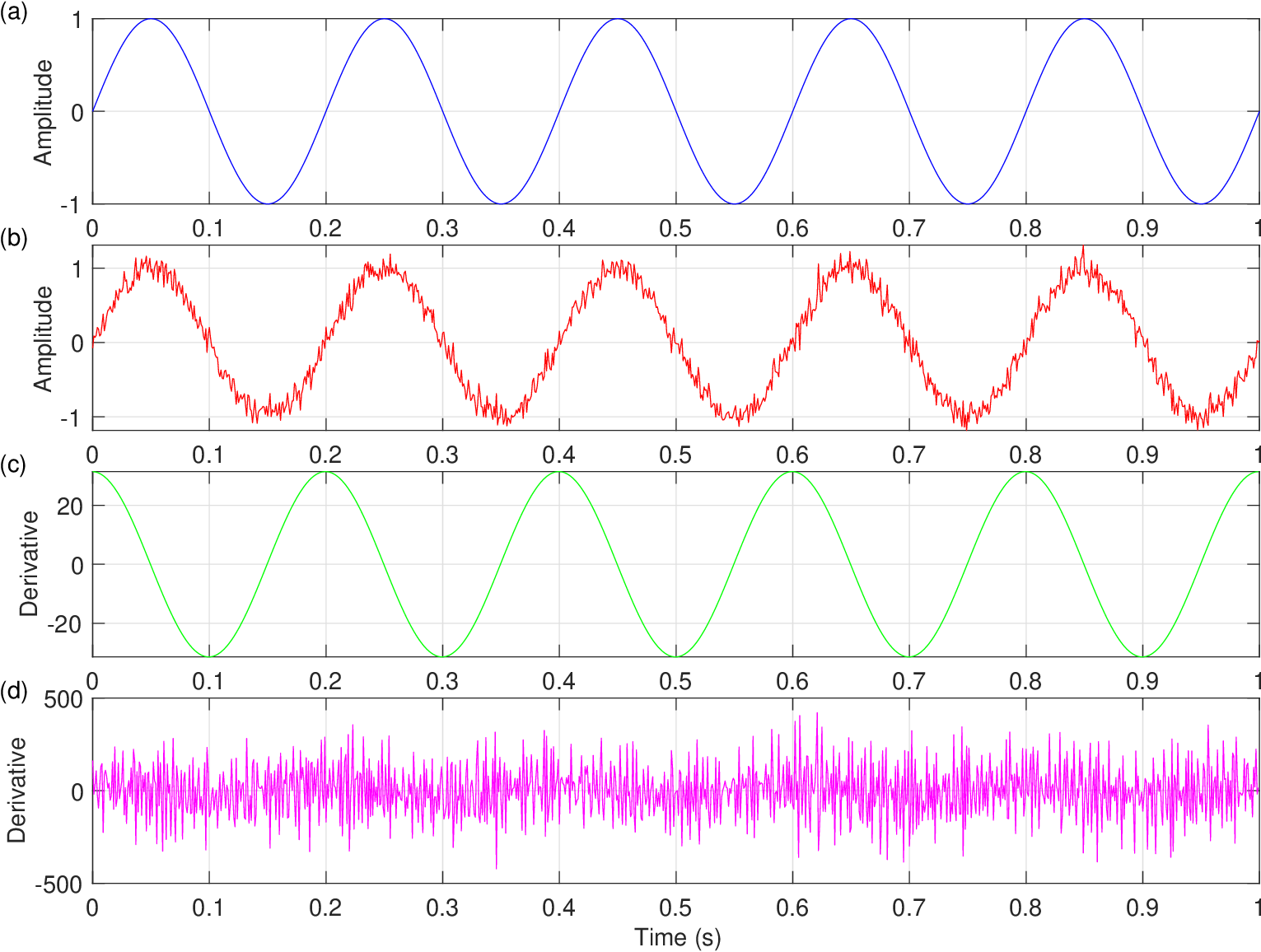}
    \caption{Demonstration of noise amplification in numerical differentiation. (a) Clean sinusoidal signal; (b) Noisy signal obtained by adding to (a) independent Gaussian noise of standard deviation 0.1; (c) Numerical derivative of (a); (d) Numerical derivative of (b). Here the derivatives are estimated via forward difference approximation. }
    \label{fig:noise_differentiation}
\end{figure}

The second difficulty arises because the observed dynamics often occur in \textbf{non-uniform environments}. For instance, in biological aggregation phenomena, organisms such as cells or animals interact through nonlocal potentials while moving in heterogeneous habitats with varying resources or obstacles~\cite{mogilner1999nonlocal,pmlrv178yao22a,he2022numerical}. In physical systems, interactions among particles can be influenced by external fields or medium heterogeneities~\cite{feng2024data,carrillo2024sparse}. As a result, the problem of identifying PDE coefficients becomes an infinite-dimensional one. Common strategies for reducing it to a finite-dimensional setting include finite element expansions~\cite{IDENT_sparsereg,Rudy_sparse_reg}, B-spline approximations~\cite{he2023group}, and implicit neural representations~\cite{long2018pdenet,long2019pde}. These methods can enhance the adaptability of the identification framework; however, the adverse effect of noise becomes more pronounced due to the increased degrees of freedom.

To address the first challenge of noise amplification in numerical differentiation, the weak formulation provides an effective framework~\cite{Messenger_2021, Messenger_2021_sparse_reg,tang2022weakident,Reinbold2019UsingNO}. More specifically, weak-form PDE identification in the  \textbf{constant-coefficient setting} seeks to identify equations of the form:
\begin{equation}  
\partial_t u = a_1 \partial_{\bx}^{\alpha_1} f_1(u) + a_2 \partial_{\bx}^{\alpha_2} f_2(u) + \cdots + a_K \partial_{\bx}^{\alpha_K} f_K(u),  
\label{eq_model_general_divergence_form_constantcoef}  
\end{equation}
where \(u: \Omega \times [0,T] \to \mathbb{R}\) is the observed function, \(a_1\ldots a_K\) represent the constant coefficients, and \(\partial_{\bx}^{\alpha_k}\) denotes the partial derivative associated with a multi-index \(\alpha_k\). Each \(f_k\) encodes a possible nonlinear function of \(u\). \(K\) denotes the total number of candidate features in the dictionary. By integrating the equation against smooth test functions over localized spatiotemporal domains, the weak form avoids direct computation of high-order derivatives from noisy data. Multiplying both sides of the PDE by compactly supported test functions and applying integration by parts, which transfers derivatives onto the test functions, avoids direct numerical differentiation from noisy measurements. This transformation inherently acts as a low-pass filter, suppressing high-frequency noise while preserving the essential dynamics. Moreover, the spectral properties of the noisy data can be exploited to optimize the choice of test functions for noise suppression. Previous studies~\cite{Messenger_2021, Messenger_2021_sparse_reg,tang2022weakident,Reinbold2019UsingNO} have demonstrated the effectiveness of weak-form-based approaches for constructing overdetermined linear systems from sparse measurements across multiple test functions. This methodology enables robust identification of PDEs from noisy data and allows stable coefficient recovery even in the presence of significant noise.

Despite these successes in constant-coefficient settings, a critical challenge remains in extending this framework to heterogeneous systems with spatially varying coefficients. The infinite-dimensional nature of coefficient functions poses both computational and theoretical difficulties. Existing methods are not readily applicable to the identification of varying coefficients, particularly when the data are further corrupted by noise. This gap motivates the development of novel weak-form strategies designed for heterogeneous systems.

To address this gap, in this paper, we focus on identifying PDEs with \textbf{spatially varying coefficients}, which can be expressed in divergence form as
\begin{equation}  
\partial_t u = c_1(\bx) \partial_{\bx}^{\alpha_1} f_1(u) + c_2(\bx) \partial_{\bx}^{\alpha_2} f_2(u) + \cdots + c_K(\bx) \partial_{\bx}^{\alpha_K} f_K(u),  
\label{eq_model_general_divergence_form}  
\end{equation}  
where \(\bx \in \mathbb{R}^d\) and  \(c_k: \Omega \to \mathbb{R}\) denote spatially varying coefficients. 
This formulation encompasses a wide range of PDEs, including those with multiple nonlinear terms and high-order derivatives. For instance, the one-dimensional viscous Burgers' equation with varying coefficients can be recovered by setting \(f_1(z)=z\), \(f_2(z)=z^2\), and selecting appropriate partial derivatives (e.g., \(\partial_x\), \(\partial_x^2\)), with other coefficients set to zero.  A feature is considered \emph{active} if its corresponding coefficient \(c_k(\bx)\) is non-zero for some \(\bx \in \Omega\).

To avoid noise amplification during numerical differentiation, the weak-form approach multiplies both sides of equation~\eqref{eq_model_general_divergence_form} by a compactly supported smooth test function \( \varphi \in C_c(\Omega \times (0, T)) \) and then integrates to obtain 
\begin{equation}  
\int_{\Omega \times (0, T)} \partial_t u(\bx,t)\cdot \varphi(\bx,t) \, d\bx \, dt = \sum_{k=1}^K \int_{\Omega \times (0, T)} c_k(\bx)\partial_{\bx}^{\alpha_k} f_k(u(\bx,t))  \varphi(\bx,t) \, d\bx \, dt.  
\end{equation}  
Integration by parts leads to the weak formulation:  
\begin{equation}  
-\langle \partial_t \varphi, u \rangle = \sum_{k=1}^K (-1)^{|\alpha_k|} \left\langle \partial_{\bx}^{\alpha_k} (\varphi c_k), f_k(u) \right\rangle,
\label{eq_model_general_weak}  
\end{equation}  
where $|\alpha_k|$ denotes the order (multi-index length) of the derivative operator $\partial_{\bx}^{\alpha_k}$. With properly chosen test functions \(\varphi\), this weak formulation effectively suppresses high-frequency noise. The denoising effect is achieved by integrating shifted versions of \(\varphi\), which effectively acts as a convolution under periodic boundary conditions~\cite{Messenger_2021,tang2022weakident}. 

In this paper, we propose a \textbf{W}eak formulation of \textbf{G}roup-sparsity-based framework for \textbf{IDENT}ifying PDEs with spatially varying coefficients,  named \textbf{WG-IDENT}. Specifically, our WG-IDENT employs to use B-splines~\cite{schumaker2007spline,unser1993b1,unser1993b2} as both the test functions and the bases for approximating the unknown varying coefficients, which is in contrast to previous works that mainly focused on PDEs with constant coefficients. Existing approaches~\cite{Messenger_2021, tang2022weakident}, typically employ truncated polynomial test functions due to their simple construction. Unlike truncated polynomial test functions, the B-spline test functions adopted in our framework form a partition of unity, thereby providing consistent weighting across the domain and improved stability in noisy settings. Additionally, It is well known that smooth varying coefficients can be approximated to arbitrary accuracy by uniform B-splines, similar in spirit to the B-spline Galerkin finite element method~\cite{fischer2009b,iqbal2020cubic,gorgulu2018numerical} where spline bases are employed to approximate solutions. Such a reduction from an infinite-dimensional problem to a finite-dimensional one allows us to cast the task of PDE identification as a group-sparse feature selection problem~\cite{he2023group,he2024groupprojectedsubspacepursuit}. Meanwhile, we note that~\eqref{eq_model_general_weak} is equivalent to passing the data and the derived features through a low-pass filter determined by the test function~\cite{Messenger_2021,tang2023fourier,he2023much}. The spectral decay of uniform B-splines is closely related~\cite{Neuman} to the order of the basis functions, the knot spacing, and the number of knots. We design an effective scheme (See Section~\ref{subsec:testfnc}) to adaptively construct an optimal set of test functions to suppress high-frequency noise.

Our method generates a pool of candidate PDEs using Group Projected Subspace Pursuit (GPSP)~\cite{he2023group,he2024groupprojectedsubspacepursuit}. To refine this pool, we introduce a technique called Group Feature Trimming (GF-Trim) to facilitate feature selection. GF-Trim extends the trimming strategy of~\cite{tang2022weakident} to accommodate PDEs with varying coefficients. This technique effectively eliminates unimportant group features, thereby enhancing the identifiability of true features.

Our contributions are summarized as follows:

\begin{enumerate}
    \item We propose \textbf{WG-IDENT}, a novel Weak-form-based Group-sparsity framework for identifying PDEs with varying coefficients.  Our method successfully identifies a broad class of PDEs even in the presence of substantial noise, outperforming existing methods.
    \item  We derive an adaptive scheme for selecting the support of the test functions based on the spectral analysis of noisy data. Numerical experiments show that our approach significantly improves robustness in identification.
    \item We develop a Group Feature Trimming (GF-Trim) technique to eliminate unimportant group features from the candidate PDEs, enhancing the efficiency and stability of identification.
    \item We conduct comprehensive numerical experiments, including ablation studies and comparisons with state-of-the-art algorithms, to demonstrate the effectiveness of our method.
\end{enumerate}

This paper is organized as follows. In Section \ref{sec:proposed_algorithm}, we formulate the problem in weak form as a group-sparse problem. In Section \ref{sec:alg.overview}, we introduce the proposed method, WG-IDENT, for identifying PDEs with varying coefficients, which includes candidate PDE generation (Section~\ref{sec:candidate_gen}), group feature trimming for feature refinement (Section~\ref{sec:G_trim}) and residual-based model selection (Section~\ref{sec:RR}). Implementation details are discussed in Section~\ref{sec:implementation_details}. In Section~\ref{sec:numerical_experiments}, we demonstrate the effectiveness of the proposed method by comprehensive numerical experiments. Finally, Section~\ref{sec:conclusion} concludes the paper.

\section{Problem formulation} \label{sec:proposed_algorithm}

We discuss the formulation of PDE identification.
For simplicity, we present the settings for one-dimensional problem. Extension to higher dimensions is straightforward. 

Consider a space-time domain $\Omega\times [0,T]$ where $\Omega = [L_1,L_2]$ with $0\leq L_1<L_2$ and $T>0$. We focus on evolutionary PDEs of form~\eqref{pde_form} that can be expressed as a linear combination:
\begin{align}
		\partial_t u(x,t) = \sum_{k=1}^K c_k(x)F_k(u), ~\text{with} ~F_k = \frac{\partial^{\alpha_k}}{\partial x^{\alpha_k}}f_k(u)\in L^\infty(\Omega\times [0,T]),
  \label{eq_ori}
	\end{align}
 where $\alpha_k\in\{0,1,\ldots,\alpha_\text{max}\}$ denotes the order of spatial derivatives, with $\alpha_{\text{max}}$ being the highest order. The functions $f_k$ are smooth (e.g., constants, polynomials), and $c_k(x)$ defined on $[L_1,L_2]$ represents spatially varying PDE coefficients. We assume that (\ref{eq_ori}) satisfies the periodic boundary condition in space, and the right-hand side is sparse, i.e., $c_k$ is constant zero for most $k$.

Using 
$N^x+1$ spatial grid points and $N^t+1$ temporal grid points, we define step sizes $\Delta x := \frac{L_2-L_1}{N^x}$ and $\Delta t := \frac{T}{N^t}$, respectively. The given data (discretized PDE solution with noise) is denoted by
\begin{align}
\mathcal{D}=\{U_i^n = u(x_i,t^n)+\varepsilon_i^n, i = 0, 1,\ldots,N^x, n = 0, 1,\ldots,N^t\}\label{dataset},
\end{align}
where $x_i=i\Delta x \in [L_1,L_2] $ and $t^n=n\Delta t\in [0,T]$; $\varepsilon_i^n$ represents independent zero-mean noise; and $u(x_i,t^n)$ represents the clean data, for $i = 0, 1,\ldots,N^x,$ and $ n = 0, 1,\ldots,N^t$.

Our objective is to identify the expression of (\ref{eq_ori}) from the dataset $\mathcal{D}$ by finding the correct terms from the dictionary $\{F_k(u)\}_{k=1}^K$ with the corresponding varying coefficients. In our method, we combine the strengths of B-spline functions with the Galerkin method, employing B-splines as both test functions and bases for representing varying coefficients in PDEs discovery. 

\subsection{Proposed basis and test functions}\label{sec:proposed_basisNtestfnc}
Our approach employs B-splines in two key roles. Firstly, for approximating varying coefficients, each $c_k(x)$ in~\eqref{eq_ori} is represented as a linear combination of B-spline basis functions $\{\psi_m(x)\}_{m=1}^M$:
\begin{align}\label{eq_varyingcoef_appr}
    c_k(x)\approx\widetilde{c}_k(x) =\sum_{m=1}^M c_{k,m} \psi_m(x).
\end{align}
where $c_{k,m}\in\mathbb{R}$ is the $m$-th coefficient in the basis expansion approximation for the $k$-th feature. 
Secondly, we construct test functions $\{\varphi_r\}_{r=1}^S$ as tensor products of B-splines to capture both spatial and temporal variations:

\begin{equation}\label{eq_testfnc_tensor}
\varphi_r(x,t) = B_{r_x}(x)B_{r_t}(t), \quad r=1,\dots,S,
\end{equation}
where $B_{r_x}(x)$ and $B_{r_t}(t)$ represent B-spline basis functions in space and time, respectively.

The effectiveness of this framework depends on careful parameter selection, particularly the number and order of B-spline bases \cite{799930}. While the number of basis functions influences approximation accuracy\cite{HUGHES20054135}, our numerical experiments demonstrate that identification results remain stable across a range of number of basis (detailed in \ref{Appendix:numberofbasis}). For test functions, we develop a comprehensive parameter selection strategy (Section \ref{subsec:testfnc}) that optimizes their filtering properties, effectively capturing signal-dominant regions while suppressing noise.

To ensure the validity of our coefficient approximation approach, we recall a fundamental property of well-posed PDEs: small perturbations in the coefficients lead to proportionally small errors in the solution. Specifically, when the underlying PDE is well-posed (i.e., it admits a unique solution that depends continuously on coefficients and initial data), sufficiently accurate approximations of spatially varying coefficients via B-splines guarantee that the resulting PDE solution remains close to the original unperturbed solution.

\subsection{PDE Identification as group sparse regression}\label{sec:group_sparse_regression}

The objective of PDE identification is to determine a sparse set of coefficients that can interpret and represent the dynamics of observed data. If we want to identify a PDE with sparsity $\theta$, based on the weak formulation (\ref{eq_model_general_weak}) and incorporating the varying-coefficient approximation~\eqref{eq_varyingcoef_appr} with B-spline test functions~\eqref{eq_testfnc_tensor}, we can formulate  the identification problem as the following group sparse regression problem:
\begin{equation}\label{eq_regression_model}
\begin{aligned} 
\bc^\theta\in\arg\min_{c_{k,m}}&\sum_{r=1}^S\left(\langle \partial_t\varphi_r,u\rangle + \sum_{k=1}^K(-1)^{|\alpha_k|}\left\langle \partial_{\bx}^{\alpha_k}\left(\varphi_r\sum_{m=1}^M c_{k,m}\psi_m\right), f_k(u)\right\rangle\right)^2\\
\text{s.t.}&\quad\#\{\|\bc_k\|_1\neq 0\}=\theta.
\end{aligned}
\end{equation}
In~\eqref{eq_regression_model}, $\langle\cdot,\cdot\rangle$ denotes the inner product in $L^2(\Omega \times [0,T])$, $c_{k,m} \in \mathbb{R}$ are the sparse coefficients to be estimated, and $\partial_t$ and $\partial_{\bx}^{\alpha_k}$ represent temporal and spatial derivatives, respectively. The sparsity parameter $\theta$ serves as a critical constraint, explicitly limiting the number of non-zero coefficient groups. This formulation allows for a flexible representation of the PDE while maintaining sparsity, which is crucial for identifying the underlying structure of the equation. The non-zero groups of $\mathbf{\mathbf{c}}^\theta$ determines which terms to be included in the corresponding PDE model; thus each solution of~\eqref{eq_regression_model} corresponding to different $\theta$ yields a candidate PDE model.

\section{Proposed  method: WG-IDENT}
\label{sec:alg.overview}
In this section, we introduce our \textbf{W}eak formulation of \textbf{G}roup-sparsity based framework \textbf{IDENT}ification (\textbf{WG-IDENT}) method. 
Our method consists of three steps: generating candidate PDEs with different sparsity levels, refining these candidates, and selecting the optimal PDE.

\noindent \textbf{Step 1: [Generate the candidate PDEs]} (Section~\ref{sec:candidate_gen}) For each sparsity level $\theta=1,...,K,$ where $K$ is the total number of group features, we generate a candidate PDE by solving the group sparse regression problem~\eqref{eq_regression_model} using the Group Projected Subspace Pursuit Algorithm (GPSP)~\cite{he2023group}. \\
\noindent \textbf{Step 2: [Refine the candidate PDEs]} (Section~\ref{sec:G_trim}) For each sparsity level $\theta$, Step 1 gives a group sparse vector. We employ a new technique, Group Feature Trimming (GF-Trim), to rule out unimportant features.\\
\noindent \textbf{Step 3: [Select the optimal PDE]} (Section~\ref{sec:RR}) We use the Reduction in Residual (RR) criterion to assess each trimmed candidate PDE, select the optimal sparsity level $\theta^*$, and recover the corresponding coefficients using least squares regression.

In the following subsections, we discuss each step of the proposed algorithm in details.

\subsection{Candidate generation} \label{sec:candidate_gen}
To facilitate algorithm presentation, we reformulate equation \eqref{eq_regression_model} as a linear system using matrix notation. Let $\phi_{k,m}^{r}=(-1)^{|\alpha_k|}\left\langle \partial_{\bx}^{\alpha_k}\left(\varphi_r\psi_m\right), f_k(u)\right\rangle$, and define matrices:
	\begin{align} \label{linearsystem_mtx}
		\mathbf{F}&=
		\begin{bmatrix}
			\phi^{1}_{1,1}&\phi^{1}_{1,2}&\cdots&\phi^{1}_{1,M}&\phi^{1}_{2,1}&\cdots&\phi^{1}_{K,M}\\
			\phi^{2}_{1,1}&\phi^{2}_{1,2}&\cdots&\phi^{2}_{1,M}&\phi^{2}_{2,1}&\cdots&\phi^{2}_{K,M}\\
			\vdots&\vdots&\vdots&\vdots&\vdots&\vdots&\vdots\\
			\phi^{S}_{1,1}&\phi^{S}_{1,2}&\cdots&\phi^{S}_{1,M}&\phi^{S}_{2,1}&\cdots&\phi^{S}_{K,M}\\
		\end{bmatrix}\in\mathbb{R}^{S\times KM},\\
	 \mathbf{b} &=
		\begin{bmatrix}
			-\langle \partial_t \varphi_{1},u \rangle\\
			\vdots\\
			-\langle \partial_t\varphi_{S},u\rangle
		\end{bmatrix}\in\mathbb{R}^{S}, \quad 
		\mathbf{c} =
		\begin{bmatrix} \label{linearsystem_c}
			c_{1,0}\\
			c_{1,1}\\
			\vdots\\
			c_{1,M}\\
			c_{2,0}\\
			\vdots\\
			c_{K,M}
		\end{bmatrix}\in\mathbb{R}^{KM}.
	\end{align}
Here, $\mathbf{F}$ is the feature matrix that represents the complete dictionary matrix which incorporates the weak formulation of all feature groups; $\mathbf{b}$ is the feature response that represents the vector of the weak formulation of $\partial_tu$; and $\mathbf{c}$ includes all the coefficients corresponding to the feature groups, concatenated together. The group-sparse regression problem \eqref{eq_regression_model} can be rewritten as

\begin{align}
\bc^{\theta}=\argmin_{\substack{\bc\in\mathbb{R}^{KM},\ \|\mathbf{c}\|_{\ell_{1,0}}=\theta}} \|\mathbf{F}\bc - \mathbf{b}\|_2^2,
\label{group_sparse_reg_linearsystem}
\end{align}
where $\|\mathbf{c}\|_{\ell_{1,0}}=\left\|\left[\|\mathbf{c}_{1:M}\|_1,...,\|\mathbf{c}_{((K-1)M+1):(KM)}\|_1\right]^{\top}\right\|_0$ counts the number of non-zero groups in $\mathbf{c}$. Here, the vector is partitioned into \(K\) groups, each comprising $M$ consecutive columns. Specifically, the first group is denoted as \(\bc_{1:M}\), the second group as \(\bc_{M+1:2M}\), and so forth, with the final group being \(\bc_{(K-1)M+1:(KM)}\).

Details of GPSP are shown in
Algorithm~\ref{SPalgo} (\ref{appendixalg}). As highlighted in~\cite{he2023group}, this sparsity-constrained framework effectively avoids generating redundant candidates, making it more efficient than approaches based on sparsity-inducing regularizations.

\subsection{Group feature trimming for feature refinement} \label{sec:G_trim}
Step 1 generates a pool of candidate PDEs with distinct sparsity.
Each candidate PDE comprises a number of group features corresponding to its sparsity level \(\theta\). However, some PDEs may include  features that contribute marginally for describing \( \mathbf{b} \), especially when \( \theta \) exceeds the true underlying sparsity \(\theta^*\). We refer to these features as low-contribution features, which typically do not exist in the actual PDE. Notably, even if the estimated coefficients of a feature group have large magnitudes, the product between the estimated coefficients and its corresponding features may remain small, indicating that the selected group features may be fitting noise rather than the true signal (see~\cite[Example 2]{he2024groupprojectedsubspacepursuit}). 

To mitigate the influence of low-contribution features, we propose Group Feature Trimming (GF-Trim)—a systematic approach for eliminating feature groups with negligible impact on the regression model. This technique is inspired by the trimming method introduced in~\cite{tang2022weakident}. Based on equations~\eqref{linearsystem_mtx} and~\eqref{linearsystem_c}, we define the feature matrix for the selected feature groups as \(\mathbf{F}^\theta \in \mathbb{R}^{S\times \theta M}\) and the corresponding estimated coefficients as \(\mathbf{c}^\theta \in \mathbb{R}^{\theta M}\). For each selected feature group \(v \in \{1,\ldots,\theta\}\), we denote the corresponding feature matrix by \(\mathbf{F}^{\theta}_v \in \mathbb{R}^{S\times M}\) and the coefficient vector by \(\mathbf{c}^{\theta}_v \in \mathbb{R}^M\).

We define a contribution score \( \chi^{\theta}_v \) for the \( v \)-th selected feature group in the candidate model containing \(\theta\) terms as
\begin{align}
\chi^{\theta}_v = \frac{\alpha^{\theta}_v}{\max_{1 \leq v' \leq \theta} \alpha^{\theta}_{v'}} \quad \text{with} \quad \alpha^{\theta}_v = \| \mathbf{F}^{\theta}_v \mathbf{c}^{\theta}_v \|_2, \quad v = 1, 2, \ldots, \theta. \label{eq_score}
\end{align}
The quantity \( \alpha^{\theta}_v \) measures the contribution of the \( v \)-th feature group towards approximating the target vector \( \mathbf{b} \) in the candidate model determined by the group sparsity $\theta$, and \( \chi^{\theta}_v \) represents the normalized contribution. For a fixed threshold \(\tau > 0\), we compare the contribution scores \( \chi^{\theta}_v \) for each selected feature group and remove any feature group for which \( \chi^{\theta}_v < \tau \). After trimming, the coefficients of the remaining feature groups are re-computed. We apply this procedure to each sparsity level to generate a pool of refined PDE candidates. Compared to the conventional column-wise trimming strategies, the proposed GF-Trim operates at the group level, ensuring that the  feature groups with limited contributions are systematically removed. This prevents relevant coefficients within a group from being mistakenly discarded, and improves the overall stability of model selection. As will be demonstrated later in Section~\ref{eq_gp_trim_contribution}, GF-Trim yields more reliable identification results under noisy conditions by providing a wider valid range for the residual-reduction criterion.

\subsection{Model selection via residuals}\label{sec:RR}
To select the optimal PDE from the refined pool of candidates, we adopt the Reduction in Residual (RR) criterion \cite{he2023group}.
For each sparsity $\theta$, we compute the residual for each candidate PDE by:
        \begin{align*}
        q^\theta=||\mathbf{F}^\theta\mathbf{c}^\theta-\mathbf{b}^\theta||_2^2.
	\end{align*}
 Let $L$, $1\leq L<K$, be an integer to be specified. The Reduction in Residual is then calculated as:
 \begin{align}
        s^\theta=\frac{q^\theta-q^{\theta+L}}{Lq^1}, ~\theta = 1,\ldots,K-L.
	\end{align}
 This metric quantifies the average reduction in residual error as the sparsity level $\theta$ increases. A positive value of $s^\theta$ indicates an improvement in accuracy with increasing sparsity. However, it is important to note that the feature index set selected at sparsity level $\theta$ may not necessarily be a subset of that at $\theta+1$. Hence, $q^\theta-q^{\theta+1}$ may not always be positive. When $s^\theta$ becomes small, we opt for the smallest sparsity level $\theta$, instead of selecting the $\theta$ with the smallest $s^\theta$ value. This decision is made to avoid overfitting to noise, ensuring that we do not select a sparsity level that may include irrelevant features. 
 
To formalize our selection process, we introduce a threshold $\rho^\text{R}>0$ and define the optimal sparsity level $\theta^*$ as follows:
        \begin{align*}
        \theta^* = \min\{\theta: 1\leq \theta \leq K-L, s^\theta<\rho^{\text{R}} \}.
	\end{align*}
 The choice of the threshold $\rho^R$ significantly impacts the determination of the optimal sparsity level \(\theta^*\). If $\rho^R$  is set too high, we risk selecting a very low sparsity level, potentially omitting important features. Conversely, if it is set too low, we may include irrelevant features in our selection. We find that the Group Feature Trimming process introduced in Section~\ref{sec:G_trim} enhances the robustness of our method, allowing for a broader valid range of $\rho^{\text{R}}$. This is particularly beneficial when dealing with a large feature dictionary, as it facilitates the effective identification and discarding of candidates with marginal contributions. The details of this enhancement are demonstrated in Section~\ref{sec:benefittrim}.

\begin{remark}
The conditioning of the feature matrix $F$ is closely related to the richness of the underlying solution dynamics. 
If the observed trajectory is relatively flat or lacks sufficient variation, the resulting feature columns become nearly indistinguishable, leading to an ill-conditioned system. 
In degenerate cases (e.g., constant solutions), the governing PDE is not uniquely identifiable, since multiple operators are compatible with the same dynamics, as also emphasized in~\cite{HeZhaoZhong2024FoCM}. 
In practice, our weak formulation alleviates part of this issue by acting as a low-pass filter, but additional stabilization such as ridge regularization can be introduced if necessary. 
Ensuring sufficiently rich initial conditions and solution dynamics is therefore crucial for maintaining a well-conditioned system and robust recovery. 

We also note that the Group Projected Subspace Pursuit (GPSP) algorithm used in WG-IDENT admits convergence guarantees under block-RIP conditions \cite{he2024groupprojectedsubspacepursuit}, which help ensure unique and stable recovery when the feature matrix satisfies such incoherence properties. In this sense, rich solution spectra and favorable matrix conditions together mitigate the risk of confusion between highly correlated terms.

\end{remark}

\section{Implementation details}\label{sec:implementation_details}
\begin{algorithm}[t!]
\caption{\textbf{WG-IDENT}\label{alg.SPalgo1}}
\footnotesize
\KwIn{
    $\mathbf{F} \in \mathbb{R}^{S \times KM}$: Feature matrix.\\
    $\mathbf{b} \in \mathbb{R}^{S}$: Feature response.\\
    $\tau$: Threshold for Group Feature Trimming (Section \ref{sec:G_trim}).\\
    $\rho^R$: Threshold for Reduction in Residual (RR) criterion.\\
    $L$: Parameter for RR criterion (Section~\ref{sec:RR}).
}

\textbf{Step 1: Generate Candidate PDEs} \\
\For{$\theta = 1$ \KwTo $K$}{
    Compute $\mathbf{c}^\theta = \underset{\substack{\mathbf{c} \in \mathbb{R}^{KM}, \\ \|\mathbf{c}\|_{\ell_{1,0}} = \theta}}{\arg\min} \|\mathbf{F}\mathbf{c} - \mathbf{b}\|_2^2$ using the Group Projected Subspace Pursuit (GPSP) algorithm \cite{he2023group}; see Algorithm \ref{SPalgo}.
}

\textbf{Step 2: Refine Candidate PDEs} \\
\For{$\theta = 1$ \KwTo $K$}{
    \textbf{a. Compute Normalized Coefficients:} \\
    \For{$v = 1$ \KwTo $\theta$}{Compute\\
        $\chi_v^\theta = \frac{n_v^\theta}{\max_{1 \leq v' \leq \theta} n_{v'}^\theta}$  with $\alpha_v^\theta = \|\mathbf{F}_v^\theta \mathbf{c}_v^\theta\|_2$,
    }
    
    \textbf{b. Group Feature Trimming:} \\
    Initialize an active feature group list for sparsity level $\theta$ with all feature groups included.\\
    \For{$v = 1$ \KwTo $\theta$}{
        \If{$\chi_v^\theta < \tau$}{
            Remove feature group $v$ from the active list.
        }
    }
    
    \textbf{c. Update Coefficients with Trimmed Features:} \\
    Recompute $\mathbf{c}^\theta$ by solving the regression problem using only the remaining active feature groups. If no feature groups are removed, retain $\mathbf{c}^\theta$ as obtained in Step 1.
}

\textbf{Step 3: Select Optimal PDE Using Reduction in Residual} \\
\For{$\theta = 1$ \KwTo $K - L$}{
    Compute the residual $q^\theta = \|\mathbf{F}\mathbf{c}^\theta - \mathbf{b}\|_2^2$. \\
    Compute the residual reduction ratio $s^\theta = \frac{q^\theta - q^{\theta+L}}{L q^1}$. \\
}
Choose the optimal sparsity level $\theta^* = \min\{\theta \mid 1 \leq \theta \leq K - L \text{ and } s^\theta < \rho^R\}$.

\KwOut{
    $\theta^*$: Optimal sparsity level.\\
    $\mathbf{c}^{\theta^*}$: Sparse coefficient vector corresponding to the optimal PDE after Group Feature Trimming.
}

\end{algorithm}

In this section, we present the \emph{boundary–condition–adapted} B-spline bases for both (i) approximating the spatially varying coefficients and (ii) constructing the test functions. By “adapted” we mean that the \textbf{spatial} B-splines are chosen \emph{periodic} on $\Omega$, while the \textbf{temporal} B-splines vanish at the endpoints $t=0$ and $t=T$ (zero Dirichlet), so that the boundary terms produced by integration by parts in time are eliminated. We then describe the design of tensor–product test functions $\varphi_r(x,t)=B^x_r(x)\,B^t_r(t)$ and derive an adaptive procedure (Section~\ref{subsec:testfnc}) for selecting their supports based on the spectral analysis of noisy data, which enhances noise suppression while preserving the essential dynamics.  The complete process is summarized in Algorithm \ref{alg.SPalgo1}.

\subsection{Boundary condition adapted B-spline basis} \label{subsec:bspline}

Let the degree of the B-splines be \( d \) and the order be \( p = d + 1 \). We consider an equidistant knot sequence \( z_j = z_0 + jh \) with knot spacing $h$ for \( j = 0, 1, \ldots, G \), where \( G \) denotes the total number of knots. For \( m = 0, 1, \ldots, G - p - 1 \), the B-spline basis is defined via the Cox-de Boor recursion formula \cite{deBoor1978} as follows:

\begin{equation}
B_{m}^{0}(x) =
\begin{cases}
  1 & \text{if } z_m \leq x < z_{m+1}, \\
  0 & \text{otherwise,}
\end{cases}
\end{equation}

\begin{equation}
B_{m}^{p}(x) = \left( \frac{x - z_m}{z_{m+p} - z_m} \right) B_m^{p-1}(x) + \left( \frac{z_{m+p+1} - x}{z_{m+p+1} - z_{m+1}} \right) B_{m+1}^{p-1}(x). \label{recursiveformula}
\end{equation}

The B-spline bases can be naturally extended to accommodate periodic boundary condition for the spatial domain, as illustrated in Figure~\ref{fig:testfnc2}(a). The corresponding formulas are detailed in \ref{Bspline_eqs}. For the time domain, we employ zero Dirichlet boundary condition for test functions to ensure that the boundary term from integration by parts vanishes  at the temporal boundaries, as shown in Figure~\ref{fig:testfnc2}(b). For higher-dimensional problems, we extend these bases by constructing tensor-product B-splines, as depicted in Figure~\ref{fig:testfnc2}(c). 

In this paper, we assume periodic boundary conditions for the PDE in space. In equations~\eqref{eq_varyingcoef_appr} and~\eqref{eq_testfnc_tensor}, the coefficient basis \(\{\psi_m(x)\}_{m=1}^M\) and spatial test functions \(B_{r_x}(x)\) are constructed using periodic B-splines, while the temporal test functions \(B_{r_t}(t)\) are constructed using B-splines with zero Dirichlet conditions.

\begin{figure}
    \centering
    \begin{tabular}{c c c}

        {\includegraphics[width=0.23\linewidth]{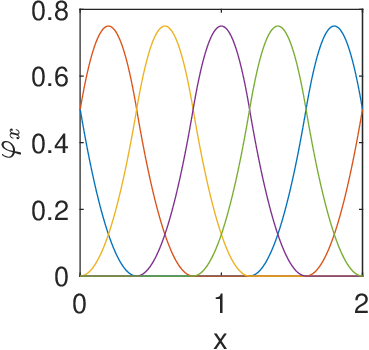}} & 
        {\includegraphics[width=0.23\linewidth]{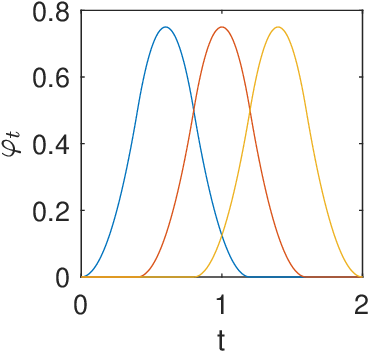}} &
        {\includegraphics[width=0.32\linewidth]{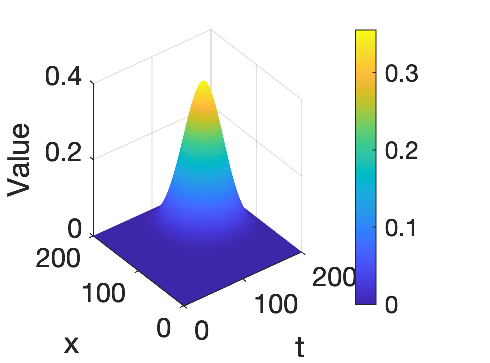}}\\
          (a) & (b) & (c) \\

    \end{tabular}
    \caption{(a) Periodic B-spline bases in the spatial domain. (b) B-spline bases in the time domain. (c) Tensor-product bivariate B-spline bases used for test functions in this work. }
    \label{fig:testfnc2}
\end{figure}

\begin{remark}

In the weak formulation with temporal derivatives, integration by parts uses boundary terms at $t=0$ and $t=T$. 
To eliminate them, the temporal test functions are chosen to vanish at the endpoints.
This choice does not aim to approximate the solution near the temporal boundaries, but rather ensures that the test functions serve as valid tool to extract features of the interior dynamics. 
Consequently, the weak formulation remains well-defined, and the constructed family of test functions continues to provide sufficient information for reliable PDE identification.

\end{remark}

\subsection{Design of the test function} \label{subsec:testfnc} 
As illustrated in Section~\ref{subsec:bspline}, different test functions are realized as translations of a single B-spline basis function. Therefore, the integral forms in~\eqref{eq_regression_model} can be associated with samples of convolutions of feature terms with the B-spline kernel. This was also exploited in~\cite{tang2022weakident,Messenger_2021}. From this perspective, the robustness of the weak formulation benefits from the low-pass filtering effect achieved through proper choice of test functions. 

In our proposed method, we use B-splines as test functions, whose spectral properties are determined by the local polynomial order and the placement of knot vectors \cite{Lenz_Bspline}. 
We discuss the design of test functions in space. Test functions in time can be designed in a similar way.

Following ideas in \cite{Messenger_2021,tang2022weakident}, we analyze the spectral property of noisy data to identify a critical frequency $k_x^*$ which divides the frequency domain into signal-dominate region $(-k_x^*,k_x^*)$ and noise-dominate region $(-\infty,-k_x^*]\cup[k_x^*,+\infty)$. By controlling the support of test functions, we ensure that the Fourier transform of test function is concentrated within the signal-dominant region, minimizing the impact of noise. Details on computing $k_x^*$ is deferred to \ref{Test_Fnc_cons}.

Let \( B^{p,h}(x) \) be an equidistant knot B-spline of order \( p \) with knot spacing \( h \) and centered at the origin, and $\widehat{B}^{p,h}(\omega)$ be its Fourier transform. Denote the support of $B^{p,h}$ by $(-\alpha_x,\alpha_x)$.
To concentrate the frequency of  B-spline within the signal-dominant region \((-k_x^*, k_x^*)\), we aim to select the spline order \(p\) and knot spacing \(h\) such that
\begin{align}\label{inequality_FT_bspline}
    1-\frac{\int_{-k_x^*}^{k_x^*}\left| \widehat{B}^{p,h}(\omega) \right|d\omega}{\int_{-\infty}^{\infty}\left| \widehat{B}^{p,h}(\omega) \right|d\omega} < \epsilon^* 
\end{align}
where \(\epsilon^*\) is a ratio measuring how $\widehat{B}^{p,h}(\omega)$ concentrates on the signal dominate region. Inequality~\eqref{inequality_FT_bspline} reveals a complicated, nonlinear relationship among the support size \(\alpha_x\), knot spacing \(h\) and spline order \(p\).

Motivated by theoretical connections between B-splines and Gaussian functions~\cite{christensen2017bsplineapproximationsgaussiangabor}, we address this challenge by approximating B-spline using a Gaussian function. 

This approximation is useful because the Fourier transform of a Gaussian function is another Gaussian function, which is easy to analyze:
\begin{lemma}[Fourier Transform of Gaussian Functions \cite{bracewell2000fourier}] \label{thm.Gaussian_Fourier}
Let $\rho(x)$ be a Gaussian function defined as $\rho(x) = \frac{1}{\sqrt{2\pi}\sigma}\exp\left(-\frac{x^2}{2\sigma^2}\right)$, where $\sigma$ is the standard deviation. Then the Fourier transform of $\rho(x)$, denoted $\widehat{\rho}(\omega)$ (with \(\omega\) as the frequency variable), is also Gaussian:

\[
\widehat{\rho}(\omega) = \exp\left(-\frac{\omega^2 \sigma^2}{2}\right).
\]

\end{lemma}

The following theorem gives a way to choose the standard deviation of Gaussian functions so that it has the same first four moments as those of $B^{p,h}(x)$ (see a proof in \ref{proof_prop}):
\begin{theorem} \label{prop:gaussian_bspline_moments}
Let $\rho(x)$ be a Gaussian function with standard deviation $\sigma$, where $\sigma$ is given by $\frac{\sqrt{p}}{2\sqrt{3}}h$. Then, the first four moments of the Gaussian function match those of a normalized equidistant knot B-spline of order $p$ and knot spacing $h$.
\end{theorem}

Figure \ref{fig:gaussian_bspline} illustrates the similarity between a B-spline basis function and a Gaussian distribution with parameters specified as in Theorem \ref{prop:gaussian_bspline_moments}. The B-spline basis function, denoted by $\varphi_x$, is constructed using equidistant knots. The Gaussian function, represented by $\rho_{\sigma}$, has a standard deviation $\sigma = \frac{\sqrt{p}}{2\sqrt{3}}h$, where $h$ is the knot spacing and $p$ is the B-spline order. As stated in the proposition, this specific choice of $\sigma$ ensures that the first four moments of the Gaussian function match those of the B-spline. The figure shows that the two functions are nearly indistinguishable.

\begin{figure}[t!] \centering \includegraphics[width=0.4\textwidth]{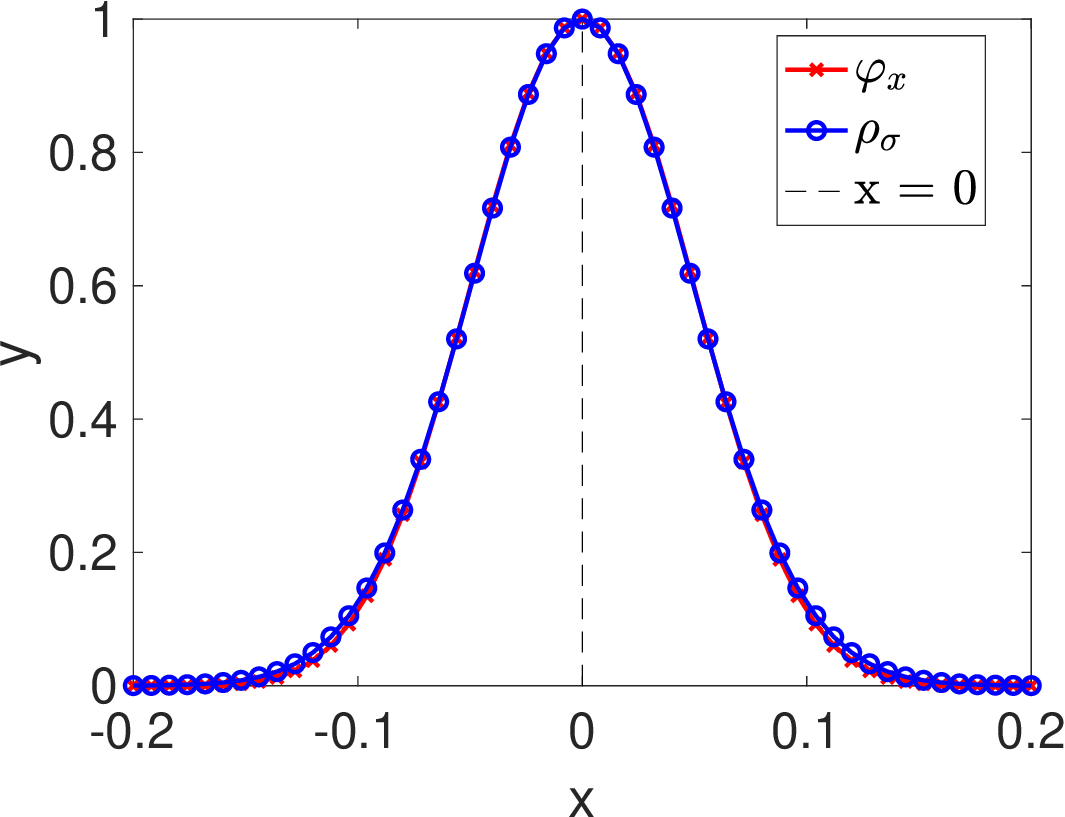} \caption{Comparison between a B-spline basis function $B^{p,h}(x)$ with $p=7,h=0.4$ and a Gaussian distribution $\rho_{\sigma}$ with $\sigma=\frac{\sqrt{p}}{2\sqrt{3}}h$. } \label{fig:gaussian_bspline} \end{figure}

The following theorem  states the accuracy of this approximation quantitatively.
\begin{theorem}\label{lemma:ft_difference_bound}

Let \( B^{p,h}(x) \) be an equidistant knot B-spline of order \( p \) with knot spacing \( h \), centered at the origin, and let
\[
\rho(x) = \frac{1}{\sqrt{2\pi}\sigma} e^{-\frac{x^2}{2\sigma^2}}, \quad \sigma = \frac{\sqrt{p}}{2\sqrt{3}} h
\]
be a Gaussian matching the first four moments of \( B^{p,h}(x) \). Define \(g^{p,h}(x)=\frac{1}{h} B^{p,h}(x)\), and denote the Fourier transform of $\rho$ and $g^{p,h}$ by \(\widehat{\rho}(\omega)\) and \(\widehat{g}^{p,h}(\omega)\), respectively. 
When $p<13$, we have
\begin{align}
    |\widehat{\rho}(\omega)-\widehat{g}^{p,h}(\omega)|\leq 
    \begin{cases}
        \frac{4}{5e^2p}\left(1+\frac{17\ln p}{7p}\right), &\mbox{for } |\omega|\leq\frac{2}{h}\sqrt{\frac{12\ln p}{p}},\\
        \exp\left(-\frac{p(h\omega)^2}{24}\right)+\left(\frac{2}{h\omega}\right)^p, &\mbox{for } |\omega|>\frac{\pi}{h}.
    \end{cases}
\end{align}

When $p\geq13$, we have
\begin{align}
    |\widehat{\rho}(\omega)-\widehat{g}^{p,h}(\omega)|\leq
        \begin{cases}
            \frac{4}{5e^2p}\left(1+\frac{17\ln p}{7p}\right), & \mbox{for }|\omega| < \frac{2}{h}\sqrt{\frac{12 \ln p}{p}},\\
            \frac{4(\ln p)^2}{5p^3}\left(1+\frac{17\ln p}{7p}\right), & \mbox{for }\frac{2}{h}\sqrt{\frac{12 \ln p}{p}} \leq |\omega| \leq \frac{\pi}{h},\\
            \exp\left(-\frac{p(h\omega)^2}{24}\right)+\left(\frac{2}{h\omega}\right)^p, &\mbox{for } |\omega|>\frac{\pi}{h}.
        \end{cases}
\end{align}
Furthermore,
\[
\lim_{p \to \infty} \left( \widehat{\rho}(\omega) -\widehat{g}^{p,h}(\omega) \right) = 0 \quad \text{for all } \omega \in \mathbb{R}~\text{and}~h>0.
\]
\end{theorem}
Theorem \ref{lemma:ft_difference_bound} is proved in \ref{proof_lemma1}. 
Theorem \ref{lemma:ft_difference_bound} shows that as the order of B-splines goes to infinity, its Fourier transform converges to that of Gaussian functions with properly chosen $\sigma$, demonstrating that  Gaussian function is a good approximator of B-splines. 
In our experiments, $p=7$ gives good numerical results.

The approximation results in Theorem \ref{prop:gaussian_bspline_moments} and \ref{lemma:ft_difference_bound} lead to a practical method for determining the support size \(\alpha_x\) of B-splines:
\begin{theorem}\label{thm.alphax}
    For a given $\tau_x>0$, consider a B-spline basis function of order $p$ and support $[-\alpha_x,\alpha_x]$. Let $\rho(x)$ be the Gaussian function whose first four moments are the same as those of the B-spline basis. Set 
    \begin{equation} 
    \alpha_x = \frac{\sqrt{3p} \tau_x}{k_x^*}.
\end{equation}
Then $k_x^*$ falls $\tau_x$ times of the standard deviations into the tail of the spectrum of $\rho(x)$ in the frequency domain.
\end{theorem}

Theorem \ref{thm.alphax} is proved in \ref{proof_thm}. Theorem \ref{thm.alphax} provides a systematic way to choose the B-spline support size $\alpha_x$. By positioning the critical frequency $k_x^*$ at $\tau_x$ standard deviations into the Gaussian tail, we ensure that 
$\epsilon^*$ in (\ref{inequality_FT_bspline}) is small, and $\widehat{B}^{p,h}$ concentrates on the signal-dominate region.

For practical implementation using the discrete Fourier transform (DFT), we adjust this relationship to account for sampling effects:
\begin{align}\label{eq_alpha}
    \alpha_x = \frac{\sqrt{3p} (N^x-1) \Delta x \tau_x}{2 \pi k_x^*}
\end{align}
where $\Delta x$ is the spatial step size, $N^x$ is the number of spatial grid points. This expression positions the critical frequency \(k_x^*\) at approximately \(\tau_x\) standard deviations into the tail of the B-spline basis function's Fourier spectrum. By doing so, the test function effectively emphasizes the signal-dominant frequency region \((-k_x^*, k_x^*)\) while minimizing its response in the noise-dominant region \((-\infty,-k_x^*]\cup[k_x^*, +\infty)\).

The relation~\eqref{eq_alpha} reveals the following properties of the proposed test function:
\begin{itemize}
    \item \textbf{Grid resolution effect:} The term $(N^x-1)\Delta x$ represents the total length of the spatial domain. A larger domain length increases $\alpha_x$, meaning that test functions have wider support.
    \item \textbf{Noise level impact:} A larger $k_x^*$ indicates that signal-dominant frequencies extend to higher values. Therefore, $\alpha_x$ decreases to ensure that test functions capture higher frequencies by having narrower spatial support.
    \item \textbf{Polynomial degree influence:} A higher B-spline order $p$ results in a smoother test function. The factor $\sqrt{p}$ in the numerator implies that increasing $p$ will increase $\alpha_x$, allowing for a wider support to maintain smoothness.
\end{itemize}

From~\eqref{eq_alpha}, we can determine the number of test functions required for the weak formulation. Given that we are using equidistant knots with knot spacing \( h = \frac{2\alpha_x}{p} \), for each knot, we will construct a test function by translating the previous one, therefore, the total number of test functions over the spatial domain $[L_1, L_2]$, denoted as $ J $, is:

\begin{equation} \label{total_number_testfn}
    J = \left\lceil \frac{L_2 - L_1}{\frac{2\alpha_x}{p}} \right\rceil = \left\lceil \frac{(L_2 - L_1) p}{2\alpha_x} \right\rceil,
\end{equation}
where the ceiling function \( \lceil \cdot \rceil \) ensures that \( J \) is an integer. This selection guarantees that the spatial domain \([L_1, L_2]\) is adequately covered by the test functions.
\begin{remark}\label{remark_test}
In existing works \cite{Messenger_2021, tang2022weakident}, a test function \(\varphi(x)\) is defined using truncated polynomials. Specifically, for \(i = 1, \ldots, N^x\), an order $p$ and a positive integer \(m_x\) depending on $p,k^*$, one defines:
\begin{equation} \label{otherstestfnc}
    \varphi(x) = 
    \begin{cases} 
        \left(1 - \left(\frac{x - x_i}{m_x \Delta x}\right)^2\right)^{p_x} 
        , & 
        |x - x_i| \leq m_x \Delta x, \\ 
        0 & \text{otherwise}.
    \end{cases}
\end{equation}

While these test functions are simple to construct, they do not satisfy the partition-of-unity property and may yield inconsistent weighting across the domain. By contrast, our B-spline test functions inherently form a partition of unity, ensuring consistent weighting and improved numerical stability. Our numerical experiments (Section~\ref{sec:robustness_of_bspline}) further demonstrate that using B-splines with adaptively chosen support sizes yields superior performance in terms of accuracy and robustness.

\end{remark}

\section{Numerical experiments} \label{sec:numerical_experiments}

\begin{table}[t!]
\centering
\footnotesize
\begin{tabular}{lll} 
\toprule
Equation  & Coefficients & Grid \\ 
\midrule

Advection diffusion 
& 
$\begin{cases}
    a(x) = 3(\sin(2\pi x)+3) \\
    c    = 0.2
\end{cases}$  
& $1001\times201$ \\

\refstepcounter{equation}\label{eq:advection-diffusion}
$\displaystyle u_t = a(x)u_x + cu_{xx}$ \hfill (\theequation)
& 
$
\begin{cases}
    u(x,0)=\sin(4\pi x)^2\cos(2\pi x)\\
    \hspace{1.5cm}+\sin(6\pi x)
\end{cases}
$

& 
$[0,0.05]\times[0,2)$ \\ 

\bottomrule

Viscous Burgers'
& 
$\begin{cases}
    a(x) = 0.8(\sin(2\pi x)+1) \\
    c   = 0.1
\end{cases}$  
&
$501 \times 201$ \\

\refstepcounter{equation}\label{eq:viscous-burgers}
$\displaystyle u_t = a(x)uu_x + cu_{xx}$ \hfill (\theequation)
& 
$
\begin{cases}
    u(x,0)=4(\sin(2\pi x)^22\cos(2\pi x)\\
    \hspace{1.5cm}+\sin(2 \pi x+0.2))
\end{cases}
$

& 
$[0,0.15] \times [0,2)$ \\  
\bottomrule

Korteweg-de Vires (KdV) 
& 
$\begin{cases}
    a(x) = 0.5(2+0.3\cos(\pi x/2)) \\
    b(x) = 0.01(0.5+0.1\sin(\pi x/2))
\end{cases}$  
&
$1212 \times 512$ \\

\refstepcounter{equation}\label{eq:kdv}
$\displaystyle u_t = a(x)uu_x + b(x)u_{xxx}$ \hfill (\theequation)
& 
$
\begin{cases}
    u(x,0)=2(6\sin(\pi x)\cos (\pi x)\\
    \hspace{1.5cm}+2\sin(\pi x)+2)
\end{cases}
$
& 
$[0,0.1] \times [-2,2)$ \\ 
\bottomrule 

Kuramoto-Sivashinsky (KS) 
& 
$\begin{cases}
    a(x) = 1 + 0.5\sin(\pi x/20) \\  
    b(x) = -1 + 0.2\sin(\pi x/20) \\  
    c(x) = -1 - 0.2\sin(\pi x/20)
\end{cases}$  
& 
$700 \times 500$ \\

\refstepcounter{equation}\label{eq:ks}
$\begin{aligned}
    \displaystyle u_t = &a(x)uu_x + b(x)u_{xx} \\
    &+ c(x)u_{xxxx}
\end{aligned}$ \hfill (\theequation)
& 
$
\begin{aligned}
    u(x,0)=\cos(\pi x /20)(1+\sin(\pi x/20))
\end{aligned}
$

& 
$[0,60]\times[0,80)$ \\ 
\bottomrule 

Schr\"{o}dinger 
& 
$\begin{cases}
    a(x) = -5\cos(\pi x/3) \\  
    c = 0.5 
\end{cases}$   
& 
$2000 \times 100$\\

\refstepcounter{equation}\label{eq:schrodinger}
$\displaystyle iu_t = cu_{xx} + a(x)u$ \hfill (\theequation)
& 
$\begin{cases}
    u_0(x,0) = \sin(2\pi x + 0.2) \\  
    \hspace{1.5cm}+0.5\cos((x-3)\pi/3)\\
    \hspace{1.5cm}\cos(2\pi x) \\
    u_1(x,0) = 0.5\cos(2\pi x/3)\\
    \hspace{1.5cm}+0.5\sin(2\pi x +0.2)  
\end{cases}$  
& $[0,2]\times[-3,3)$ \\ 
\bottomrule 

Nonlinear Schr\"{o}dinger (NLS) 
& 
$\begin{cases}
    a(x) = -5\cos(\pi x/3) \\  
    c = 0.5 
\end{cases}$   
& 
$2000 \times 200$ \\

\refstepcounter{equation}\label{eq:nls}
$\displaystyle iu_t = cu_{xx} + a(x)|u|^2u$ \hfill(\theequation)
&
$\begin{cases}
    v_0(x,0) = \sin(4\pi x + 0.2)\\
    \hspace{1.5cm}+\cos((x-3)\pi /3)\cos(2\pi x )\\
    w_0(x,0) = 0.5\cos(2\pi x/6)
     +\sin(2\pi x)  
\end{cases}$  
& $[0,0.5]\times[-3,3)$ \\ 
\bottomrule
\end{tabular}
\caption{Summary of the differential equations tested in this paper. For the complex function $u(x,t)$ in Schr\"{o}dinger and NLS equations, we denote $u(x,t)=v(x,t)+\sqrt{-1}w(x,t)$ with $v(x,t)$ being the real part and $w(x,t)$ being the imaginary part. We use $v_0,w_0$ denote the real and imaginary parts of initial conditions, respectively. }
\label{tab:differential_equations}
\end{table}

We demonstrate the effectiveness of WG-IDENT through comprehensive experiments\footnote{Our code is available at GitHub repository
\url{https://github.com/ChengTang1205/WG_IDENT_demo}.
}, with additional numerical results provided in \ref{appendix_comparisonmtd}. First, we outline the settings used in these experiments.

\textbf{Tested PDE models:}~We evaluate our method on the PDEs listed in Table \ref{tab:differential_equations}. For the complex function $u(x,t)$ in  Schr\"{o}dinger and NLS equations, we denote $u(x,t)=v(x,t)+\sqrt{-1}w(x,t)$ with $v(x,t)$ being the real part and $w(x,t)$ being the imaginary part.

\textbf{Feature dictionary:}~Our dictionary comprises polynomials up to degree 3 and their derivatives up to the fourth order. In the first four numerical examples — the advection-diffusion equation, viscous Burgers' equation, KdV equation, and KS equation — we use a feature matrix consisting of 16 distinct feature groups, including the terms: 1, $u$, $u^2$, $u^3$, and their first four derivatives. For the last two examples, involving the Schr\"{o}dinger equation (Sch) and the Nonlinear Schr\"{o}dinger equation (NLS), which contain both real and imaginary components, we extend our dictionary by adding 30 additional terms involving imaginary components. This expansion increase the total number of terms to 46, including $w$, $w^2$, $w^3$, $vw$, $v^2w$ and $vw^2$ along with their first four spatial derivatives.

\textbf{Noise model:}~All of our experiments are conducted with both clean and noisy data. 

Let \[U_{\text{max}} = \max_{j=1,\ldots,N^x} \max_{m=1,\ldots,N^t} U_j^m\] and 
\[U_{\text{min}} = \min_{j=1,\ldots,N^x} \min_{m=1,\ldots,N^t} U_j^m.\] We consider a noise model $\epsilon\sim\mathcal{N}(0,\sigma^2)$ with  standard deviation \(\sigma\)  defined as:
\[
\sigma = \frac{\sigma_{\text{NSR}}}{N^t N^x} \sum_{i=1}^{N^x} \sum_{n=1}^{N^t} \left| U_i^n - \frac{U_{\text{max}} + U_{\text{min}}}{2} \right|^2.
\]
Here, $\sigma_{\text{NSR}}\in [0,1]$ indicates the noise-to-signal ratio (NSR).

\textbf{Hyper-parameters:} We utilize seven B-spline basis functions (\(M = 7\)) for approximating varying coefficients, as defined in equation~\eqref{eq_regression_model}. To construct test functions, we manually select \(\tau_x = 3.5\) for all PDEs, except for the Nonlinear Schrödinger (NLS) equation, for which we set \(\tau_x = 4.0\). The temporal parameter \(\tau_t = 0.6\) is applied uniformly across all PDEs. During the GF-trim process, we employ a threshold \(\tau = 0.1\) to eliminate the low-contribution features of those candidate PDEs. Subsequently, for the Reduction in Residual (RR) process, we set \(L = 3\) and \(\rho^{\text{R}} = 0.01\) for all PDEs, except for the Nonlinear Schrödinger (NLS) equation and the Schrödinger equation, where we use \(\rho^{\text{R}} = 0.1\). To ensure sufficient smoothness of the test functions, the degree of B-splines used for both varying coefficients and test functions is set to \(d = 6\). 
Importantly, we further perform systematic robustness tests with respect to the spectral support parameter \(\tau_x\) and the spline degree \(p\). It  confirms that the identification results remain stable across a broad range of choices; these experiments are reported in \ref{app:tau_sensitivity}

\textbf{Evaluation Metrics:} We use the following four criteria to quantify the effectiveness of our method.

\begin{itemize}
    \item \textbf{Relative Coefficient Error}:
    \begin{equation} \label{relative_e2}
        E_2 = \frac{||\textbf{c}^* - \textbf{c}||_2}{||\textbf{c}^*||_2}.
    \end{equation}
    This metric measures the accuracy of the recovered coefficients \(\textbf{c}\) against the true coefficients \(\textbf{c}^*\) using the relative \(l_2\) norm. Specifically, we treat both vectors \(\textbf{c}\) and \(\textbf{c}^*\) as having length $KM$. The recovered coefficient vector \(\textbf{c}\) is initialized as an all-zero vector of size $KM$, with nonzero entries corresponding only to the active (identified) features.

    \item \textbf{Residual Error}:
    \begin{equation} \label{Residual_error}
        E_{\text{res}} = \frac{||\textbf{F}\textbf{c} - \textbf{b}||_2}{||\textbf{b}||_2}.
    \end{equation}
    The residual error quantifies how well the learned differential equation \(\textbf{F}\textbf{c}\) fits the observed data \(\textbf{b}\). Specifically, it measures the relative difference between the predicted values \(\textbf{F}\textbf{c}\) and the actual data \(\textbf{b}\) using the \(l_2\) norm. A lower residual error indicates a better fit of the model to the data.

    \item \textbf{True Positive Rate (TPR)}:
    \begin{equation} \label{tpr}
        \text{TPR} = \frac{|\{ l : \textbf{c}^*(l) \neq 0 \ \text{and} \ \textbf{c}(l) \neq 0 \}|}{|\{ l : \textbf{c}^*(l) \neq 0 \}|}.
    \end{equation}
    This metric computes the ratio of correctly identified non-zero coefficients to the total number of true non-zero coefficients. It quantifies the method's ability to correctly identify relevant features in the underlying differential equation.

    \item \textbf{Positive Predictive Value (PPV)}:
    \begin{equation} \label{ppv}
        \text{PPV} = \frac{|\{ l : \textbf{c}^*(l) \neq 0 \ \text{and} \ \textbf{c}(l) \neq 0 \}|}{|\{ l : \textbf{c}(l) \neq 0 \}|}.
    \end{equation}
    
    This metric measures the proportion of correctly identified non-zero coefficients out of all coefficients that the method identifies as non-zero. It indicates the rate of false positives, where a high PPV signifies that most identified features are indeed relevant.
\end{itemize}

\subsection{General performances}
\begin{figure}[t!]
    \centering
    \begin{tabular}{cccc}
    \hline
    \multicolumn{4}{c}{Advection diffusion equation}\\\hline
            $\sigma_{\text{NSR}}=0\%$ & $3\%$ & $5\%$ & $10\%$ \\
    \hline
    \multicolumn{4}{c}{Given data (noisy observations)}\\
    {\includegraphics[width=0.2\linewidth]{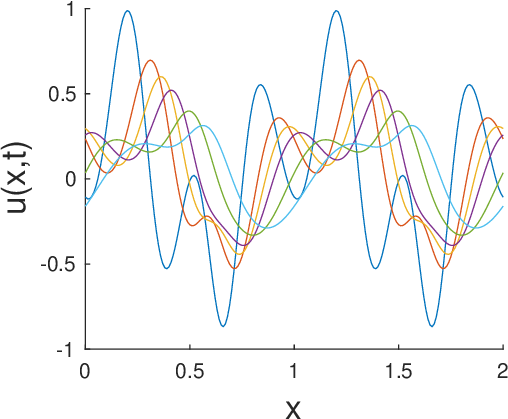}} &
        {\includegraphics[width=0.2\linewidth]{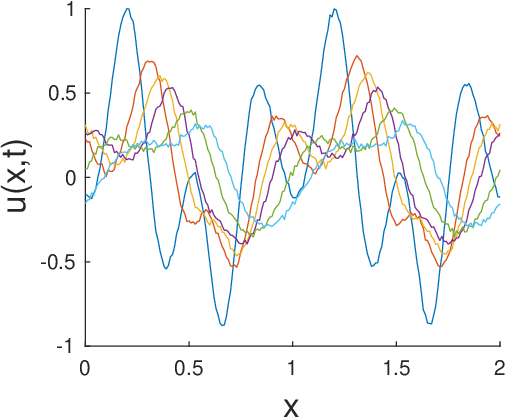}} &
        {\includegraphics[width=0.2\linewidth]{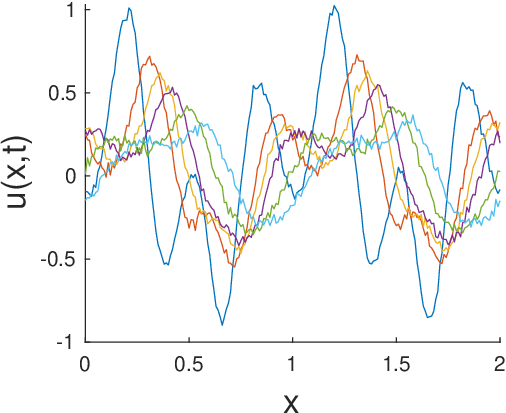}}&
        {\includegraphics[width=0.2\linewidth]{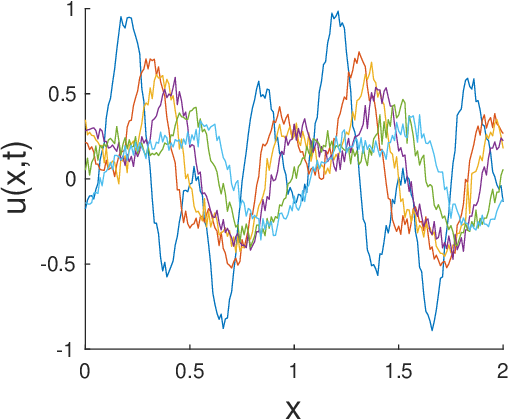}}\\
    \hline 
    \multicolumn{4}{c}{Simulated solution from the identified PDE}\\
        {\includegraphics[width=0.2\linewidth]
        {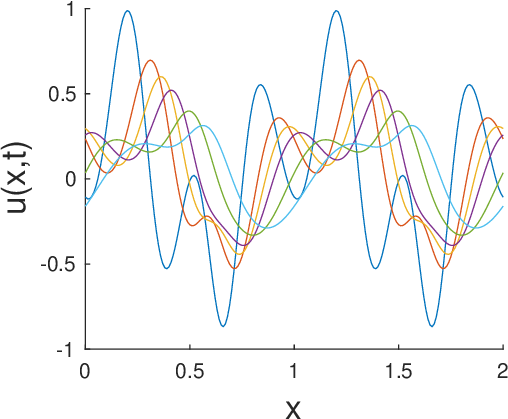}} &
        {\includegraphics[width=0.2\linewidth]{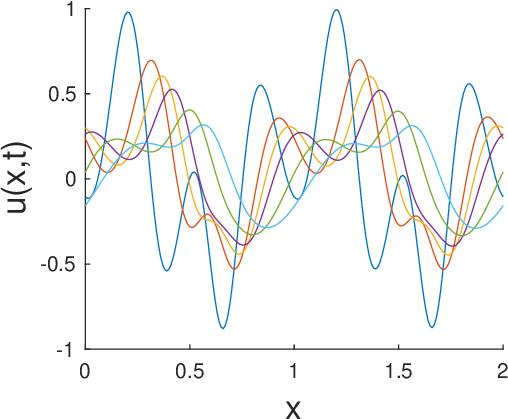}} &
        {\includegraphics[width=0.2\linewidth]{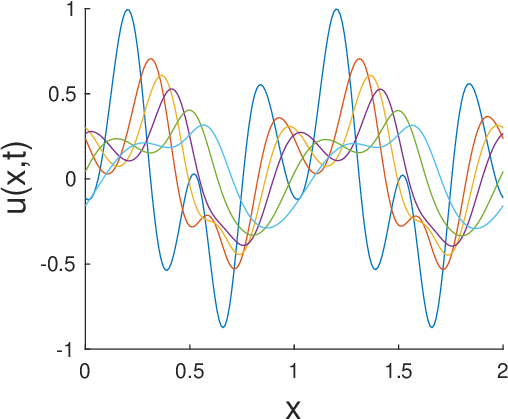}}&
        {\includegraphics[width=0.2\linewidth]{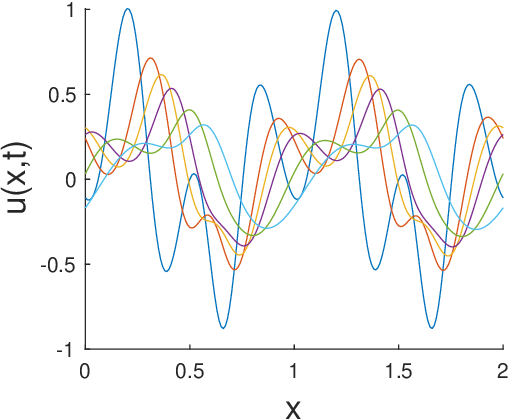}}\\
    \hline 
    \multicolumn{4}{c}{Spatial profile at $t=0.03$ (blue = true solution, red = identified solution)}\\
        {\includegraphics[width=0.2\linewidth]{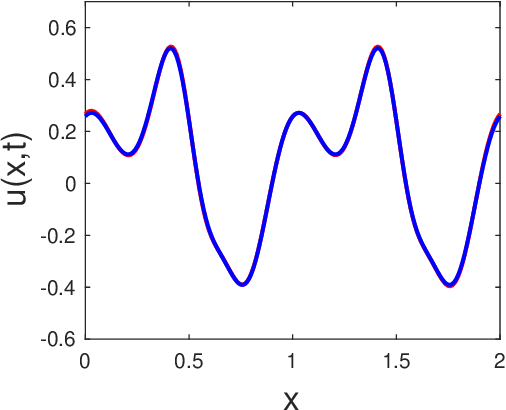}} &
        {\includegraphics[width=0.2\linewidth]{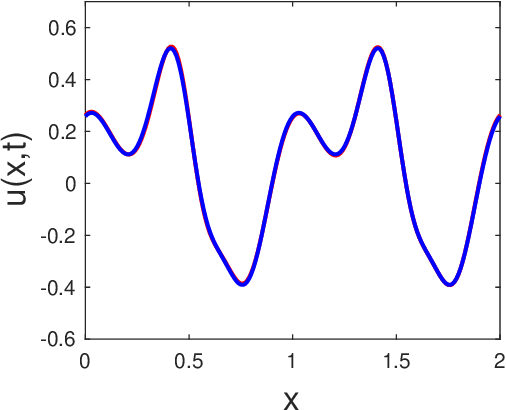}} &
        {\includegraphics[width=0.2\linewidth]{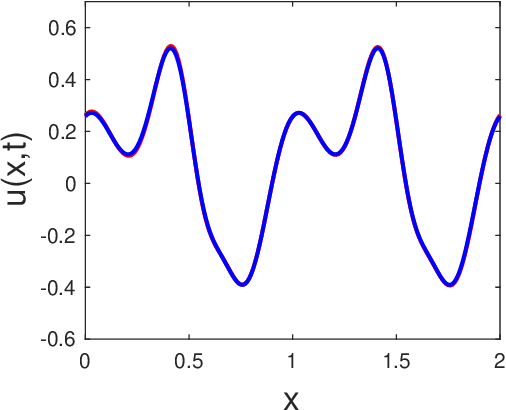}}&
        {\includegraphics[width=0.2\linewidth]{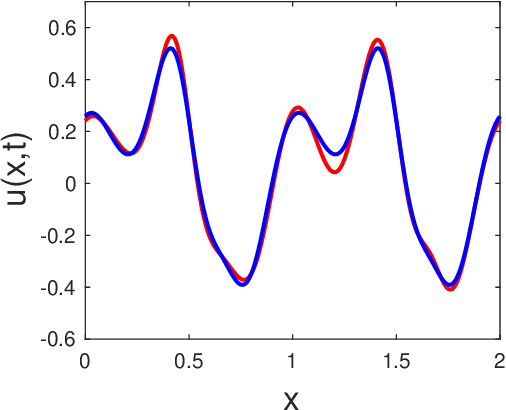}}\\
    \hline 
    \multicolumn{4}{c}{Absolute error}\\
        {\includegraphics[width=0.2\linewidth]{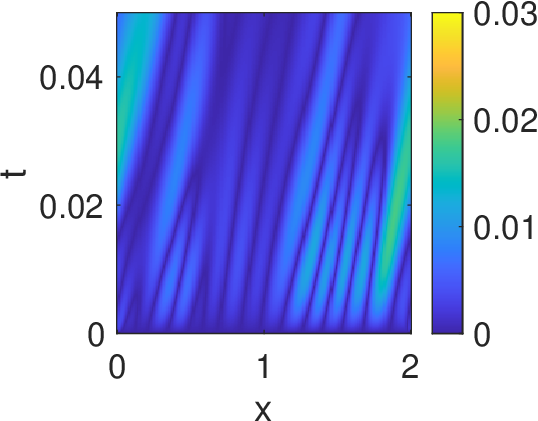}} &
        {\includegraphics[width=0.2\linewidth]{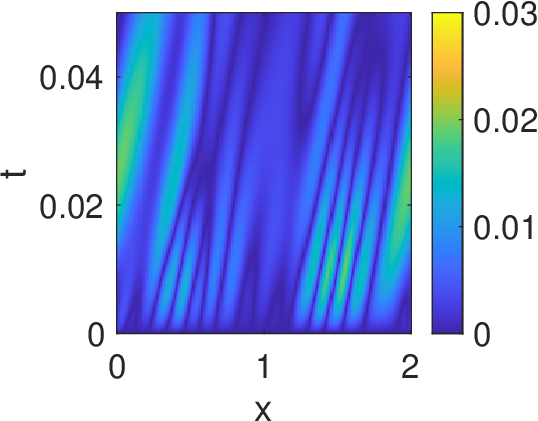}} &
        {\includegraphics[width=0.2\linewidth]{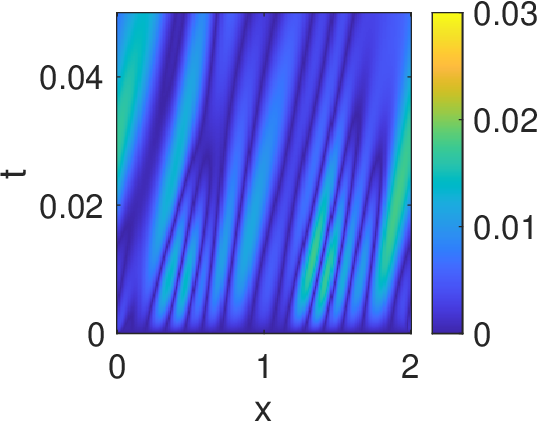}}&
        {\includegraphics[width=0.2\linewidth]{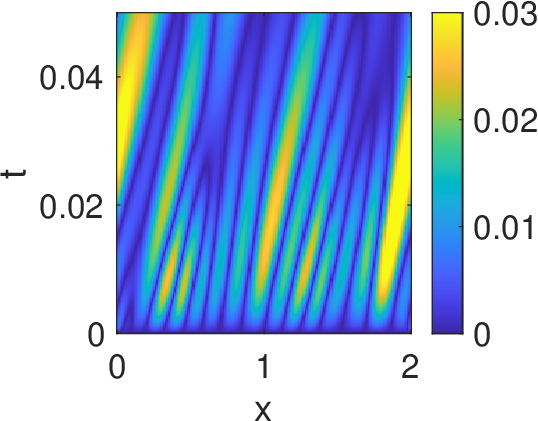}}\\
    \hline
    \end{tabular}
    \caption{{\textbf{Advection--diffusion equation under increasing observation noise.}
Columns correspond to $\sigma_{\text{NSR}}=0\%, 3\%, 5\%, 10\%$.
Row~1 shows the given (noisy) observations as spatial trajectories at time frames $t\in\{0.01,0.02,0.025,0.03,0.04,0.05\}$; 
Row~2 shows trajectories at the same time frames simulated from the PDE identified by WG-IDENT; 
Row~3 shows a 1D spatial profile at $t=0.03$ overlaying the true solution (blue) and identified curves (red); 
Row~4 presents the point-wise absolute errors between the simulated and true solutions across the computational domain $x\in[0,2]$ and $t\in[0,0.05]$.}}
    \label{fig:2d3d_transport}
\end{figure}

\begin{figure}[t!]
    \centering
    \begin{tabular}{cccc}
    \hline
    \multicolumn{4}{c}{Viscous Burgers' equation}\\\hline
        $\sigma_{\text{NSR}}=0\%$ & 3\% & 5\% & 10\% \\
    \hline
    \multicolumn{4}{c}{Given data}\\
    {\includegraphics[width=0.2\linewidth]{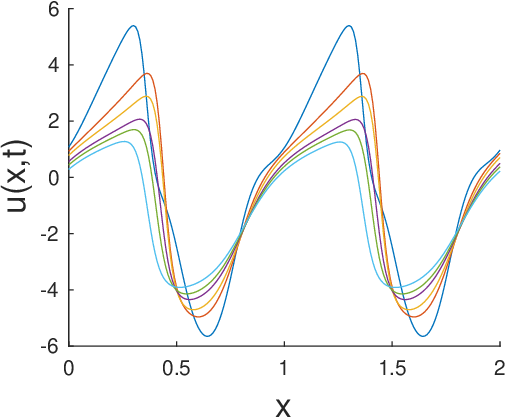}} &
        {\includegraphics[width=0.2\linewidth]{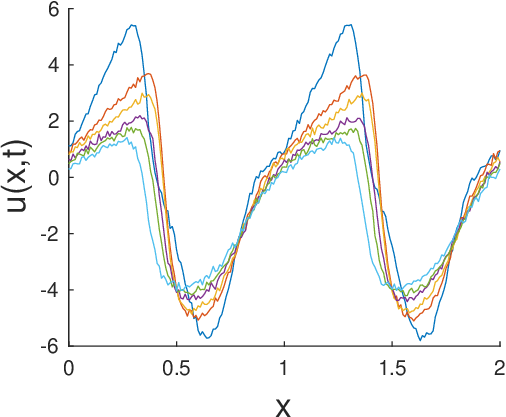}} &
        {\includegraphics[width=0.2\linewidth]{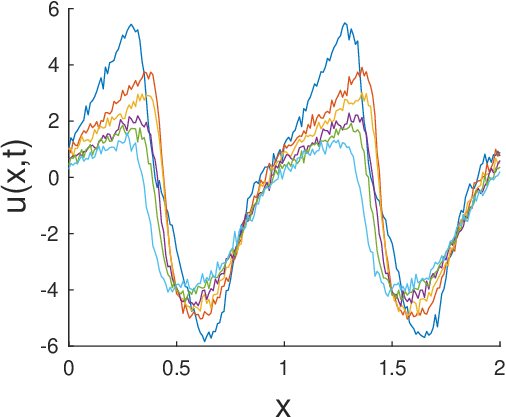}}&
        {\includegraphics[width=0.2\linewidth]{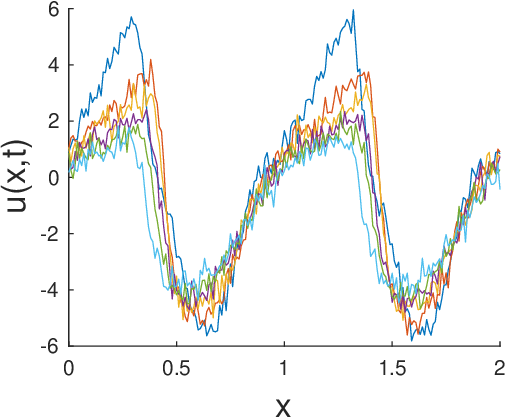}}\\
        \hline
       \multicolumn{4}{c}{Simulated solution from the identified PDE}\\
        {\includegraphics[width=0.2\linewidth]{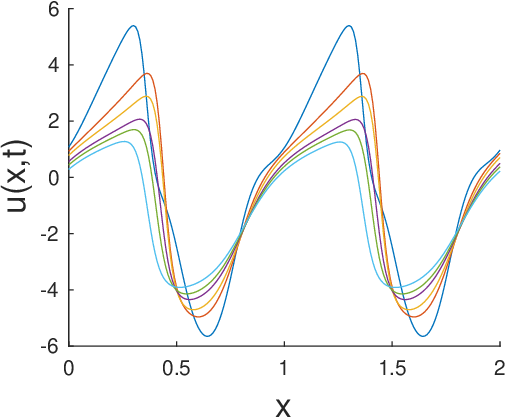}} &
        {\includegraphics[width=0.2\linewidth]{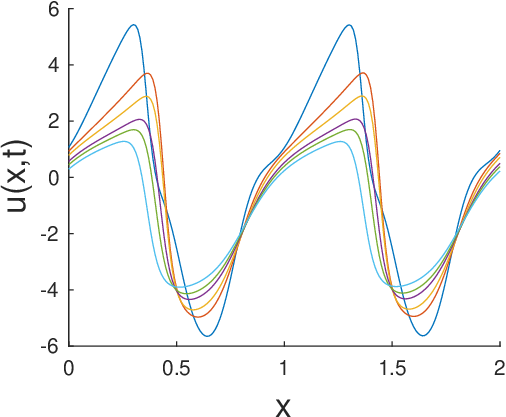}} &
        {\includegraphics[width=0.2\linewidth]{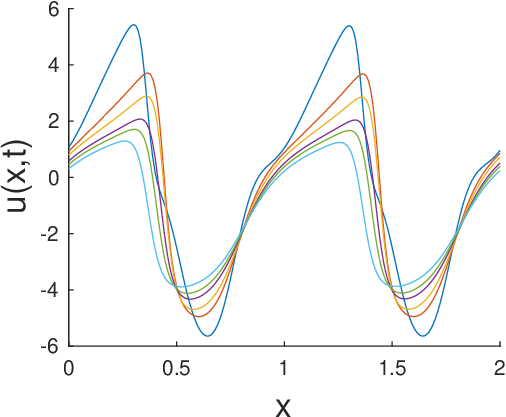}}&
        {\includegraphics[width=0.2\linewidth]{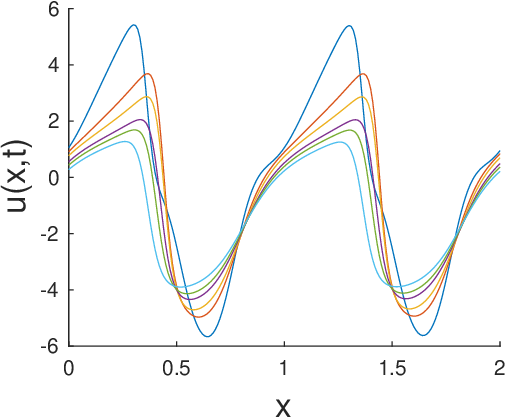}}\\
    \hline
    \multicolumn{4}{c}{Spatial profile at $t=0.10$ (blue = true solution, red = identified solution)}\\
    {\includegraphics[width=0.2\linewidth]{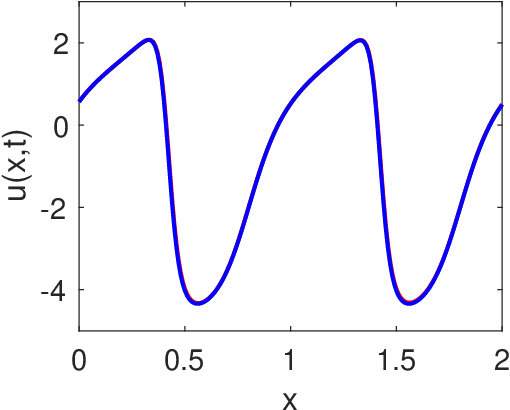}} &
    {\includegraphics[width=0.2\linewidth]{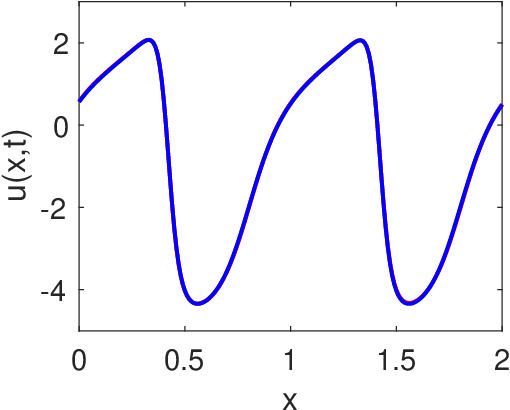}} &
    {\includegraphics[width=0.2\linewidth]{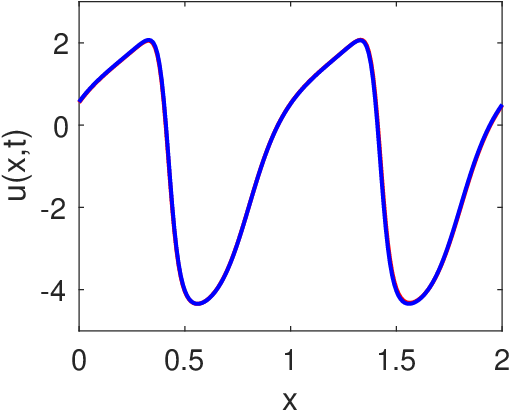}}&
    {\includegraphics[width=0.2\linewidth]{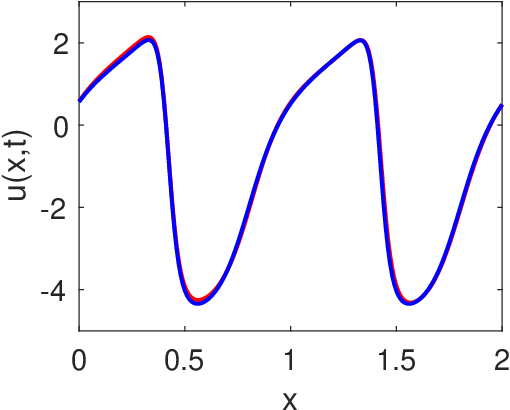}}\\
    \hline
    \multicolumn{4}{c}{Absolute error}\\
        {\includegraphics[width=0.2\linewidth]{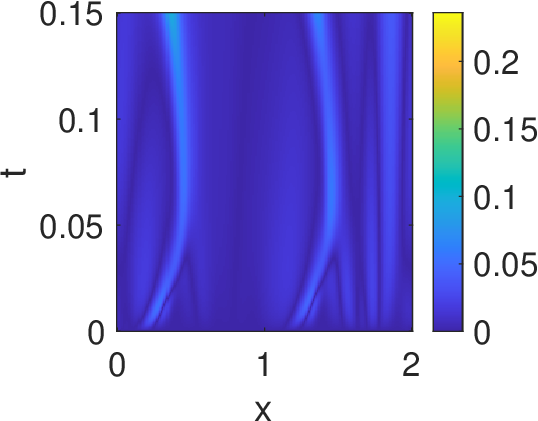}} &
        {\includegraphics[width=0.2\linewidth]{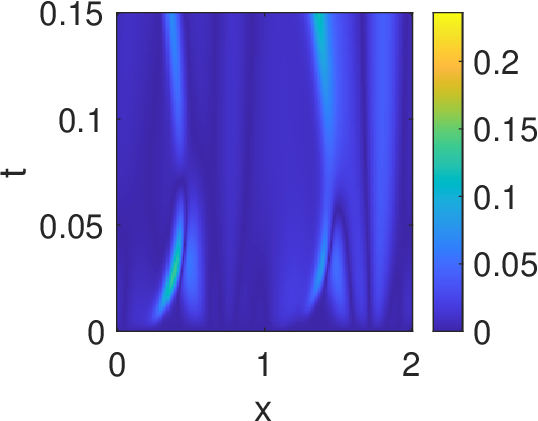}} &
        {\includegraphics[width=0.2\linewidth]{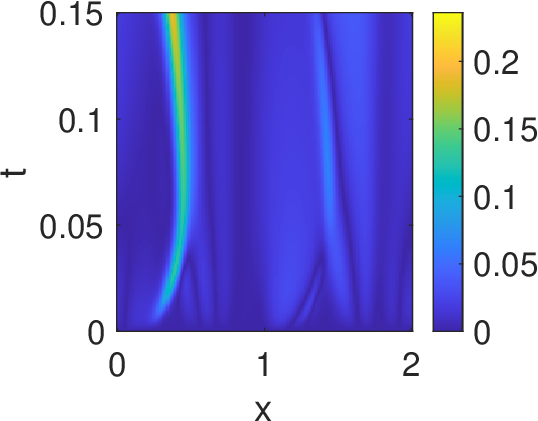}}&
        {\includegraphics[width=0.2\linewidth]{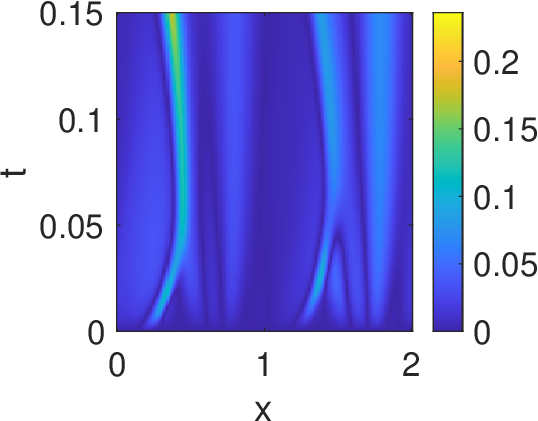}}\\
    \hline
    \end{tabular}
    \caption{\textbf{Viscous Burgers' equation under increasing observation noise.}
Columns correspond to $\sigma_{\text{NSR}}=0\%, 3\%, 5\%, 10\%$.
Row~1 shows the given (noisy) observations as spatial trajectories at time frames $t\in\{0.02, 0.05, 0.07, 0.1, 0.12, 0.15\}$; 
Row~2 shows trajectories at the same time frames simulated from the PDE identified by WG-IDENT; 
Row~3 shows a 1D spatial profile at $t=0.10$ overlaying the true (blue) and identified (red) trajectories; 
Row~4 shows the point-wise absolute errors between the simulated and true solutions across the computational domain $x\in[0,2]$ and $t\in[0,0.15]$.
    }
    \label{fig:2d3d_burgerdiff}
\end{figure}

 We demonstrate the effectiveness of WG-IDENT by identifying two PDEs with varying coefficients: the advection-diffusion equation~\eqref{eq:advection-diffusion} and the viscous Burgers' equation~\eqref{eq:viscous-burgers}.

\textbf{Advection-Diffusion Equation:} 
In the first row of Figure~\ref{fig:2d3d_transport}, we present the observed trajectories of the advection-diffusion equation, which have been contaminated with varying levels of noise. The figure shows results for noise-to-signal ratios (NSR) of 0\%, 3\%, 5\%, and 10\%. Despite the presence of noise, WG-IDENT is able to successfully identify the two true features from $16$ candidate features, even when the noise-to-signal ratio reaches 10\%. The second row of the figure displays the simulated trajectories generated from the identified PDEs under each noise condition. These simulated trajectories closely align with the true trajectories of the underlying PDE, confirming that WG-IDENT accurately captures the dynamics of the system. The third row shows the spatial profile at $t=0.03$, overlaying the true solution (blue) and identified curves (red); we can see that they closely overlap with each other under different noise levels. The fourth row shows the absolute errors between the simulated and true solutions, where the errors are small over the computational domain, indicating that WG-IDENT remains robust and precise even in the presence of substantial noise.

\textbf{Viscous Burgers' Equation:} 
In Figure~\ref{fig:2d3d_burgerdiff}, we demonstrate the performance of WG-IDENT on the viscous Burgers' equation under various noise levels. Again, using 16 candidate features, WG-IDENT accurately identifies the underlying PDE even when the noise-to-signal ratio is as high as 10\%. The first row of the figure presents the observed data trajectories, which are contaminated with different levels of noise. The second row illustrates the simulated trajectories derived from the identified PDE models corresponding to each noise condition. These simulated solutions are in good agreement with the true solutions of the system, confirming that the PDE identification method is capable of recovering the underlying dynamics despite the noise. The third row shows the spatial profile at $t=0.10$, overlaying the true solution (blue) and identified curves (red); we can see that they closely overlap with each other under different noise levels. The fourth row displays the absolute errors between the simulated solutions and the true solutions. These error maps indicate that WG-IDENT can produce solutions with small deviations from the true model, demonstrating its robustness and precision in noisy settings. 

In both case studies, advection-diffusion and viscous Burgers' equations, WG-IDENT demonstrates robust accuracy and stability, effectively recovering the correct PDE forms even under significant noise conditions. Beyond these two case studies, we also provide additional experiments on the inviscid Burgers' equation with a discontinuous solution and a more complicated form of varying coefficient (see~\ref{Appendixd4_inviscid}). These results further confirm that WG-IDENT can reliably identify the governing terms under diverse and challenging conditions. In the following section, we quantitatively investigate the performance of WG-IDENT.

\begin{figure}[t!]
    \centering
    \begin{tabular}{cccc}
    \hline
        \textbf{E\textsubscript{2}} & \textbf{E\textsubscript{res}} & \textbf{TPR} & \textbf{PPV} \\
    \hline
    \multicolumn{4}{c}{Viscous Burgers' equation}\\\hline
        \includegraphics[width=0.2\linewidth]{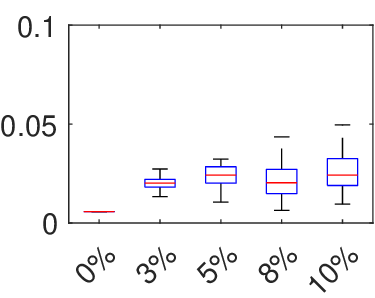} &
        \includegraphics[width=0.2\linewidth]{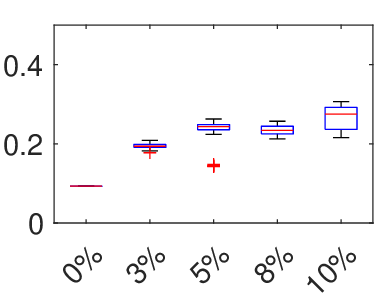}&
        \includegraphics[width=0.2\linewidth]{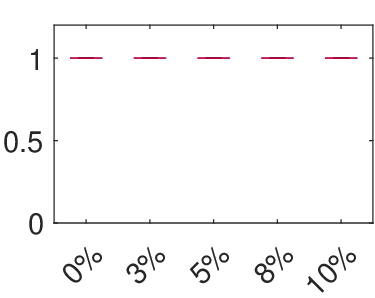}&
        \includegraphics[width=0.2\linewidth]{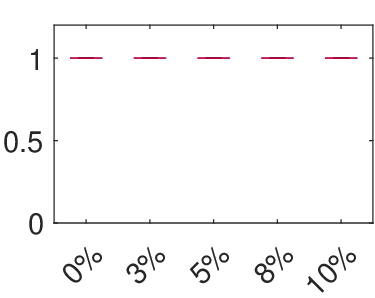} \\

    \hline 
            \multicolumn{4}{c}{Advection diffusion equation}\\\hline%
        \includegraphics[width=0.2\linewidth]{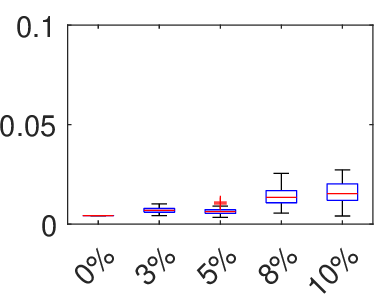} &
        \includegraphics[width=0.2\linewidth]{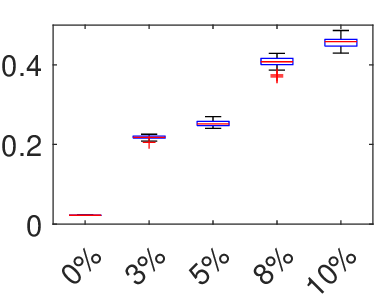}&
        \includegraphics[width=0.2\linewidth]{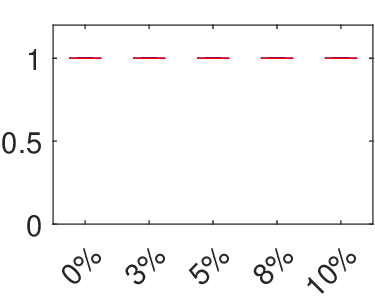}&
        \includegraphics[width=0.2\linewidth]{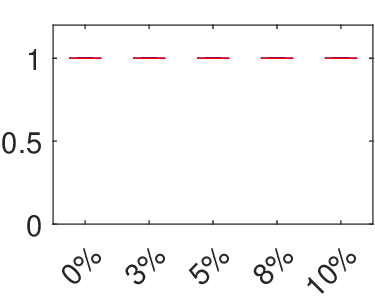}  \\

    \hline
    \multicolumn{4}{c}{KdV equation}\\\hline
        \includegraphics[width=0.2\linewidth]{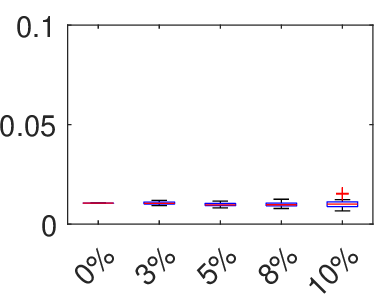} &
        \includegraphics[width=0.2\linewidth]{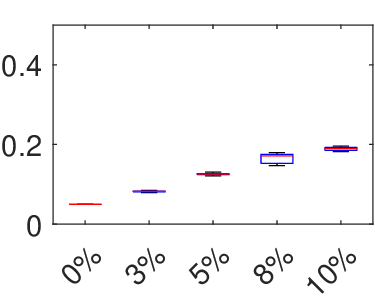}&
        \includegraphics[width=0.2\linewidth]{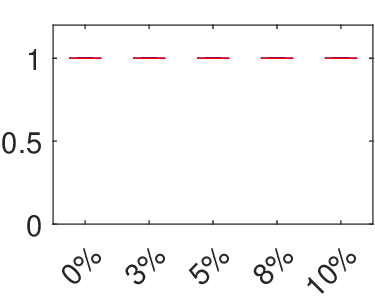}&
        \includegraphics[width=0.2\linewidth]{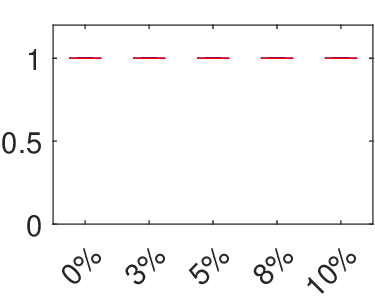}  \\

    \hline
    \multicolumn{4}{c}{KS equation}\\\hline
        \includegraphics[width=0.2\linewidth]{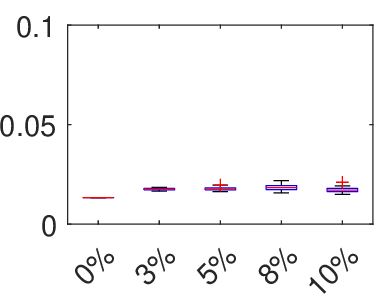} &
        \includegraphics[width=0.2\linewidth]{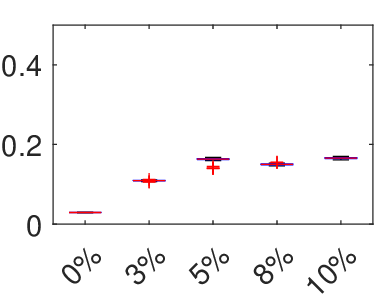}&
        \includegraphics[width=0.2\linewidth]{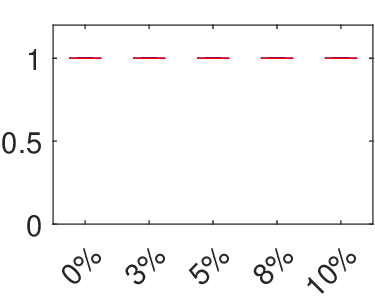}&
        \includegraphics[width=0.2\linewidth]{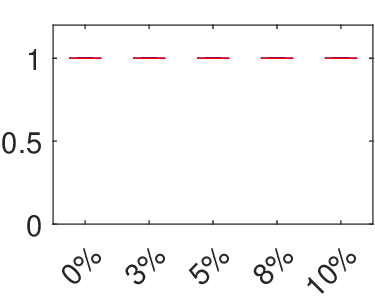}  \\

    \hline
    \end{tabular}
    \caption{$E_2$, $E_{\text{res}}$, TPR, PPV of viscous Burgers' equation (first row), advection diffusion equation (second row), KdV equation (third row), and KS equation (fourth row) with different noise levels, $0\%$, $3\%$, $5\%$, $8\%$, $10\%$. Each experiment is repeated for 50 times.}

    \label{fig:E2_Eres_TPR_PPV_burger_trans_KdV_KS}
\end{figure}

\begin{figure}[t!]
    \centering
    \begin{tabular}{cccc}
    \hline
        \textbf{E\textsubscript{2}} & \textbf{E\textsubscript{res}} & \textbf{TPR} & \textbf{PPV} \\
    \hline
    \multicolumn{4}{c}{Schr\"{o}dinger equation real part}\\\hline
        \includegraphics[width=0.2\linewidth]{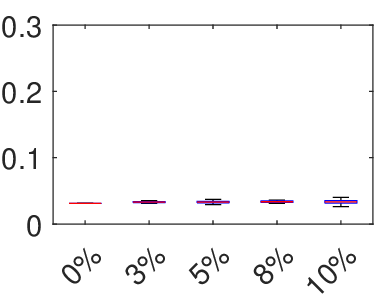} &
        \includegraphics[width=0.2\linewidth]{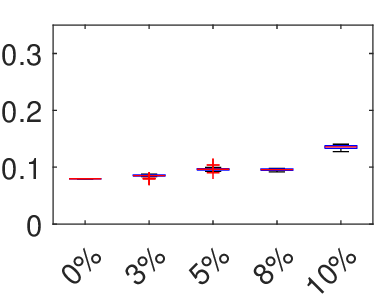}&
        \includegraphics[width=0.2\linewidth]{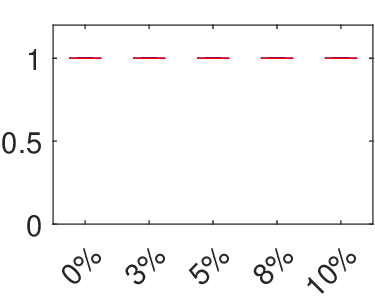}&
        \includegraphics[width=0.2\linewidth]{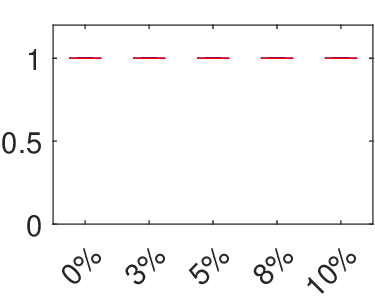} \\

    \hline 
       \multicolumn{4}{c}{Schr\"{o}dinger equation imaginary part}\\\hline
        \includegraphics[width=0.2\linewidth]{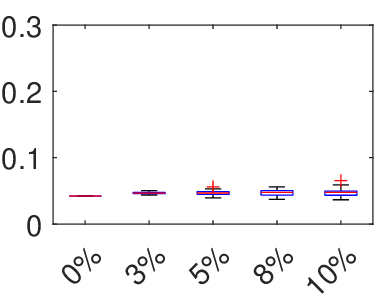} &
        \includegraphics[width=0.2\linewidth]{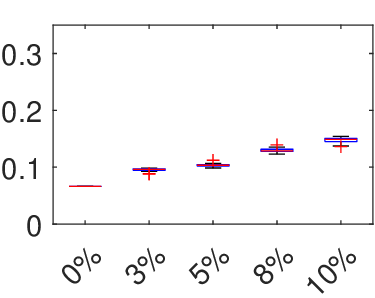}&
        \includegraphics[width=0.2\linewidth]{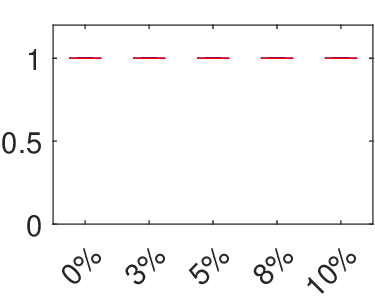}&
        \includegraphics[width=0.2\linewidth]{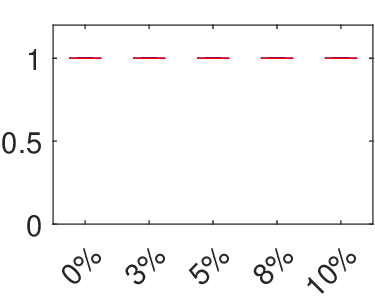}  \\

    \hline
        \multicolumn{4}{c}{Nonlinear Schr\"{o}dinger equation real part}\\\hline
        \includegraphics[width=0.2\linewidth]{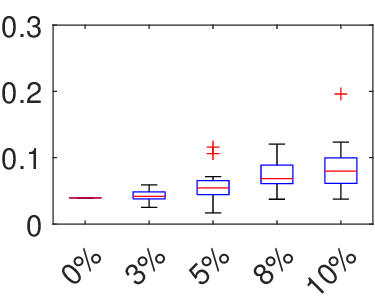} &
        \includegraphics[width=0.2\linewidth]{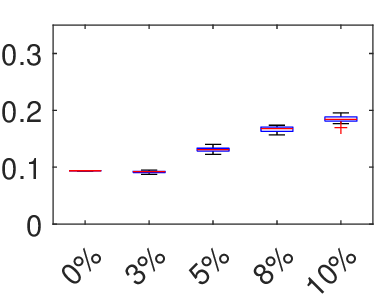}&
        \includegraphics[width=0.2\linewidth]{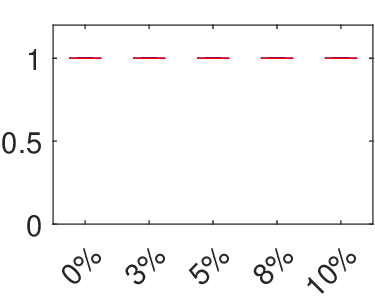}&
        \includegraphics[width=0.2\linewidth]{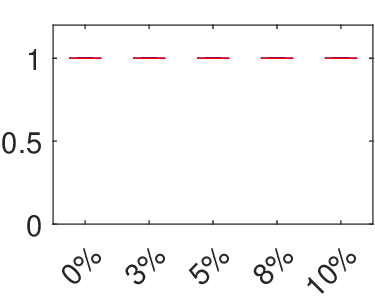}  \\

    \hline
        \multicolumn{4}{c}{Nonlinear Schr\"{o}dinger equation imaginary part}\\\hline
        \includegraphics[width=0.2\linewidth]{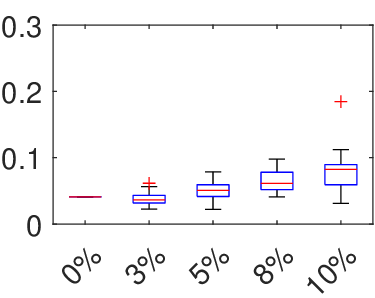} &
        \includegraphics[width=0.2\linewidth]{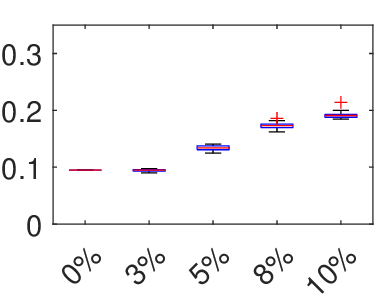}&
        \includegraphics[width=0.2\linewidth]{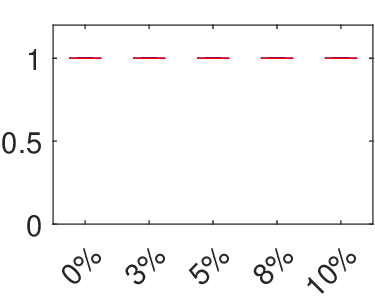}&
        \includegraphics[width=0.2\linewidth]{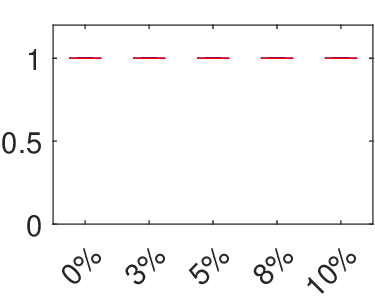}  \\

    \hline
    \end{tabular}
    \caption{$E_2$, $E_{\text{res}}$, TPR, PPV of Schr\"{o}dinger equation real part (first row), Schr\"{o}dinger equation imaginary part (second row), Nonlinear Schr\"{o}dinger equation real part (third row), and  Nonlinear Schr\"{o}dinger equation imaginary part (forth row) among different noise levels, 0\%, 3\%, 5\%, 8\%, 10\%. Each experiment is repeated 50 times.}
    \label{fig:example2}
    \label{fig:E2_Eres_TPR_PPV_Sch_NLS}
\end{figure}

\subsection{Quantitative Analysis}
We assess the identification performance of WG-IDENT on the PDEs listed in Table~\ref{tab:differential_equations} using four quantitative metrics: the relative error \( E_2 \) defined in~\eqref{relative_e2}, the residual error \( E_{\text{res}} \) defined in~\eqref{Residual_error}, the True Positive Rate (TPR) defined in~\eqref{tpr}, and the Positive Predictive Value (PPV) defined in~\eqref{ppv}. Each experiment is repeated 50 times under varying noise levels \( \sigma_{\text{NSR}} \in \{1\%, 3\%, 5\%, 8\%, 10\%\} \). Figures~\ref{fig:E2_Eres_TPR_PPV_burger_trans_KdV_KS} and~\ref{fig:E2_Eres_TPR_PPV_Sch_NLS} present box plots illustrating the distributions of these metrics. For all PDEs, \( E_2 \) and \( E_{\text{res}} \) remain consistently low across all noise levels, indicating high accuracy in coefficient recovery. Additionally, TPR and PPV values remain near 1, demonstrating WG-IDENT's ability to accurately identify true features with minimal false positives, even under significant noise.

Furthermore, we evaluate the recovery of varying coefficients for the Kuramoto-Sivashinsky (KS) equation, the Nonlinear Schrödinger (NLS) equation—which includes both real and imaginary components—and the Korteweg-de Vries (KdV) equation. Figures~\ref{fig:coef2_reconstruction_KS} through~\ref{fig:coef2_reconstruction_KdV} display the reconstructed coefficients under noise levels of 3\%, 5\%, and 10\%. Across all noise levels, the reconstructed coefficients closely match the true coefficients, confirming WG-IDENT's robustness and effectiveness in recovering the underlying dynamics even in the presence of substantial noise.

We note that in some cases (e.g., the coefficients of $w_{xx}$ in Fig.~\ref{fig:coef2_reconstruction_NLS_Ut} and $v_{xx}$ in Fig.~\ref{fig:coef2_reconstruction_NLS_Vt}), the true coefficients are constant but the reconstructions exhibit fluctuations due to noise.

 While this does not affect the correct identification of the active PDE terms, it highlights the need for future extensions that incorporate additional regularization directly on the coefficients to better distinguish constant from varying coefficients.

In addition to the results shown here, we include more experiments in the Appendix: robustness tests under higher noise levels (\ref{appendixd5:highernoiselevel}) and coefficient recovery results for additional benchmark equations, including the advection–diffusion, viscous Burgers’ and Schrödinger equations (\ref{appendixd7:coeffrecovery}, Fig.\ref{fig:coef2_reconstruction_advectiondiff}–\ref{fig:coef2_reconstruction_Sch_vt}). These supplementary results provide a more complete picture and further corroborate the accuracy and robustness of WG-IDENT.

Overall, these quantitative results validate WG-IDENT's capability to accurately identify the correct PDE structures and recover varying coefficients under noisy conditions.

\begin{figure}[t!]
    \centering
    \begin{tabular}{cccc}
    \hline
        \multicolumn{4}{c}{KS equation}\\\hline
             &  $\sigma_{\text{NSR}}=3\%$ & $5\%$ & $10\%$ \\
    \hline
        $uu_x$ & 
        \raisebox{-0.5\height}{\includegraphics[width=0.25\linewidth]{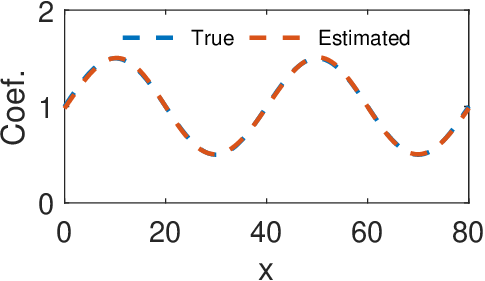}} &
        \raisebox{-0.5\height}{\includegraphics[width=0.25\linewidth]{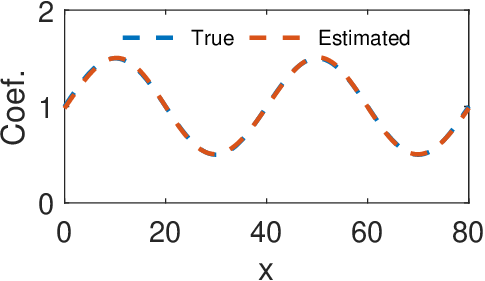}}&
        \raisebox{-0.5\height}{\includegraphics[width=0.25\linewidth]{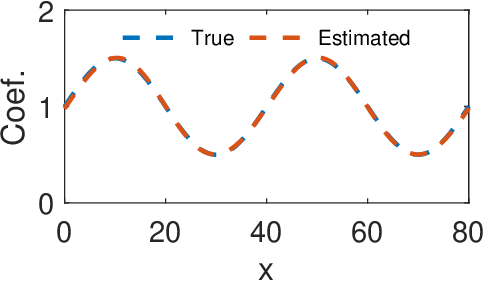}}\\
    \hline 
        $u_{xx}$ &
        \raisebox{-0.5\height}{\includegraphics[width=0.25\linewidth]{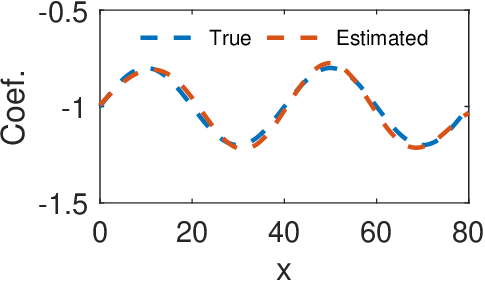}} &
        \raisebox{-0.5\height}{\includegraphics[width=0.25\linewidth]{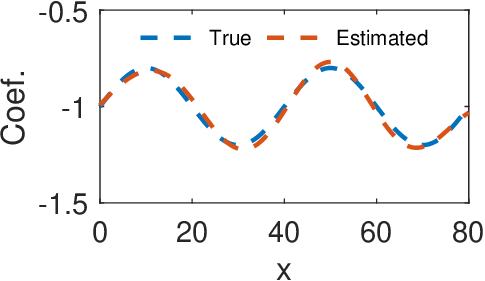}}&
        \raisebox{-0.5\height}{\includegraphics[width=0.25\linewidth]{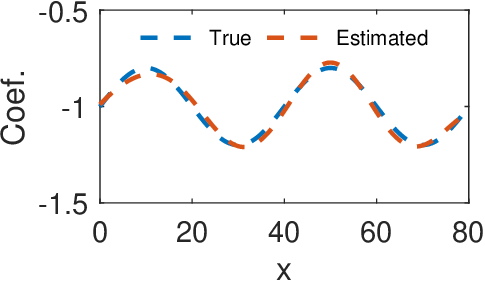}}\\
    \hline
        $u_{xxxx}$ &
        \raisebox{-0.5\height}{\includegraphics[width=0.25\linewidth]{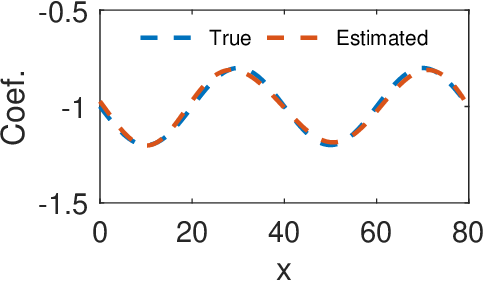}} &
        \raisebox{-0.5\height}{\includegraphics[width=0.25\linewidth]{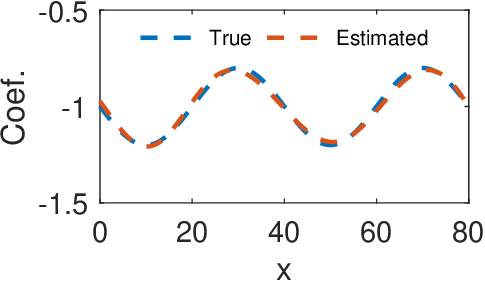}}&
        \raisebox{-0.5\height}{\includegraphics[width=0.25\linewidth]{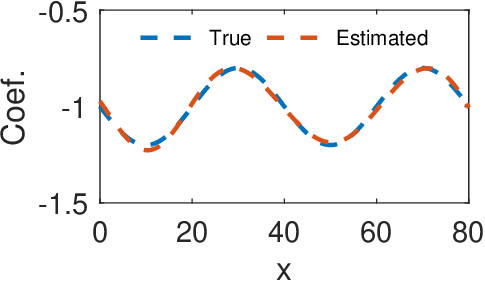}}\\
    \hline
    \end{tabular}
    \caption{Estimated coefficients of different terms in KS equation with $3\%$, $5\%$ and $10\%$ noise. The first row shows the estimated coefficients of $uu_x$. The second row shows the estimated coefficients of $u_{xx}$.  The third row shows the estimated coefficients of $u_{xxxx}$. }
    \label{fig:coef2_reconstruction_KS}
\end{figure}

\begin{figure}[t!]
    \centering
    \begin{tabular}{cccc}
    \hline
    \multicolumn{4}{c}{Nonlinear Schr\"{o}dinger equation real part}\\\hline
         &  $\sigma_{\text{NSR}}=3\%$ & $5\%$ & $10\%$ \\
    \hline
        $v^2w$ & 
        \raisebox{-0.5\height}{\includegraphics[width=0.25\linewidth]{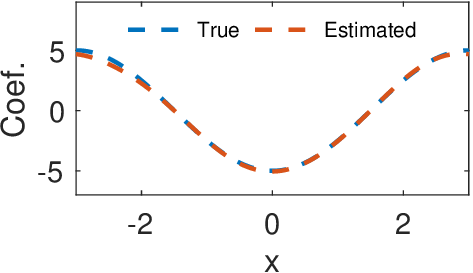}} &
        \raisebox{-0.5\height}{\includegraphics[width=0.25\linewidth]{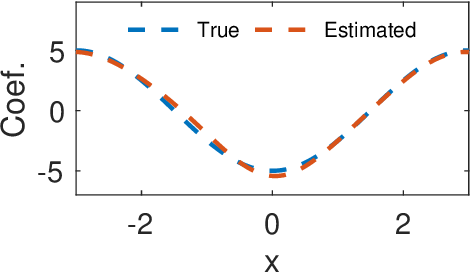}}&
        \raisebox{-0.5\height}{\includegraphics[width=0.25\linewidth]{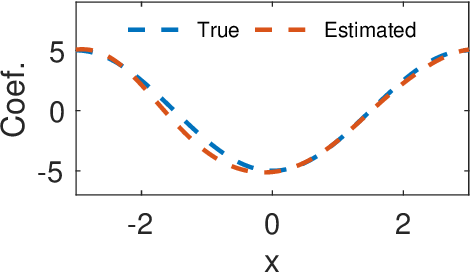}}\\
    \hline 
        $w^3$ &
        \raisebox{-0.5\height}{\includegraphics[width=0.25\linewidth]{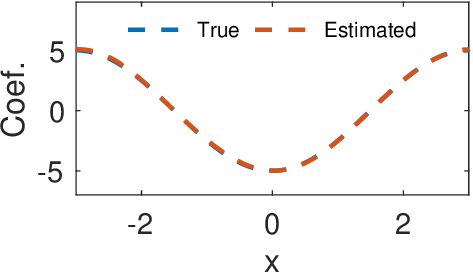}} &
        \raisebox{-0.5\height}{\includegraphics[width=0.25\linewidth]{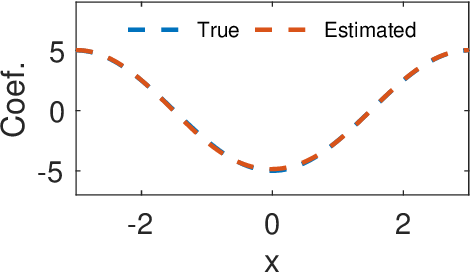}}&
        \raisebox{-0.5\height}{\includegraphics[width=0.25\linewidth]{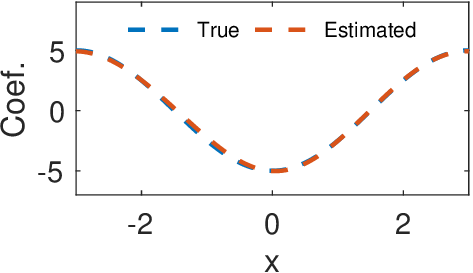}}\\
    \hline
        $w_{xx}$ &
        \raisebox{-0.5\height}{\includegraphics[width=0.25\linewidth]{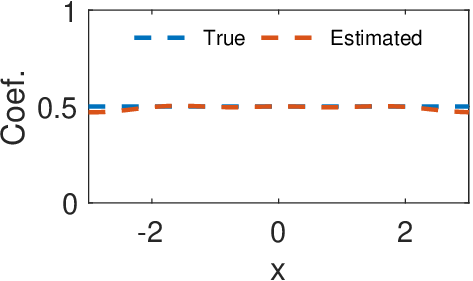}} &
        \raisebox{-0.5\height}{\includegraphics[width=0.25\linewidth]{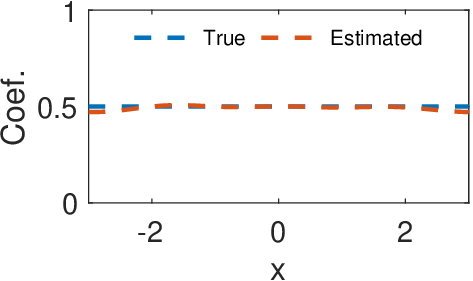}}&
        \raisebox{-0.5\height}{\includegraphics[width=0.25\linewidth]{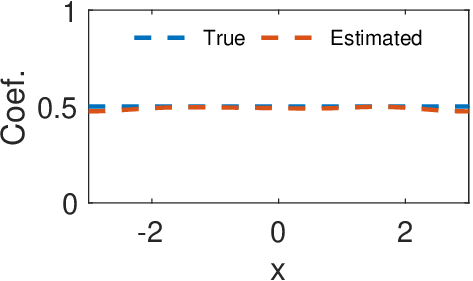}}\\
    \hline
    \end{tabular}
    \caption{Estimated coefficients of the real components of different terms in NLS equation with $3\%$, $5\%$ and $10\%$ noise. The first row shows the estimated coefficients of $v^2w$. The second row shows the estimated coefficients of $w^3$.  The third row shows the estimated coefficients of $w_{xx}$.}
    \label{fig:coef2_reconstruction_NLS_Ut}
\end{figure}

\begin{figure}[t!]
    \centering
    \begin{tabular}{cccc}
    \hline
    \multicolumn{4}{c}{Nonlinear Schr\"{o}dinger equation imaginary part}\\\hline
         & $\sigma_{\text{NSR}}=3\%$ & $5\%$ & $10\%$ \\
    \hline
        $w^2v$ & 
        \raisebox{-0.5\height}{\includegraphics[width=0.25\linewidth]{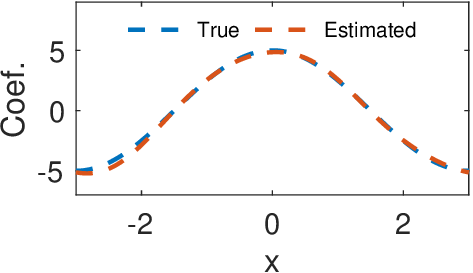}} &
        \raisebox{-0.5\height}{\includegraphics[width=0.25\linewidth]{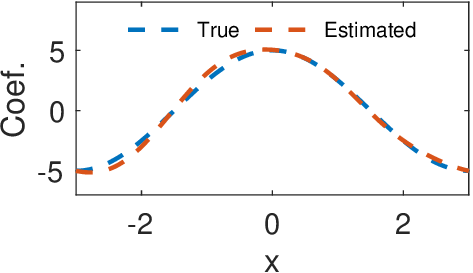}}&
        \raisebox{-0.5\height}{\includegraphics[width=0.25\linewidth]{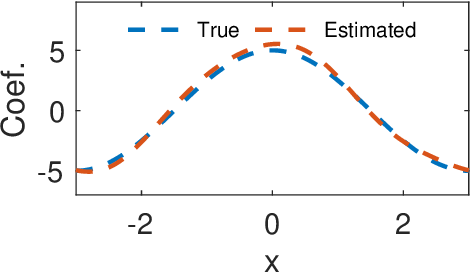}}\\
    \hline 
        $v^3$ &
        \raisebox{-0.5\height}{\includegraphics[width=0.25\linewidth]{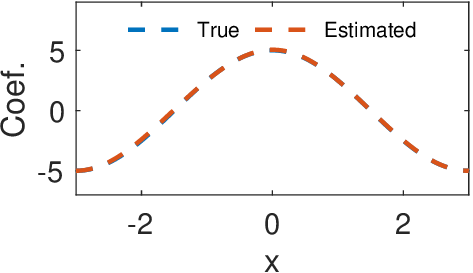}} &
        \raisebox{-0.5\height}{\includegraphics[width=0.25\linewidth]{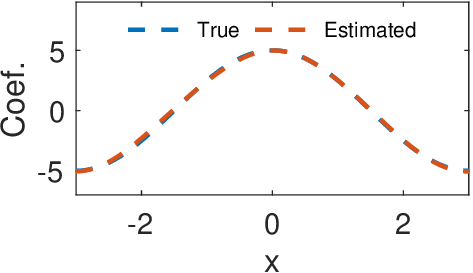}}&
        \raisebox{-0.5\height}{\includegraphics[width=0.25\linewidth]{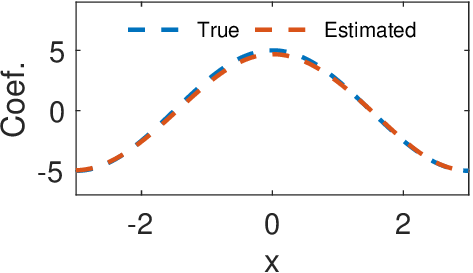}}\\
    \hline
        $v_{xx}$ &
        \raisebox{-0.5\height}{\includegraphics[width=0.25\linewidth]{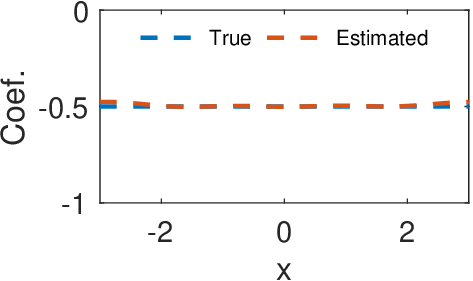}} &
        \raisebox{-0.5\height}{\includegraphics[width=0.25\linewidth]{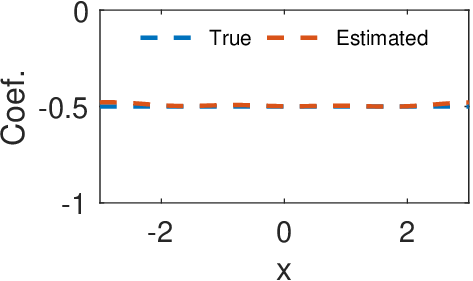}}&
        \raisebox{-0.5\height}{\includegraphics[width=0.25\linewidth]{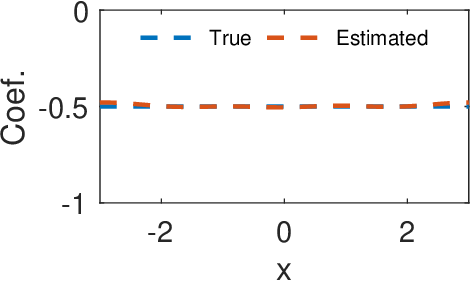}}\\
    \hline
    \end{tabular}
    \caption{Estimated coefficients of the imaginary components of different terms in NLS equation with $3\%$, $5\%$ and $10\%$ noise. The first row shows the estimated coefficients of $w^2v$. The second row shows the estimated coefficients of $v^3$.  The third row shows the estimated coefficients of $v_{xx}$.  
    }
    \label{fig:coef2_reconstruction_NLS_Vt}
\end{figure}

\begin{figure}[t!]
    \centering
    \begin{tabular}{cccc}
    \hline
    \multicolumn{4}{c}{KdV equation}\\\hline
            & $\sigma_{\text{NSR}}=3\%$ & $5\%$ & $10\%$ \\
    \hline
        $uu_x$ & 
        \raisebox{-0.5\height}{\includegraphics[width=0.25\linewidth]{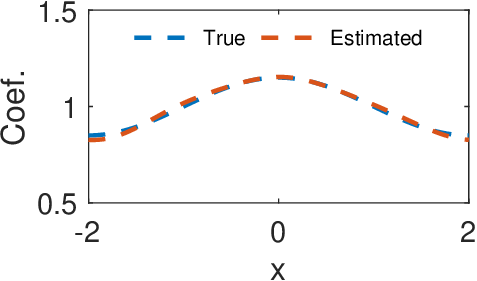}} &
        \raisebox{-0.5\height}{\includegraphics[width=0.25\linewidth]{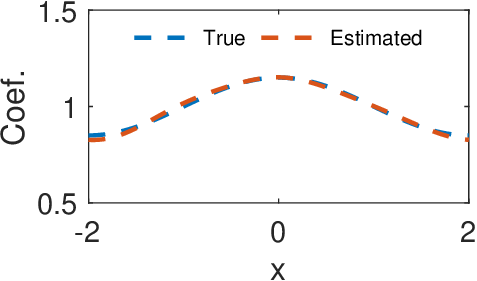}}&
        \raisebox{-0.5\height}{\includegraphics[width=0.25\linewidth]{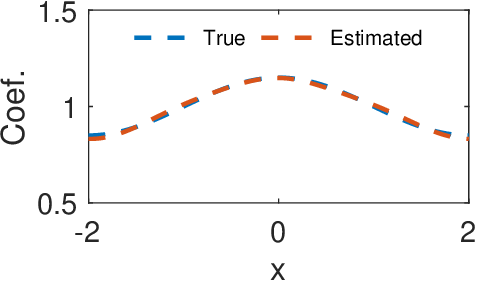}}\\
    \hline 
        $u_{xxx}$ &
        \raisebox{-0.5\height}{\includegraphics[width=0.25\linewidth]{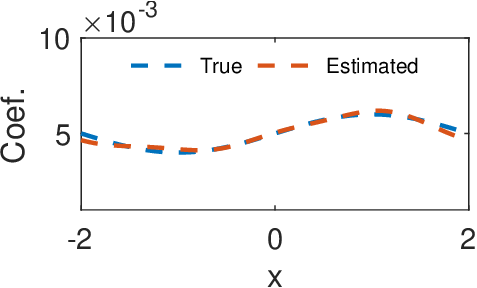}} &
        \raisebox{-0.5\height}{\includegraphics[width=0.25\linewidth]{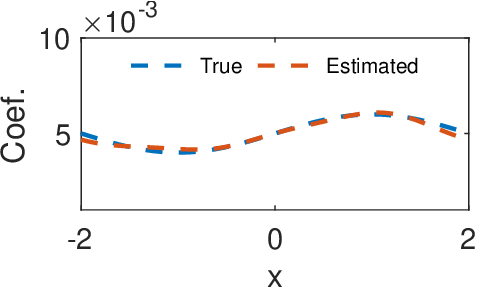}}&
        \raisebox{-0.5\height}{\includegraphics[width=0.25\linewidth]{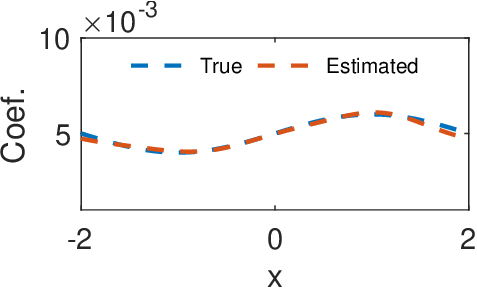}}\\
    \hline
    \end{tabular}
    \caption{Estimated coefficients of different terms in KdV equation with $3\%$, $5\%$ and $10\%$ noise. The first row shows the estimated coefficients of $uu_x$. The second row shows the estimated coefficients of $u_{xxx}$.}
    \label{fig:coef2_reconstruction_KdV}
\end{figure}

\subsection{Effects of GF-Trim} \label{sec:benefittrim}
We show that GF-Trim improves the robustness of the Reduction in Residual (RR) criterion by allowing a wider range of the threshold \(\rho^R\).

In this experiment, we use the viscous Burgers' equation with the following initial condition:
\[
u(x,0) = 2.5\left(\sin(2\pi x) + 2.5\cos\left(2\pi x + 0.25\right)\right)\;, x\in[0,2)
\]
to demonstrate the benefits of the GF-Trim process. We compare the performance of our method with and without trimming across different sparsity levels, as illustrated in Figure~\ref{fig:RRcomparison}. In the figure, the `$\times$' marker indicates incorrect PDE identifications, while the `$\circ$' marker signifies correct identifications.

\textbf{Without GF-Trim} as illustrated in Figure~\ref{fig:RRcomparison}(a): Only the candidate PDE with the correct sparsity level (\(\theta = 2\)), corresponding to a reduction in residual (RR) value of 0.0254, matches the true PDE. This implies that our method is effective only if the threshold \(\rho^R\) is set below 0.0254. Such a narrow valid range for \(\rho^R\) makes the method less robust, especially in scenarios with higher noise and perturbations.

\textbf{With GF-Trim} as illustrated in Figure~\ref{fig:RRcomparison}(b): The GF-Trim process effectively removes incorrect features in candidate PDEs with sparsity levels greater than 2. This refinement ensures that only the most significant features contributing to the true PDE are retained. Consequently, our method successfully identifies the correct PDE even when the sparsity level exceeds 2, as long as the threshold \(\rho^R\) is set below 0.3026. This substantially increases the acceptable range for \(\rho^R\), thereby enhancing the method's robustness and reducing sensitivity to the exact sparsity level.

This comparison demonstrates that GF-Trim significantly enhances the robustness and reliability of our method. By effectively eliminating spurious features and allowing for a wider selection of the threshold parameter \(\rho^R\) within RR, GF-Trim ensures accurate identification of the true PDE even in the presence of substantial noise.

\subsection{Group-wise feature contribution analysis}\label{eq_gp_trim_contribution}
In this experiment, we compare our proposed GF-Trim method, designed to remove entire low-contribution feature groups, against a more traditional ``column-wise'' trimming approach \cite{tang2022weakident}, which discards columns (features) based solely on their individual contributions. We stimulate the viscous Burger's equation with 16 candidate feature groups. In the ground-truth configuration, only two group features (labeled Group 4 and Group 8) are active.

We first apply the GPSP algorithm to identify the candidate features under the corresponding sparsity level specified. Here, with the specified sparsity level \(\theta=6\), the generated candidate group features are group 2, 4, 7, 8, 12 and 13. Then we apply the GF-Trim technique as introduced in Section \ref{sec:G_trim}. The proposed GF-Trim method correctly identified Group 4 and 8 (highlighted in blue) as the only active feature groups. By focusing on each group's aggregate contribution, GF-Trim effectively filters out spurious contributions, which might otherwise inflate the importance of individual features within low-importance groups. When trimming is performed purely on a column-by-column basis, some columns in Group 7 and Group 12 appear to have higher individual contributions than some columns in the true groups (Group 4 and 8). As shown in Figure~\ref{fig:GFtrimNTrim_comparison}, the lowest column contribution among the true feature groups is 0.0964, which is even lower than several columns in the wrong groups Group 7 and Group 12. This can mislead the column-wise method into retaining columns that largely fit noise rather than the true underlying dynamics. 

While one might expect the identified feature groups at successive sparsity levels $\theta$ to be nested (as in Group Lasso), our framework does not impose such monotonicity. As illustrated in \ref{appendixd6:gftrim}, the active groups can change across $\theta$, but the GF-Trim step can prune spurious terms and recover the correct set once $\theta$ reaches the true sparsity level.

The numerical experiment confirms that GF-Trim outperforms column-wise trimming in detecting correct features. By capturing feature importance at the group level rather than at the column level, GF-Trim effectively discards spurious columns with misleadingly high contributions. GF-Trim is a more robust and accurate method for identifying the underlying PDE structure, especially in problems where varying coefficients for inactive features give rise to columns that might appear significant but do not truly contribute to the model. 

\begin{figure}[t!]
    \centering
    \begin{tabular}{cc}
    \includegraphics[width=0.38\textwidth]{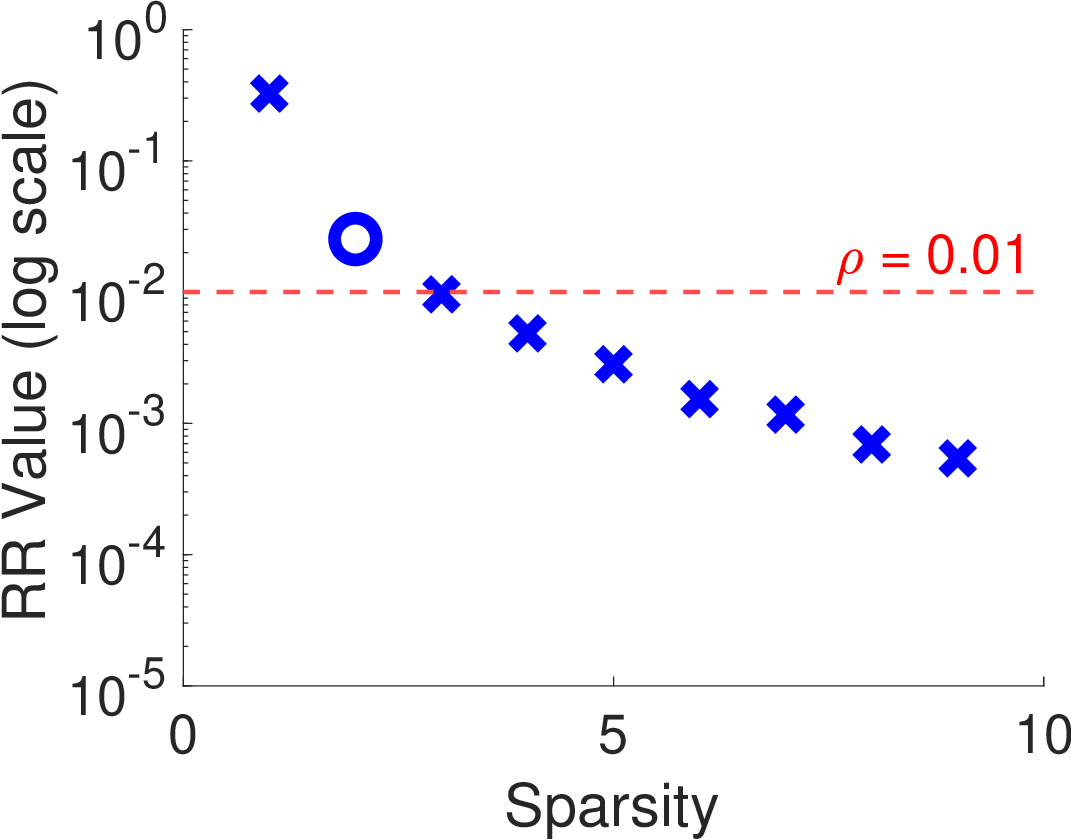}&
         \includegraphics[width=0.38\textwidth]{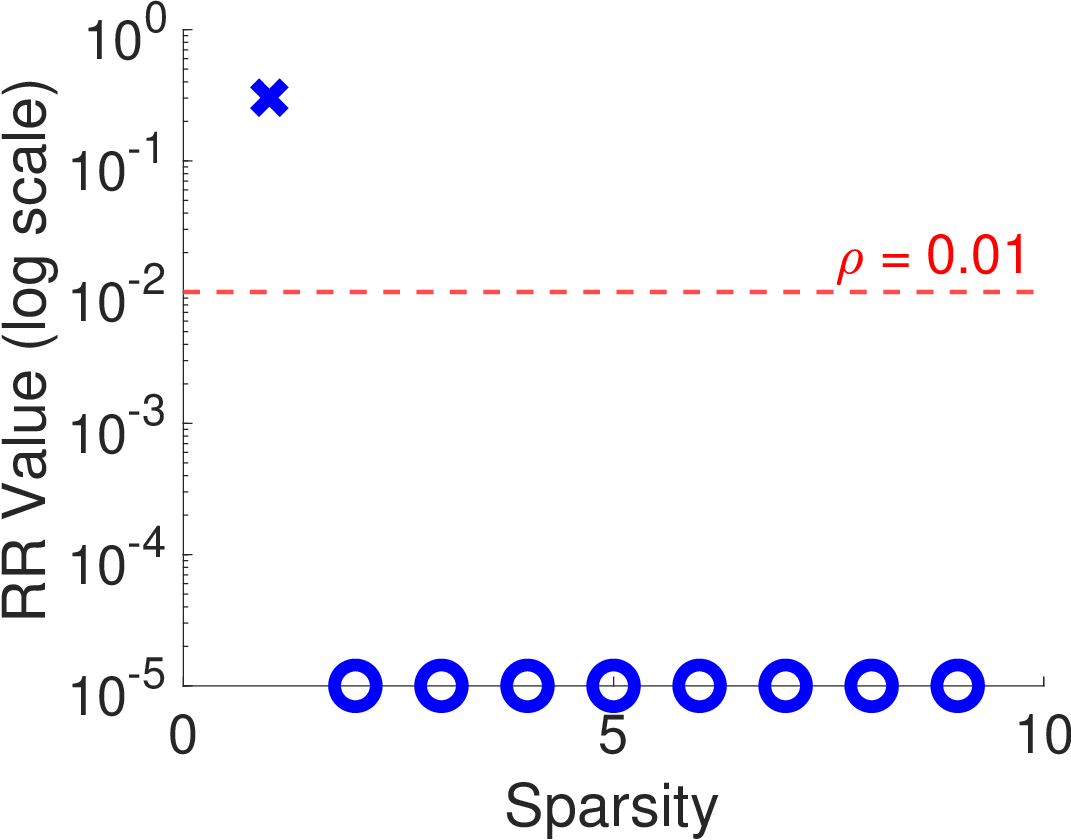} \\
         (a) RR without GF-Trim & (b) RR with GF-Trim
    \end{tabular}
    \caption{Comparison of Residual Ratio (RR) values with and without GF-Trim. (a) Without GF-Trim, (b) With GF-Trim. In both plots, an `$\times$' marker denotes an incorrect PDE candidate, while a `$\circ$' marker signifies a correct PDE identification. The y-axis is presented on a logarithmic scale to highlight differences in RR values, especially at lower magnitudes. For visualization purposes, RR values below $10^{-5}$ are represented as $10^{-5}$ to improve the clarity of the plot.}
    \label{fig:RRcomparison}
\end{figure}

\begin{figure}[t!] \centering \includegraphics[width=0.55\textwidth]{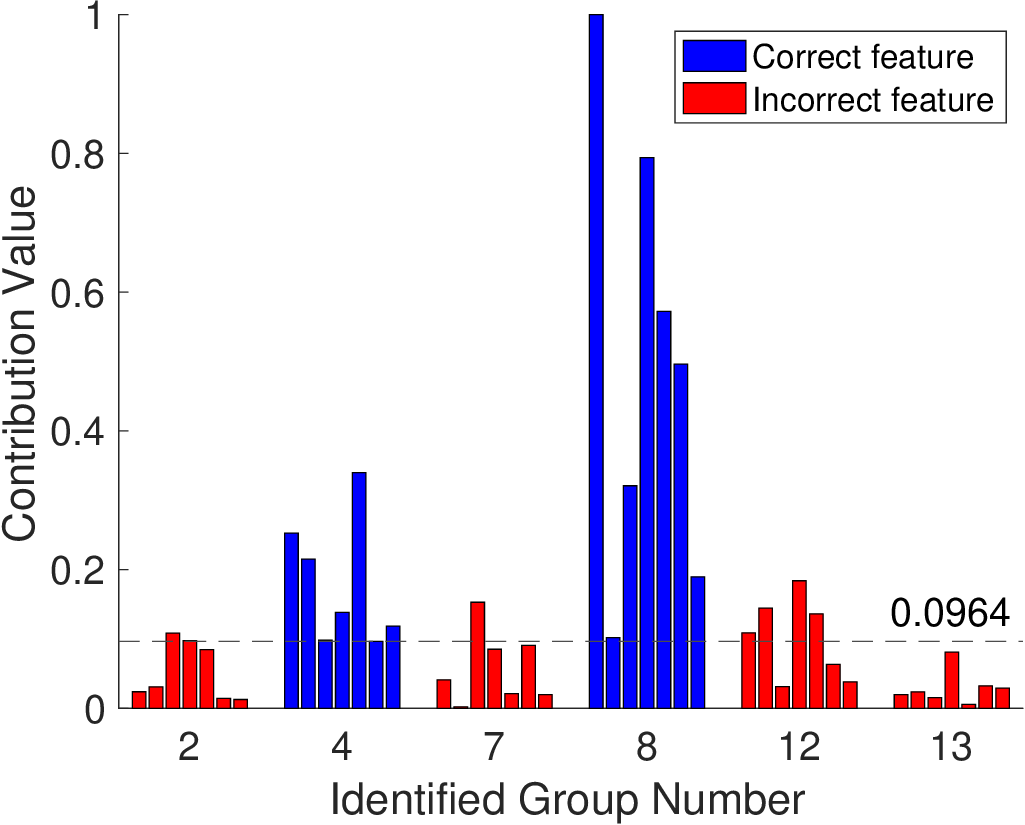} \caption{Group-wise feature contribution analysis. The plot illustrates the contribution of each feature within the identified groups (2, 4, 7, 8, 12, and 13). Correctly identified groups (4 and 8) are shown in blue, while incorrectly identified groups are shown in red. The value 0.0964 is the minimum contribution within the true feature groups, which is even lower than the contribution of certain columns within the incorrect group features. Whereas our proposed GF-Trim can still effectively identify the correct group features.} \label{fig:GFtrimNTrim_comparison} \end{figure}

\subsection{Robustness of B-spline Test Functions} \label{sec:robustness_of_bspline}

In our proposed method, B-splines are utilized as test functions, differing from the truncated polynomials employed in \cite{Messenger_2021,tang2022weakident} (see Remark~\ref{remark_test}). A key distinction between the two test functions is that B-splines form a partition of unity, ensuring that their combined influence sums exactly to one across the domain. This property provides a consistent weighting framework for feature importance throughout the domain. In contrast, truncated polynomial test functions do not satisfy the partition of unity property, potentially leading to inconsistent weighting across the domain. As a result, B-splines provide a more robust and stable framework for accurate PDE identification, especially in noisy environments.

To evaluate the robustness of B-spline test functions, we compare their performance against truncated polynomial test functions within our framework by identifying the viscous Burgers' equation under various noise levels. The results are presented in Figure~\ref{fig:comparing_WSINDy_testfnc}, where each experiment is repeated 20 times. Our findings demonstrate that B-spline test functions maintain a True Positive Rate (TPR) of 1 across all noise levels, outperforming truncated polynomial test functions. 
Furthermore, to further validate the stability of our weak formulation, we conduct an additional comparison between the space--time weak formulation and the conventional spatial-only weak formulation augmented with numerical differentiation (\ref{appendix_wfinspace}). The results show that the space--time formulation consistently yields smaller identification errors ($E_2$), confirming its superior noise resistance and numerical stability.

Overall, the comparison highlights that employing B-spline test functions significantly enhances the robustness and reliability of WG-IDENT, enabling accurate identification of the true PDE even under substantial noise conditions.

\begin{figure}[t!]
    \centering
    \begin{tabular}{cc}

        \hline
        \multicolumn{2}{c}{Viscous Burgers' equation} \\
        \hline
        \hspace{1.0cm} truncated polynomial  & B-spline \\
        \includegraphics[width=0.35\linewidth]{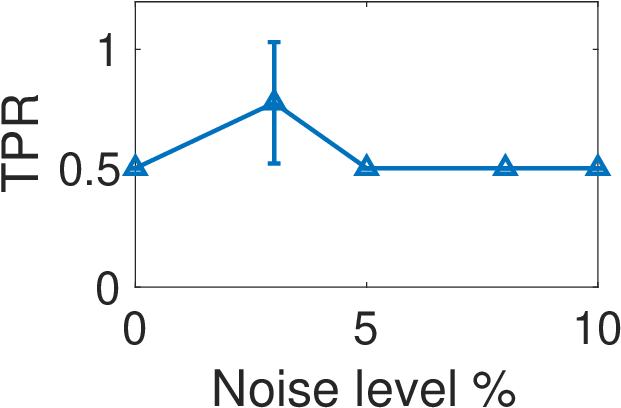} &
        \includegraphics[width=0.35\linewidth]{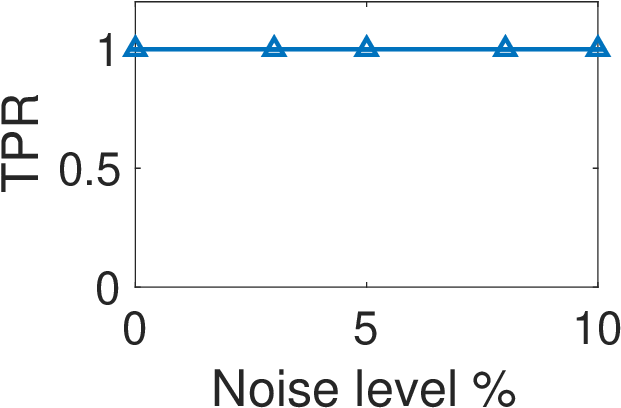} \\
        \hline
        \multicolumn{2}{c}{KS equation} \\
        \hline
        \hspace{1.0cm} truncated polynomial  & B-spline \\
        \includegraphics[width=0.35\linewidth]{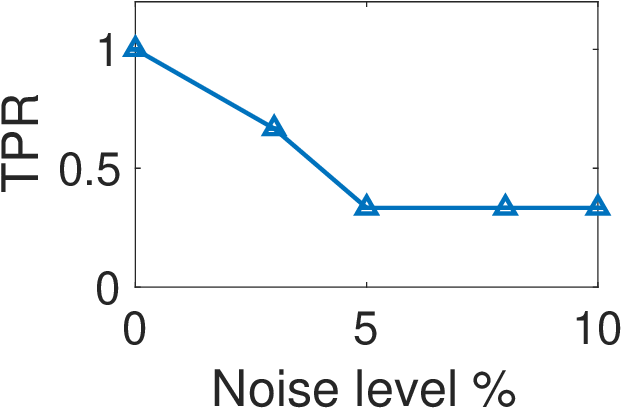} &
        \includegraphics[width=0.35\linewidth]{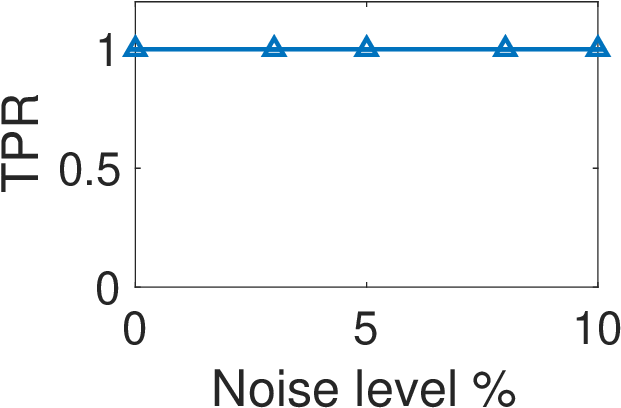}
        \\
        \hline
    \end{tabular}

    \caption{Comparison of test functions. The TPR of identifying viscous Burgers' equation and KS equation by our method with truncated polynomial test function \cite{Messenger_2021,tang2022weakident}  and B-spline test function (proposed in this paper) with various noise levels.}
    \label{fig:comparing_WSINDy_testfnc}
\end{figure}

\subsection{Comparison study with state-of-the-art methods} \label{sec:comparison_study}

We conduct a series of comprehensive comparison studies between WG-IDENT and other state-of-the-art methods for identifying PDEs with varying coefficients. Specifically, we investigate their effectiveness from the following two perspectives.
\\

\noindent\textbf{Metric evaluation:}~We compare the proposed method with some other  methods: GLASSO \cite{friedman2010note}, SGTR \cite{Rudy_sparse_reg} and rSGTR \cite{Li_Sun_Zhao_Lehman_2020}.
To ensure fairness, all methods utilize identical dictionaries: the default 16-feature dictionary for viscous Burgers' equation, and an extended 46-feature dictionary for the Schr\"{o}dinger equation and nonlinear Schr\"{o}dinger (NLS) equation, as detailed in Table~\ref{tab:differential_equations}. As evidenced in Table~\ref{tab:burgerSchNLS_complexsys}, WG-IDENT demonstrates  consistent feature identification accuracy across all noise levels and for all tested differential equations. While for other methods, GLASSO fails to identify the correct features across all noise levels. SGTR and rSGTR can identify the correct features at low noise levels and fail to identify under the high noise levels. This comparison demonstrates the robustness of the proposed method against noise. The comparison results for the other three PDEs, advection diffusion equation, KdV equation and KS equations, are shown in the Table \ref{tab:AdvectionKdVKS_eq} in the \ref{appendix_comparemethods}.\\[1pt]

\begin{table}[t!]
\footnotesize
\centering
{\centering \bfseries (a) Viscous Burgers' Equation \par}
\begin{tabular}{ccccc}
\toprule
Noise level & WG-IDENT & GLASSO & SGTR & rSGTR \\
\midrule
No Noise & $uu_x, u_{xx}$ & $\geq$ 4 terms & $uu_x, u_{xx}$ & $uu_x, u_{xx}$ \\
0.1\%  & $uu_x, u_{xx}$ & $\geq$ 4 terms & $uu_x, u_{xx}$ & $uu_x, u_{xx}$ \\
0.5\%  & $uu_x, u_{xx}$ & $\geq$ 4 terms & $uu_x, u_{xx}$ & $uu_x, u_{xx}$ \\
1\%    & $uu_x, u_{xx}$ & $\geq$ 4 terms & $uu_x, u_{xx}$ & $uu_x, u_{xx}$ \\
5\%    & $uu_x, u_{xx}$ & $\geq$ 4 terms & $\geq$ 4 terms & $\geq$ 4 terms  \\
8\%    & $uu_x, u_{xx}$ & $\geq$ 4 terms & $\geq$ 4 terms & $\geq$ 4 terms  \\
10\%   & $uu_x, u_{xx}$ & $\geq$ 4 terms & $\geq$ 4 terms & $\geq$ 4 terms  \\
\bottomrule
\end{tabular}
\vspace{0.5cm} 

{\centering \bfseries (b) Schr\"{o}dinger Equation (Real component) \par}
\begin{tabular}{ccccc}
\toprule
Noise level & WG-IDENT & GLASSO & SGTR & rSGTR \\
\midrule
No Noise & $w, w_{xx}$ & $\geq 4\ \text{terms}$ & $w, w_{xx}$ & $w, w_{xx}$ \\
0.1\%  & $w, w_{xx}$ & $\geq 4\ \text{terms}$ & $w, w_{xx}$ & $w, w_{xx}$ \\
0.5\%  & $w, w_{xx}$ & $\geq 4\ \text{terms}$ & $w, w_{xx}, w_{xxxx}$ & $w, w_{xx}, w_{xxxx}$ \\
1\%    & $w, w_{xx}$ & $\geq 4\ \text{terms}$ & $\geq 4\ \text{terms}$ & $\geq 4\ \text{terms}$ \\
5\%    & $w, w_{xx}$ & $\geq 4\ \text{terms}$ & $w, w_x, w^2, w^3$ & $w, w_x, w^2, w^3$  \\
8\%    & $w, w_{xx}$ & $\geq 4\ \text{terms}$ & $w, w^2, w^3$ & $w, w^2, w^3$  \\
10\%   & $w, w_{xx}$ & $\geq 4\ \text{terms}$ & $w$ & $w$  \\
\bottomrule
\end{tabular}
\vspace{0.5cm} 

{\centering \bfseries (c) Nonlinear Schr\"{o}dinger (NLS) Equation (Real component) \par}
\begin{tabular}{ccccc}
\toprule
Noise level & WG-IDENT & GLASSO & SGTR & rSGTR \\
\midrule
No Noise & $w_{xx}, w^3, v^2w$ & $\geq 5\ \text{terms}$ & $w_{xx}, w^3, v^2w$ & $w_{xx}, w^3, v^2w$ \\
0.1\%  & $w_{xx}, w^3, v^2w$ & $\geq 5\ \text{terms}$ & $w_{xx}, w^3, v^2w$ & $w_{xx}, w^3, v^2w$ \\
0.5\%  & $w_{xx}, w^3, v^2w$ & $\geq 5\ \text{terms}$ & $w_{xx}, w^3, v^2w$ & $w_{xx}, w^3, v^2w$ \\
1\%    & $w_{xx}, w^3, v^2w$ & $\geq 5\ \text{terms}$ & $w_{xx}, w^3, v^2w$ & $w_{xx}, w^3, v^2w$ \\
5\%    & $w_{xx}, w^3, v^2w$ & $\geq 5\ \text{terms}$ & $w_{xx}, w^3$ & $w_{xx}, w^3$  \\
8\%    & $w_{xx}, w^3, v^2w$ & $\geq 5\ \text{terms}$ & $w_{xx}, w^3$ & $w_{xx}, w^3$  \\
10\%   & $w_{xx}, w^3, v^2w$ & $\geq 5\ \text{terms}$ & $w_{xx}, w^3$ & $w_{xx}, w^3$  \\
\bottomrule
\end{tabular}
\caption{Comparison of different methods for identifying (a) viscous Burgers' equation, (b) Schr\"{o}dinger equation, and (c) Nonlinear Schr\"{o}dinger (NLS) equation. The table presents the identified features using WG-IDENT (the proposed method), GLASSO, SGTR, and rSGTR across various noise levels. WG-IDENT consistently identifies the correct features at all noise levels and exhibits superior robustness compared to the other methods. For the Schr\"{o}dinger equation and NLS equation, the table presents results for real component. The results for the imaginary component are similar.}
\label{tab:burgerSchNLS_complexsys}
\end{table}

\noindent\textbf{Dictionary size.} We investigate the influence of different dictionary sizes on the identification results of the proposed method, WG-IDENT, and compare our method with GLASSO \cite{friedman2010note}, SGTR \cite{Rudy_sparse_reg}, and rSGTR \cite{Li_Sun_Zhao_Lehman_2020} across varying noise levels.
We consider three dictionaries: Dictionary I comprises 7 features, including partial derivatives of $u$ and $u^2$ up to the second order. Dictionary II is our default dictionary, consisting of 16 features, incorporating partial derivatives of $u$, $u^2$ and $u^3$ up to the fourth order. Dictionary III contains 43 features, covering partial derivatives of $u$, $u^2$, $u^3$, $u^4$, $u^5$ and $u^6$ up to the sixth order.  The results by different methods on identifying viscous Burgers' equation are presented in Table \ref{tab:comparison1_dic}. WG-IDENT can accurately identify all correct features across different noise level among the three dictionaries. In Dictionary I. all the three methods can identify the correct features across different noise levels. GLASSO completely fails in Dictionaries II and III regardless of noise. Both SGTR and rSGTR successfully identify correct features across all three dictionaries under noise-free conditions. However, their performance degrades significantly with increasing noise levels—they maintain accuracy only in Dictionary II under 1\% noise level. Notably, our method does not require parameter adjustments for different noise levels or dictionary sizes. In contrast, other methods need adjusting hyper-parameters in order to successively identify the PDE. 

We apply the same procedures to identify advection diffusion equation, with results summarized in Table \ref{tab:comparison2_dic}. With clean data, all methods correctly identify the true PDE. Under 3\% noise level, GLASSO, SGTR, and rSGTR remain successful. At higher noise levels of 8\% and 10\%, they succeed only with Dictionary I. These methods also require adjustments to their differentiation techniques to handle varying noise levels and dictionary sizes.\\[1pt]

\begin{table}[t!]
\centering
\begin{tabular}{ c c c c c c}
\hline
\multicolumn{6}{c}{Viscous Burgers' equation}\\\hline
Noise level& & {WG-IDENT} & {GLASSO} & {SGTR} & {rSGTR} \\
\hline
\multirow{3}{*}{{No noise}} 
& Dict. I & $u_{xx}, uu_x$ & $u_{xx}, uu_x$ & $u_{xx}, uu_x$ & $u_{xx}, uu_x$  \\
& Dict. II & $u_{xx}, uu_x$ & $\geq$ 4 terms & $u_{xx}, uu_x$ & $u_{xx}, uu_x$  \\
& Dict. III & $u_{xx}, uu_x$ & $\geq$ 4 terms & $u_{xx}, uu_x$ & $u_{xx}, uu_x$\\
\hline
\multirow{3}{*}{{1\% noise}} 
& Dict. I & $u_{xx}, uu_x$ & $u_{xx}, uu_x$ & $u_{xx}, uu_x$ & $u_{xx}, uu_x$  \\
& Dict. II & $u_{xx}, uu_x$ & $\geq$ 4 terms & $u_{xx}, uu_x$ & $u_{xx}, uu_x$ \\
& Dict. III & $u_{xx}, uu_x$ & $\geq$ 4 terms & $\geq$ 4 terms & $\geq$ 4 terms  \\
\hline
\multirow{3}{*}{{5\% noise}} 
& Dict. I & $u_{xx}, uu_x$ & $u_{xx}, uu_x$ & $u_{xx}, uu_x$ & $u_{xx}, uu_x$  \\
& Dict. II & $u_{xx}, uu_x$ & $\geq$ 4 terms & $\geq$ 4 terms & $\geq$ 4 terms\\
& Dict. III & $u_{xx}, uu_x$ & $\geq$ 4 terms & $\geq$ 4 terms & $\geq$ 4 terms \\
\hline

\multirow{3}{*}{{10\% noise}} 
& Dict. I & $u_{xx}, uu_x$ & $u_{xx}, uu_x$ & $u_{xx}, uu_x$ & $u_{xx}, uu_x$  \\
& Dict. II & $u_{xx}, uu_x$ & $\geq$ 4 terms & $\geq$ 4 terms & $\geq$ 4 terms  \\
& Dict. III & $u_{xx}, uu_x$ & $\geq$ 4 terms & $\geq$ 4 terms & $\geq$ 4 terms \\
\hline
\end{tabular}
\caption{Experiments on identifying the viscous Burgers' equation with different dictionary sizes and various noise levels. 
Dictionary~I comprises 7 features, including partial derivatives of $u$ and $u^2$ up to the second order. 
Dictionary~II, which serves as our default setting, consists of 16 features incorporating partial derivatives of $u$, $u^2$, and $u^3$ up to the fourth order. 
Dictionary~III contains 42 features, covering partial derivatives of $u$, $u^2$, $u^3$, $u^4$, $u^5$, and $u^6$ up to the sixth order. 
The table presents identification results using WG-IDENT (the proposed method), GLASSO~\cite{friedman2010note}, SGTR~\cite{Rudy_sparse_reg}, and rSGTR~\cite{Li_Sun_Zhao_Lehman_2020}. 
WG-IDENT consistently identifies the correct terms across all dictionaries and noise levels without requiring parameter adjustments, while higher-order dictionaries demonstrate the increasing difficulty of estimating nonlinear and high-derivative features in the presence of noise.}
\label{tab:comparison1_dic}
\end{table}

\begin{table}[t!]
\centering
\begin{tabular}{ c c c c c c}
\hline
    \multicolumn{6}{c}{Advection diffusion equation}\\\hline
Noise level&  & {WG-IDENT} & {GLASSO} & {SGTR} & {rSGTR} \\
\hline
\multirow{3}{*}{{No noise}} 
& Dict. I & $u_{x}, u_{xx}$ & $u_{x}, u_{xx}$ & $u_{x}, u_{xx}$ & $u_{x}, u_{xx}$  \\
& Dict. II & $u_{x}, u_{xx}$ & $u_{x}, u_{xx}$ & $u_{x}, u_{xx}$ & $u_{x}, u_{xx}$\\
& Dict. III & $u_{x}, u_{xx}$ & $u_{x}, u_{xx}$ & $u_{x}, u_{xx}$ & $u_{x}, u_{xx}$\\
\hline
\multirow{3}{*}{{3\% noise}} 
& Dict. I & $u_{x}, u_{xx}$ & $u_{x}, u_{xx}$ & $u_{x}, u_{xx}$ & $u_{x}, u_{xx}$ \\
& Dict. II & $u_{x}, u_{xx}$ & $u_{x}, u_{xx}$ & $u_{x}, u_{xx}$ & $u_{x}, u_{xx}$ \\
& Dict. III & $u_{x}, u_{xx}$ & $u_{x}, u_{xx}$ & $u_{x}, u_{xx}$ & $u_{x}, u_{xx}$\\
\hline
\multirow{3}{*}{{8\% noise}} 
& Dict. I & $u_{x}, u_{xx}$ & $u_{x}, u_{xx}$ & $u_{x}, u_{xx}$ & $u_{x}, u_{xx}$   \\
& Dict. II & $u_{x}, u_{xx}$ & $\geq$ 4 terms & $u_{x}, u_{xxxx}$ & $u_{x}, u_{xxxx}$ \\
& Dict. III & $u_{x}, u_{xx}$ & $\geq$ 4 terms & $\geq$ 4 terms & $\geq$ 4 terms  \\
\hline
\multirow{3}{*}{{10\% noise}} 
& Dict. I & $u_{x}, u_{xx}$ & $u_{x}, u_{xx}$ & $u_{x}, u_{xx}$ & $u_{x}, u_{xx}$   \\
& Dict. II & $u_{x}, u_{xx}$ & $\geq$ 4 terms & $\geq$ 4 terms & $\geq$ 4 terms \\
& Dict. III & $u_{x}, u_{xx}$ & $\geq$ 4 terms & $\geq$ 4 terms & $\geq$ 4 terms  \\
\hline
\end{tabular}
\caption{Experiments on identifying the advection–diffusion equation with different dictionary sizes and various noise levels. 
Dictionary~I comprises 7 features, including partial derivatives of $u$ and $u^2$ up to the second order. 
Dictionary~II, which serves as our default setting, consists of 16 features incorporating partial derivatives of $u$, $u^2$, and $u^3$ up to the fourth order. 
Dictionary~III contains 42 features, covering partial derivatives of $u$, $u^2$, $u^3$, $u^4$, $u^5$, and $u^6$ up to the sixth order. The table presents identification results using WG-IDENT (the proposed method), GLASSO~\cite{friedman2010note}, SGTR~\cite{Rudy_sparse_reg}, and rSGTR~\cite{Li_Sun_Zhao_Lehman_2020}. 
WG-IDENT consistently identifies the correct terms across all dictionaries and noise levels without requiring parameter adjustments.}
\label{tab:comparison2_dic}
\end{table}

\section{Conclusion}
\label{sec:conclusion}
In this paper, we propose an effective and robust method for identifying  partial differential equations (PDEs) with varying coefficients using a weak form based group framework. 

In our formulation, B-spline bases are used to represent varying coefficients and as test functions. This choice of test functions improves the performances over other test functions in existing works based on the weak formulation. We utilize the GPSP algorithm to obtain a pool of candidate PDEs with different sparsity levels, which are further refined by GF-Trim. The optimal PDE is selected by the reduction in residual (RR) criterion.
The GF-Trim step eliminates features with low contribution scores, reducing the sensitivity of the RR with respect to hyper parameters selection.

Additionally, we demonstrate the effectiveness of our method through comprehensive experiments across various types of PDEs, and compare it with state-of-the-art methods. Our method consistently produces accurate and robust results. Ablation studies on the impact of different components of our method are also conducted to verify the robustness.  

\bibliographystyle{abbrv}
\bibliography{ref}
 
\newpage
\appendix
\section*{Appendix}

\section{B-splines with periodic boundary conditions} \label{Bspline_eqs}

To satisfy periodic boundary conditions, we extend the knot sequence by adding \( d \) knots to the left of the existing sequence, resulting in the new sequence \( z_{-d} < z_{-d+1} < \cdots < z_{-1} < z_0 < \cdots < z_G \). On these additional knots, we construct \( d \) additional B-spline functions \( \widetilde{B}_m^p \) for \( m = -d, -d+1, \ldots, -1 \), defined as follows:

\begin{equation}
\widetilde{B}_m^p(x) = 
\begin{cases}
    B_m^p(x) & \text{if } z_m \leq x < z_{m+1}, \\
    B_m^p\left(x - (z_G - z_0)\right) & \text{if } z_m + (z_G - z_0) \leq x < z_{m+1} + (z_G - z_0), \\
    0 & \text{otherwise}.
\end{cases}
\label{Periodic_extension_knots}
\end{equation}

This extension ensures that the B-spline basis functions seamlessly wrap around the domain, maintaining continuity and periodicity.

\section{More details on test function design} \label{splinespace_supplement}

In this section, we provide details on the design of the test functions \(\varphi\) used within the weak formulation. This includes computation of the critical frequencies $k_x^*,k_t^*$, and proofs of our theorems. 

\subsection{Test function construction and critical frequency computation}\label{Test_Fnc_cons}

Test function \(\varphi\) is constructed as the tensor product of univariate B-spline basis functions in the spatial and temporal domains:

\begin{equation}
\varphi_{r}(x, t) = B_{r_x}(x)B_{r_t}(t),
\end{equation}
where \( B_{r_x}(x) \) are the B-spline basis functions in the spatial domain, and \( B_{r_t}(t) \) are the B-spline basis functions in the temporal domain. This construction is illustrated in Figure~\ref{fig:testfnc2}(c). The spatial component \(B_{r_x}(x)\) satisfies periodic boundary conditions (Figure~\ref{fig:testfnc2}(a)). The temporal component \(\varphi^t(t)\) utilizes B-spline bases that vanish at endpoints (Figure~\ref{fig:testfnc2}(b)) in order to satisfy the zero Dirichlet boundary condition in time. The construction of \(B_{r_x}(x)\) and \(B_{r_t}(t)\) follows the same procedure. We focus on \(B_{r_x}(x)\) in the subsequent discussion. The same methodology applies to \(B_{r_t}(t)\).

Assume that the data \( U_i^n \) contains Gaussian noise, and the PDE solution is smooth. The Fourier transform of the noisy data along the spatial domain, \( \mathcal{F}_x(U_i^n) \), is expected to exhibit a distinct corner in its frequency spectrum. This corner represents the critical frequency \( k_x^* \) that separates the signal-dominant region from the noise-dominant region. To construct effective test functions, we identify this changepoint \( k_x^* \) by fitting a piecewise linear polynomial to the cumulative sum of the vectorized data \( |\mathcal{F}_x(U_i^n)| \). A similar analysis is conducted in the temporal domain to determine the corresponding changepoint \( k_t^* \). This method of identifying changepoints in the frequency domain to distinguish between signal and noise components has been successfully employed in other works, such as \cite{Messenger_2021, tang2022weakident}.

\subsection{Proof of Theorem \ref{prop:gaussian_bspline_moments}
}\label{proof_prop}

\begin{proof}[Proof of Theorem \ref{prop:gaussian_bspline_moments}]
For a Gaussian distribution $\rho(x)$, the zeroth moment is defined as:
$$
\int_{-\infty}^{\infty} \rho(x) dx = 1,
$$
which reflects the fact that the total area under the probability density function is 1. 
An equidistant knots B-spline $B^{p,h}(x)$ with support [$-\alpha_x$,$\alpha_x$] and knot spacing $h$ has the property that:
$$
\int_{-\alpha_x}^{\alpha_x}B^{p,h}(x) dx = \frac{2\alpha_x}{p}.
$$

To match the zeroth moment of the B-spline with the Gaussian function, we normalize this integral to 1 by scaling the B-spline. This normalization is $\frac{p}{2\alpha_x}$, making the zeroth moment of the B-spline equal to 1.

In the case of equidistant knots, we have the explicit formulas for the zero moment, first moment and second moment of the B-splines \cite{Neuman}
\begin{equation} \label{momentofBspline}
    \mu_0(p,\textbf{t})= 1, \ \ \  \mu_1(p,\textbf{t})= L_1+\frac{ph}{2}, \ \ \ \mu_2(p,\textbf{t})=\mu_1^2(p,\textbf{t})+\frac{ph^2}{12}.
\end{equation}
Here \(L_1=-\alpha_x\) and the space domain is defined in the range \([-\alpha_x,\alpha_x]\) which is symmetric around \(x=0\) and $\textbf{t} = (t_0,t_1,\ldots,t_p)$, where, $t_s = -\alpha_x+sh$, for arbitrary $s=1,2,\ldots, p$.

Next, consider the odd moments of both functions. Due to the symmetry of the Gaussian function and the equidistant knots B-spline around zero, all odd moments (first, third, fifth, etc.) are equal to zero.
It remains to show that both functions have the same second moment.

For the second moment (or variance) of the B-spline, it is given by:
$$
\mathrm{Var}_{\text{B-spline}} = \mu_2 = \int_{-\alpha_x}^{\alpha} x^2B^{p,h}(x) dx -\left( \int_{-\alpha_x}^{\alpha_x} xB^{p,h}(x) dx\right)^2.
$$
Since the first moment is zero, this simplifies to:
$$
\mathrm{Var}_{\text{B-spline}} = \int_{-\alpha_x}^{\alpha} x^2B^{p,h}(x) dx.
$$
According to (\ref{momentofBspline}), the second moment is:
$$
\mathrm{Var}_{\text{B-spline}} = \frac{ph^2}{12}.
$$
For the Gaussian function, the variance is $\sigma^2$. Given $\sigma = \frac{\sqrt{p}}{2\sqrt{3}}h$, the variance of the Gaussian function is:
$$
\sigma^2 = \left( \frac{\sqrt{p}}{2\sqrt{3}}h\right)^2 = \frac{ph^2}{12}.
$$

\end{proof}

\subsection{Proof of Theorem \ref{lemma:ft_difference_bound}} \label{proof_lemma1}
\begin{proof}[Proof of Theorem \ref{lemma:ft_difference_bound}]

We start with proving the statement when $p\geq 13$. Note that 
$$
B^{p,h}(x)=B^{p,1}(x/h).
$$
Let $g^{p,h}(x)$ be defined as in the theorem. Using the fact that 
$$
\widehat{B}^{p,1}(\omega)=\left(\frac{\sin  \omega}{ \omega}\right)^p,
$$
we have
\begin{align}
    \widehat{g}^{p,h}(\omega)=\left(\frac{\sin  h\omega}{ h\omega}\right)^p.
\end{align}
For the Gaussian function, according to Lemma \ref{thm.Gaussian_Fourier}, we have

\[
\widehat{\rho}(\omega) = \exp\left(-\frac{\omega^2 \sigma^2}{2}\right) = \exp\left(-\frac{p h^2 \omega^2}{24}\right).
\]

The distance between $\widehat{g}^{p,h}$ and $\widehat{\rho}$ can be measured by the following lemma:
\begin{lemma}[Lemma 2.1 of \cite{christensen2017bsplineapproximationsgaussiangabor}]\label{lemma.exp}
    Let $p\geq13$ be an integer. We have
    \begin{align}
        \left|\exp\left(-\frac{px^2}{6}\right)-\left(\frac{\sin x}{x}\right)^p\right|\leq
        \begin{cases}
            \frac{4}{5e^2p}\left(1+\frac{17\ln p}{7p}\right) = \frac{5}{5e^2p}(1+o(1)), & |x|<\sqrt{\frac{12\ln p}{p}},\\
            \frac{4(\ln p)^2}{5p^3}\left(1+\frac{17\ln p}{7p}\right)=\frac{4(\ln p)^2}{5p^3}(1+o(1)), & \sqrt{\frac{12\ln p}{p}}\leq|x|\leq\frac{\pi}{2}.
        \end{cases}
    \end{align}
\end{lemma}

Substituting $x$ by $\frac{h\omega}{2}$ in Lemma \ref{lemma.exp}, we deduce for $p\geq13$,

\begin{align}
    |\widehat{\rho}(\omega)-\widehat{g}^{p,h}(\omega)|=&\left|\exp\left(-\frac{p( h\omega)^2}{24}\right)-\left(\frac{\sin \frac{h\omega}{2}}{\frac{h\omega}{2}}\right)^p\right| \nonumber\\
    \leq&
        \begin{cases}
            \frac{4}{5e^2p}\left(1+\frac{17\ln p}{7p}\right), & |\omega| < \frac{2}{h}\sqrt{\frac{12 \ln p}{p}},\\
            \frac{4(\ln p)^2}{5p^3}\left(1+\frac{17\ln p}{7p}\right), & \frac{2}{h}\sqrt{\frac{12 \ln p}{p}} \leq |\omega| \leq \frac{\pi}{h}.
        \end{cases}
\end{align}
For $|\omega|>\frac{\pi}{h}$, clearly we have
\begin{align}
    |\widehat{\rho}(\omega)-\widehat{g}^{p,h}(\omega)|\leq \exp\left(-\frac{p(h\omega)^2}{24}\right)+\left(\frac{2}{ h\omega}\right)^p.
\end{align}

Property
\[
\lim_{p \to \infty} \left( \widehat{\rho}(\omega)-\widehat{g}^{p,h}(\omega)  \right) = 0 \quad \text{for all } \omega \in \mathbb{R}
\]
is a direct result of the upper bounds above.

When $p<13$, note that \(\frac{2}{h}\sqrt{\frac{12 \ln p}{p}}>\frac{\pi}{h}\), the upper bounds follow the result of $p\geq 13$.

\end{proof}

\subsection{Proof of Theorem \ref{thm.alphax}} \label{proof_thm}

To approximate a B-spline basis function by a Gaussian distribution with carefully chosen parameters, we properly select the standard deviation $\sigma$ of the Gaussian distribution and the support range of the B-spline basis function, denoted by $[-\alpha_x, \alpha_x]$ with $\alpha_x > 0$. By doing so, we can ensure that the B-spline and the Gaussian distribution share the same first four moments.

\begin{proof}[Proof of Theorem \ref{thm.alphax}]

We begin with the relationship between the changepoint in Fourier space and the standard deviation $\sigma$ of the Gaussian function in Fourier space. From Lemma \ref{thm.Gaussian_Fourier}, the Fourier transform of the Gaussian function $\rho(x)$ is given by:
\[
\widehat{\rho}(\omega) = \exp\left(-\frac{\omega^2 \sigma^2}{2}\right).
\]
The Fourier transform $\widehat{\rho}(\omega)$ is also a Gaussian function with standard deviation $\sigma_\omega = \frac{1}{\sigma}$ (up to scaling factors).

We aim to choose $\sigma$ such that the frequency content of $\widehat{\rho}(\omega)$ (and thus the approximated B-spline basis function) is concentrated within the signal-dominant region $(-k_x^*, k_x^*)$. Specifically, we want the changepoint $k_x^*$ to be approximately $\tau_x$ times standard deviations into the tail of $\widehat{\rho}(\omega)$:
\[
k_x^* = \tau_x \sigma_\omega = \tau_x \left(\frac{1}{\sigma}\right).
\]
Solving for $\sigma$, we get:
\[
\sigma = \frac{\tau_x}{k_x^*}.
\]
Using the relationship between $\sigma$ and $\alpha_x$ from Theorem \ref{prop:gaussian_bspline_moments}, $\sigma = \frac{\alpha_x}{\sqrt{3p}}$, we substitute:
\[
\frac{\alpha_x}{\sqrt{3p}} = \frac{\tau_x}{k_x^*}.
\]
Solving for $\alpha_x$, we obtain:
\[
\alpha_x = \frac{\sqrt{3p} \, \tau_x}{k_x^*}.
\]

\end{proof}

\section{Group Subspace Pursuit Algorithm} \label{appendixalg}
The Group Subspace Pursuit (GPSP) algorithm proposed in \cite{he2023group} is an iterative method designed to solve group-sparsity-constrained optimization problems by selecting the most significant feature groups. In each iteration, it updates the residuals and the set of selected indices to efficiently estimate the coefficients corresponding to the chosen features. Its convergence is analyzed in \cite{he2024groupprojectedsubspacepursuit}. The algorithm is summarized in Algorithm \ref{SPalgo}. For here, we use the tensor product between spatial and temporal test functions as the complete test functions. Therefore, the size of feature matrix becomes $\mathbf{F}\in\mathbb{R}^{S\times KM}$ and the size of the time derivatives becomes $\mathbf{b}\in\mathbb{R}^{S\times 1}$, where $S$ is the total number of test functions.

\begin{algorithm}[t!]
	\KwIn{$M$: number of basis function for varying coefficients.\\

		$\mathbf{F}\in\mathbb{R}^{S\times KM}$: the approximated feature matrix.\\
		$\mathbf{b}\in\mathbb{R}^{S}$: the approximated time derivatives. \\
		$\theta\in\mathbb{N}$, $\theta\geq 1$: the sparsity level.\\
	}
	
	\textbf{Initialization: } Denote $\mathbf{F}_{\{l\}}$ the column(s) of $\mathbf{F}$ corresponding to the $l^{th}$ feature, $l=1,...,K$.\\
	\For{$i=1:K$}{ $\mathbf{g}^{0,i}=\mathbf{F}_{\{i\}}\mathbf{F}_{\{i\}}^{\dagger}\mathbf{b}$.\\
		$\mathbf{g}^{0,i}=\mathbf{g}^{0,i}/\|\mathbf{g}^{0,i}\|_2.$
	}
	Set $\mathbf{G}^0=[\mathbf{g}^{0,1},\mathbf{g}^{0,2},...,\mathbf{g}^{0,K}]$.
	Let $j=0$, $\mathcal{I}^0=\{\theta$ indices corresponding to the largest magnitude entries in the vector $\mathbf{G}^*\mathbf{b}\}$, and $\mathbf{b}_{\text{res}}^{0}= \mathbf{b}-\mathbf{F}_{\{\mathcal{I}^0\}}\mathbf{F}_{\{\mathcal{I}^0\}}^\dagger\mathbf{b}$\;
	\textbf{Iteration: }\\
	\While{True}{
		\textbf{Step 1.}\\
		
		\For{$i=1:K$}{ $\mathbf{g}^{j+1,i}=\mathbf{F}_{\{i\}}\mathbf{F}_{\{i\}}^{\dagger}\mathbf{b}_{\text{res}}^j$.\\
			$\mathbf{g}^{j+1,i}=\mathbf{g}^{j+1,i}/\|\mathbf{g}^{j+1,i}\|_2.$
		}
		Set $\mathbf{G}^{j+1}=[\mathbf{g}^{j+1,1},\mathbf{g}^{j+1,2},...,\mathbf{g}^{j+1,K}]$. Construct $\widetilde{\mathcal{I}}^{j+1}=\mathcal{I}^{j}\cup\{\theta$ indices corresponding to the largest magnitude entries in the vector $(\mathbf{G}^{j+1})^*\mathbf{b}_{\text{res}}^{j}\}$.

		\textbf{Step 2.}\\
		$\mathbf{v}=\mathbf{0}\in\mathds{R}^K, \widetilde{\mathbf{c}}_{\{\mathcal{I}^{j+1}\}}=\mathbf{0}\in \mathds{R}^{K+(r-1)s}$. $\widetilde{\mathbf{c}}_{\{\mathcal{I}^{j+1}\}}=\mathbf{F}_{\{\mathcal{I}^{j+1}\}}^{\dagger}\mathbf{b}$.\\
		\For{$i$ in $\widetilde{\mathcal{I}}^{j+1}$}{ $\mathbf{v}^{j+1}_i=\|\mathbf{F}_{\{i\}}\widetilde{\mathbf{c}}_{\{i\}}\|_2$.\\
		}
		Let $\mathcal{I}^{j+1}=\{\theta$ indices corresponding to the largest elements of $\mathbf{v}\}$, and compute $\mathbf{b}_{\text{res}}^{j+1}=\mathbf{b}-\mathbf{F}_{\{\mathcal{I}^{j+1}\}}\mathbf{F}_{\{\mathcal{I}^{j+1}\}}^\dagger \mathbf{b}$\;
		
		\textbf{Step 3.} \\
		If $\|\mathbf{b}_{\text{res}}^{l+1}||_2>||\mathbf{b}_{\text{res}}^{l}||_2$, let $\widehat{\mathcal{I}} = \mathcal{I}^j$, output $\widehat{\mathbf{c}}$ such that $\widehat{\mathbf{c}}_{\{\mathcal{I}^{j}\}^{C}}=\mathbf{0}$ and $\widehat{\mathbf{c}}_{\{\mathcal{I}^j\}}=\mathbf{F}_{\{\mathcal{I}^j\}}^{\dagger}\mathbf{b}$;  and terminate the algorithm. Otherwise let $j\gets j+1$, and iterate.
	}
	
	\KwOut{A set of indices $\widehat{\mathcal{I}}$; $\widehat{\mathbf{c}}\in\mathbb{R}^{K\times 1}$ with $\|\widehat{\mathbf{c}}\|_0=\theta$ satisfying $\widehat{\mathbf{c}}_{\{\widehat{\mathcal{I}}\}}=\mathbf{F}_{\{\widehat{\mathcal{I}}\}}^\dagger \mathbf{b}$}
	\caption{{\bf Group Subspace Pursuit Algorithm($\text{GPSP}(\theta;M,\mathbf{F},\mathbf{b})$)} \label{SPalgo}}
\end{algorithm}

\section{Additional numerical studies} \label{appendix_comparisonmtd}
\subsection{More comparisons with the state-of-the-art methods}
We presents in Table \ref{tab:AdvectionKdVKS_eq} a detailed comparison of WG-IDENT (the proposed method) with GLASSO, SGTR, and rSGTR on the advection-diffusion equation, KdV equation, and KS equation, as a complement for the comparison in Section \ref{sec:comparison_study}. In this comparison, the dictionary consists of 16 features, incorporating partial derivatives of $u$, $u^2$ and $u^3$ up to the fourth order. The results demonstrate that our method consistently identifies the correct features across various noise levels, highlighting its robustness compared to other methods.

\begin{table}[t!]
\centering
{\centering \bfseries (a) Advection-Diffusion Equation \par}
\begin{tabular}{ccccc}
\toprule
Noise level & WG-IDENT & GLASSO & SGTR & rSGTR \\
\midrule
No Noise & $u_x, u_{xx}$ & $u_x, u_{xx}$ & $u_x, u_{xx}$ & $u_x, u_{xx}$ \\
0.1\%  & $u_x, u_{xx}$ & $u_x, u_{xx}$ & $u_x, u_{xx}$ & $u_x, u_{xx}$ \\
0.5\%  & $u_x, u_{xx}$ & $u_x, u_{xx}$ & $u_x, u_{xx}$ & $u_x, u_{xx}$ \\
1\%    & $u_x, u_{xx}$ & $u_x, u_{xx}$ & $u_x, u_{xx}$ & $u_x, u_{xx}$ \\
5\%    & $u_x, u_{xx}$ & $u_x, u_{xx}$ & $u_x, u_{xx}$ & $u_x, u_{xx}$  \\
8\%    & $u_x, u_{xx}$ & $\geq 4\ \text{terms}$ & $u_x, u_{xxxx}$ & $u_x, u_{xxxx}$  \\
10\%   & $u_x, u_{xx}$ & $\geq 4\ \text{terms}$ & $\geq 4\ \text{terms}$ & $\geq 4\ \text{terms}$  \\
\bottomrule
\end{tabular}
\vspace{0.5cm}

{\centering \bfseries (b) KdV Equation \par}
\begin{tabular}{ccccc}
\toprule
Noise level & WG-IDENT & GLASSO & SGTR & rSGTR \\
\midrule
No Noise & $uu_x, u_{xxx}$ & $3\ \text{terms}$ & $uu_x, u_{xxx}$ & $uu_x, u_{xxx}$ \\
0.1\%  & $uu_x, u_{xxx}$ & $3\ \text{terms}$ & $uu_x, u_{xxx}$ & $uu_x, u_{xxx}$ \\
0.5\%  & $uu_x, u_{xxx}$ & $3\ \text{terms}$ & $\geq 4\ \text{terms}$ & $\geq 4\ \text{terms}$ \\
1\%    & $uu_x, u_{xxx}$ & $3\ \text{terms}$ & $\geq 4\ \text{terms}$ & $\geq 4\ \text{terms}$ \\
5\%    & $uu_x, u_{xxx}$ & $3\ \text{terms}$ & $\geq 4\ \text{terms}$ & $\geq 4\ \text{terms}$  \\
8\%    & $uu_x, u_{xxx}$ & $3\ \text{terms}$ & $\geq 4\ \text{terms}$ & $\geq 4\ \text{terms}$  \\
10\%   & $uu_x, u_{xxx}$ & $4\ \text{terms}$ & $\geq 4\ \text{terms}$ & $\geq 4\ \text{terms}$  \\
\bottomrule
\end{tabular}
\vspace{0.5cm}

{\centering \bfseries (c) KS Equation \par}
\begin{tabular}{ccccc}
\toprule
Noise level & WG-IDENT & GLASSO & SGTR & rSGTR \\
\midrule
No Noise & $uu_x, u_{xx}, u_{xxxx}$ & $\geq 4\ \text{terms}$ & $uu_x, u_{xx}, u_{xxxx}$ & $uu_x, u_{xx}, u_{xxxx}$ \\
0.1\%  & $uu_x, u_{xx}, u_{xxxx}$ & $\geq 4\ \text{terms}$ & $\geq 4\ \text{terms}$ & $\geq 4\ \text{terms}$ \\
0.5\%  & $uu_x, u_{xx}, u_{xxxx}$ & $\geq 4\ \text{terms}$ & $\geq 4\ \text{terms}$ & $\geq 4\ \text{terms}$ \\
1\%    & $uu_x, u_{xx}, u_{xxxx}$ & $\geq 4\ \text{terms}$ & $\geq 4\ \text{terms}$ & $\geq 4\ \text{terms}$ \\
5\%    & $uu_x, u_{xx}, u_{xxxx}$ & $\geq 4\ \text{terms}$ & $\geq 4\ \text{terms}$ & $\geq 4\ \text{terms}$  \\
8\%    & $uu_x, u_{xx}, u_{xxxx}$ & $\geq 4\ \text{terms}$ & $\geq 4\ \text{terms}$ & $\geq 4\ \text{terms}$  \\
10\%   & $uu_x, u_{xx}, u_{xxxx}$ & $\geq 4\ \text{terms}$ & $\geq 4\ \text{terms}$ & $\geq 4\ \text{terms}$  \\
\bottomrule
\end{tabular}
\caption{Comparison of results for the rest of the PDEs: (a) advection-diffusion equation, (b) KdV equation, and (c) KS equation. The table presents the identified features by our method, GLASSO, SGTR, and rSGTR across various noise levels. Our method consistently identifies the correct features at all noise levels.}
\label{tab:AdvectionKdVKS_eq}
\end{table}

\subsection{Comparison of weak formulation in space-time domain and only in space}
In our formulation, the weak formulation is computed by integrating over the space-time domain, so that both spatial and temporal derivatives are transferred to test functions, see (\ref{eq_model_general_weak}). Another strategy is to only utilize the weak formulation in space, and use numerical differentiation to calculate time derivative with the the successively denoised differentiation (SDD) \cite{he2022robust} method. In this subsection, we compare these two strategies and demonstrate the advantages of using space-time weak formulation.

In the presence of noisy data, numerical differentiation can amplify noise, leading to a substantial amplification of errors in the time derivative. To address this issue in the second strategy, we adopt SDD, which computes the time derivative as
$$
\partial_t u_i^n \approx S_{(t)}D_tS_{(\textbf{x})}[U_i^n],
$$
where $S_{(t)},S_{(x)}$ are smoothing operators in time and space, and $D_t$ is a numerical differentiation scheme. By setting $S(t)$ as 
$$
S_{(t)}[U_\textbf{i}^n] = p_\textbf{i}^n(t^n), \text{ where } p_\textbf{i}^n = \argmin_{p\in P_2}\sum_{0\leq k\leq N}(p((t^k)-U_\textbf{i}^n))^2\exp\left (-\frac{||t^n-t^k||^2}{\gamma^2}\right ),
$$ 
where $P_2$ is the set of polynomials with a degree not exceeding 2, and $\gamma$ is a paramter controlling the weight. The operator $S_{(x)}$ can be defined similarly. For $D_t$, we employ a 4-point central difference scheme given by
$$
D_t U_i^n=\frac{-U_i^{n+2}+8U_i^{n+1}-8U_i^{n-1}+U_i^{n-2}}{12\Delta t}.
$$

We compare our method with the two strategies for identifying viscous Burgers' equation and KS equation. The $E_2$ values of the results are shown in Figure \ref{fig:SDD_phix_N_phixt}. For both PDE, the approach using SDD in the time domain and weak formulation in the spatial domain results in higher $E_2$ values. While using weak formulation in space-time domain gives constantly smaller $E_2$, demonstrating the benefits of this strategy.

\begin{figure}[t!]
    \centering
    \begin{tabular}{cc|cc}
    \hline
    \multicolumn{2}{c|}{viscous Burgers' equation} & \multicolumn{2}{c}{KS equation}\\
    \hline
    strategy 1 & strategy 2 & strategy 1 & strategy 2\\
    \includegraphics[width=0.2\linewidth]{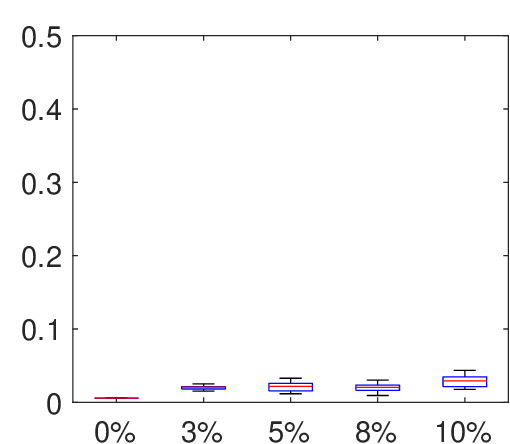}&
    \includegraphics[width=0.2\linewidth]{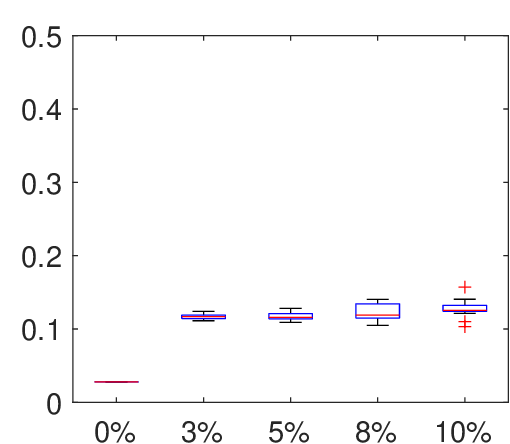}&
    \includegraphics[width=0.2\linewidth]{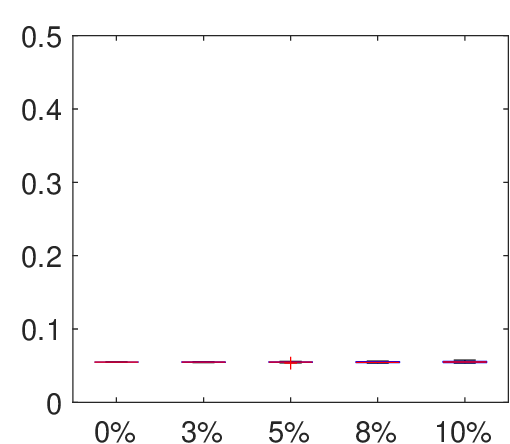}&
    \includegraphics[width=0.2\linewidth]{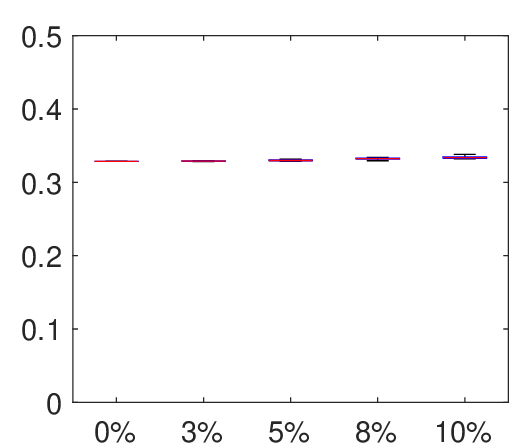}\\
    \hline 
    \end{tabular}
    \caption{Comparison of weak formulation in space-time domain 
    (strategy 1) and only in space (strategy 2). The plots show the $E_2$ value ($y$-axis) of different strategies in identifying viscous Burgers' equation and KS equation with various noise levels ($x$-axis). In strategy 1, weak formulation is used in space-time domain. In strategy 2, weak formulation is only used in space.}
    \label{fig:SDD_phix_N_phixt}
\end{figure}

\subsection{Sensitivity to the number of B-spline basis for approximating coefficients }\label{Appendix:numberofbasis}
In our proposed method, the varying coefficients are approximated via a linear combination of B-spline basis functions. We investigate how the number of basis functions affects identification accuracy. We apply this methodology to both the KS and KdV equations with varying numbers of basis functions. The true positive rate (TPR) for each experiment is reported in Figure \ref{fig:numb_basis_KS_NLS}. For the KS equation, when the number of basis function ranges from 2 to 46, the underlying PDE can be accurately identified by using 2-38 bases.

For KdV equation, our method using 3-12 basis functions successfully identifies the underlying PDE.
These results demonstrate that our method maintains robustness to the number of basis functions within a wide range.

\begin{figure}[t!]
    \centering
    \begin{tabular}{c c c}
    
    \hline
        \multicolumn{1}{c}{KS equation} & \multicolumn{1}{c}{KdV equation}\\
    \hline
        {\includegraphics[width=0.45\linewidth]{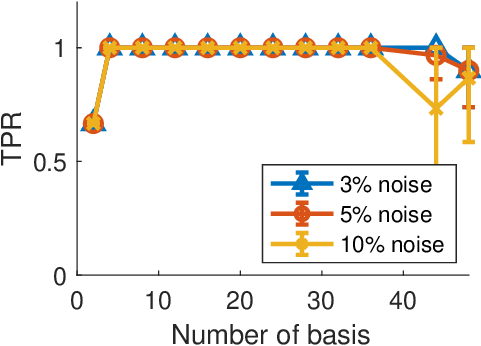}} & 
        {\includegraphics[width=0.45\linewidth]{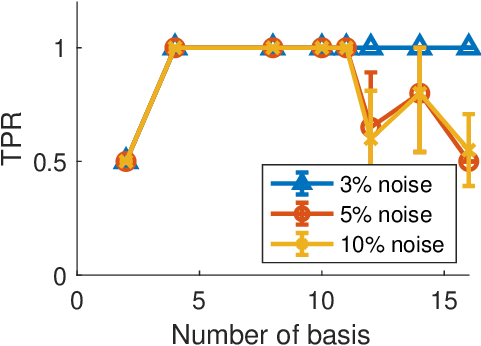}}\\
          (a)& (b) \\
    \hline 
    \end{tabular}
    \caption{With different numbers of bases functions for varying coefficients, identification accuracy of feature groups under varying noise levels (3\%, 5\%, 10\%) for Kuramoto-Sivashinsky (KS) and Korteweg-de Vries (KdV) equations, evaluated using True Positive Rate (TPR). The proposed method demonstrates robust performance across a wide range of basis function counts: (a) the KS equation consistently achieves 100\% TPR when using between 3 and 38 basis functions, and (b) the KdV equation consistently maintains 100\% TPR when using between 3 and 12 basis functions.}
\label{fig:numb_basis_KS_NLS}

\end{figure}

\subsection{Sensitivity of the Spatial Tuning Parameter $\tau_x$ and B-spline Order $p$ }
\label{app:tau_sensitivity}

 We investigate the effect of the spatial spectral parameter $\tau_x$,---which sets the spatial support of the B-spline test functions via the Gaussian–moment heuristic---affects identification performance for the advection–diffusion equation
\begin{equation}
  u_t \;=\; a(x)\,u_x \;+\; c\,u_{xx}, 
  \qquad a(x)=3\big(\sin(2\pi x)+3\big),\quad c=0.2,
\end{equation}
on $[0,2)\times[0,0.05]$ with periodic boundary conditions in space, the same setting as in Table~\ref{tab:differential_equations}. In our framework, the spatial test-function support $\alpha_x$ is chosen so that the Fourier cutoff $k_x^\ast$ lies $\tau_x$ standard deviations into the (approximately Gaussian) spectral tail of the B-spline kernel as shown in equation~\eqref{eq_alpha} (see Sec.~4.2).

For the following experiment, Fig.~\ref{fig:advectiondiff_tpr_taux}, the noise is added to the observations using the same noise-to-signal model as the main text at four levels $\sigma_{\mathrm{NSR}}\in\{0\%,3\%,5\%,10\%\}$. For each noise level and each $\tau_x \in \{3.0,3.3,3.5,3.7,4.0\}$, we perform 50 independent trials. Identification uses WG-IDENT with GPSP for candidate generation, GF-Trim for feature refinement, and RR for model selection. Performance is summarized by the \emph{group-level True Positive Rate (TPR)} for recovering the two ground-truth operators $\{u_x,\ u_{xx}\}$.

For each noise level, we report the distribution of TPR across 50 trials as a function of $\tau_x$. Varying $\tau_x$ adjusts the spatial support of the test functions (and thus the effective low-pass filtering). Smaller $\tau_x$ produces narrower supports that admit higher frequencies and can overfit high-frequency noise; larger $\tau_x$ increases support and suppresses noise more aggressively but risks attenuating informative mid-frequency content. This experiment quantifies the robustness window of $\tau_x$ under realistic noise levels while keeping all other design choices fixed. Across all tested values of $\tau_x \in \{3.0,3.3,3.5,3.7,4.0\}$ and noise levels (0\%, 3\%, 5\%, 10\%), the identification procedure consistently achieved a true positive rate (TPR) of 1. This consistent success indicates that the method is  robust to the choice of spectral support parameter $\tau_x$, even under nontrivial noise perturbations.

Similarly, we examine the effect of the spatial test-function order $p$ (B-spline order) on identification accuracy for the advection--diffusion equation with periodic boundaries. For each noise level $\{0\%, 3\%, 5\%, 10\%\}$, we vary $p \in \{4,5,6,7,8\}$. For every pair $(\text{noise},p)$, we perform 50 independent tests. The identification results are illustrated in Fig.~\ref{fig:advectiondiff_tpr_p}. Across all tested spline orders $p \in \{4,5,6,7,8\}$ and noise levels (0\%, 3\%, 5\%, 10\%), the identification procedure consistently achieved a true positive rate (TPR) of 1. This uniform recovery performance highlights the robustness of our weak formulation and model selection strategy, demonstrating its insensitivity to spline order choice even under substantial noise.

\begin{figure}[t!]
    \centering
    \begin{tabular}{cccc}
    \hline
        \textbf{0\%} & \textbf{3\%} & \textbf{5\%} & \textbf{10\%} \\
    \hline
    \multicolumn{4}{c}{Advection diffusion equation - TPR \& $\tau_x$}\\\hline
        \includegraphics[width=0.2\linewidth]{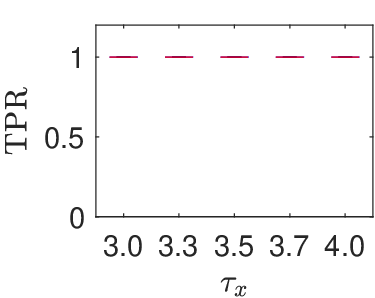} &
        \includegraphics[width=0.2\linewidth]{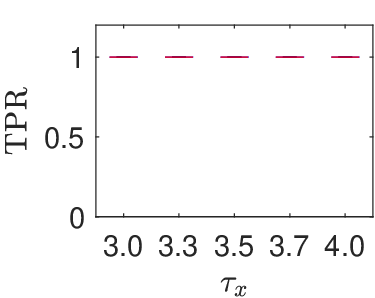}&
        \includegraphics[width=0.2\linewidth]{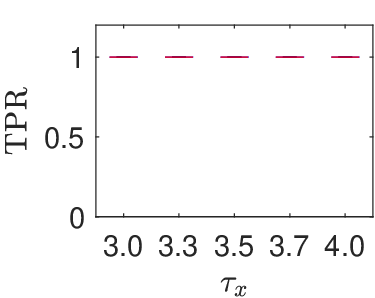}&
        \includegraphics[width=0.2\linewidth]{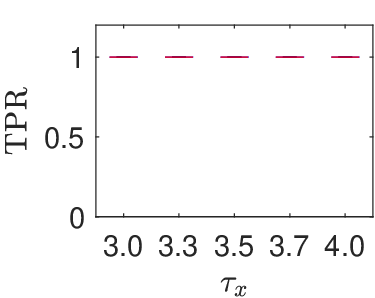}  \\

    \hline

    \end{tabular}
    \caption{True positive rate (TPR) of term identification across $\tau_x \in \{3.0,3.3,3.5,3.7,4.0\}$, shown as boxplots over 50 independent experiments per setting. Columns correspond to additive Gaussian noise levels of 0\%, 3\%, 5\%, and 10\%. In all cases, TPR = 1, confirming that our weak formulation and selection procedure recovers the correct PDE terms uniformly across spectral parameter choices and noise levels.}

    \label{fig:advectiondiff_tpr_taux}
\end{figure}

\begin{figure}[t!]
    \centering
    \begin{tabular}{cccc}
    \hline
        \textbf{0\%} & \textbf{3\%} & \textbf{5\%} & \textbf{10\%} \\
    \hline
    \multicolumn{4}{c}{Advection diffusion equation - TPR \& $p$}\\\hline
        \includegraphics[width=0.2\linewidth]{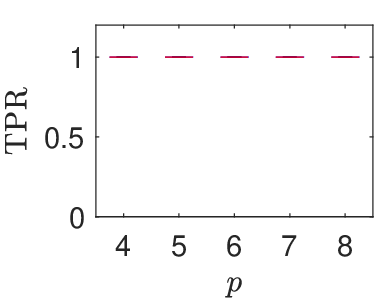} &
        \includegraphics[width=0.2\linewidth]{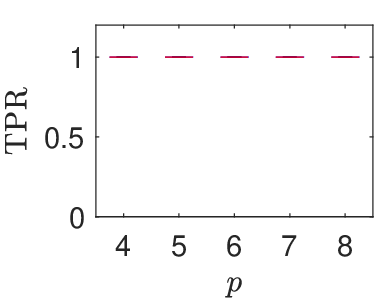}&
        \includegraphics[width=0.2\linewidth]{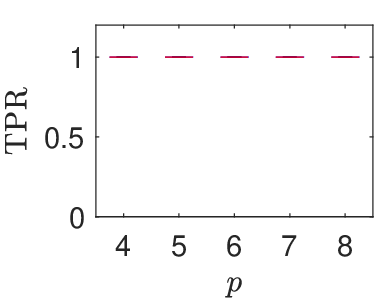}&
        \includegraphics[width=0.2\linewidth]{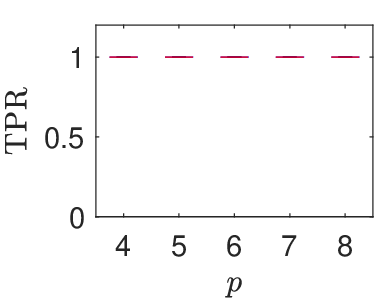}  \\

    \hline

    \end{tabular}
    \caption{True positive rate (TPR) of term identification across B-spline order $p \in \{4,5,6,7,8\}$, shown as boxplots over 50 independent experiments per setting. Columns correspond to additive Gaussian noise levels of 0\%, 3\%, 5\%, and 10\%. In all cases, TPR = 1, confirming that our weak formulation and selection procedure recovers the correct PDE terms uniformly across spline orders and noise levels.}

    \label{fig:advectiondiff_tpr_p}
\end{figure}

\subsection{Inviscid Burgers' equation with a discontinuity}\label{Appendixd4_inviscid}
    Here we provide further details on the additional experiment with the inviscid Burgers’ equation exhibiting shock formation. Specifically, we consider the one-dimensional inviscid Burgers’ equation with a spatially varying advective coefficient:
    \begin{equation}
        u_t = a(x)uu_x,
        \qquad x \in [0,2], \; t \in [0,0.1],
    \end{equation}
    subject to periodic boundary condition on space. 
    The varying coefficient is  chosen to go beyond the first-degree trigonometric functions used in earlier examples. In particular, we take 

    \begin{equation}
        a(x) = e^{\sin(\pi x)},
    \end{equation}
    which incorporates contributions from all harmonic orders in its Fourier expansion. 

    The initial condition is a smooth trigonometric combination with strong variations, which naturally generates shock formation during the temporal evolution:
    \begin{equation}
        u(x,0) = u_0(x) = \Big[\,0.2 + 1.2\sin(2\pi x) + 0.7\sin(6\pi x) + 0.35\sin(12\pi x)\,\Big] - \overline{u_0},
    \end{equation}
    where the spatial mean $\overline{u_0} = \frac{1}{|\Omega|} \int_{\Omega} u_0(x)\,dx$.

    For the numerical discretization, the spatial domain $[0,2]$ is partitioned into 
    \(
    N = 400 \quad \text{grid points}, \quad \Delta x = \tfrac{2}{N},
    \)
    and the temporal domain $[0,0.1]$ is discretized with 
    \(
    \Delta t = 10^{-4},  T = 0.1, \text{so that } N_t = 1000 \text{ steps}.
    \)

\begin{figure}[t!]
    \centering
    \begin{tabular}{cccc}
    \hline
        \textbf{E\textsubscript{2}} & \textbf{E\textsubscript{res}} & \textbf{TPR} & \textbf{PPV} \\
    \hline
    \multicolumn{4}{c}{Inviscid Burgers' equation}\\\hline
        \includegraphics[width=0.2\linewidth]{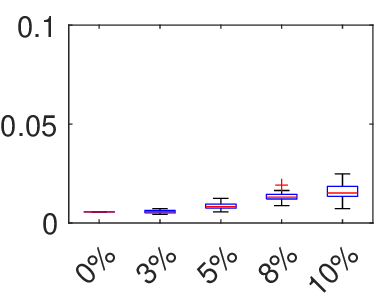} &
        \includegraphics[width=0.2\linewidth]{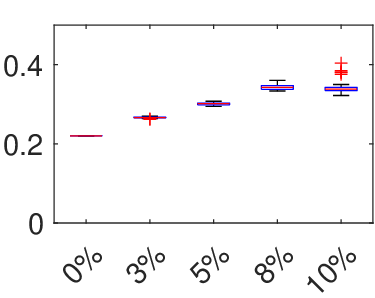}&
        \includegraphics[width=0.2\linewidth]{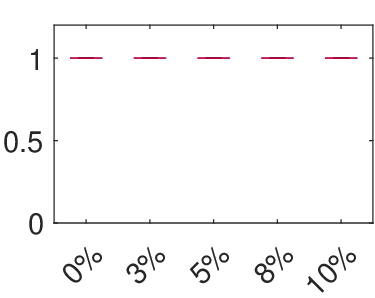}&
        \includegraphics[width=0.2\linewidth]{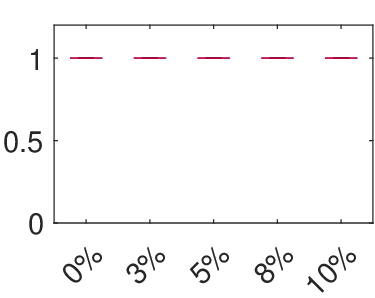} \\
    \hline
    \end{tabular}
    \caption{$E_2$, $E_{\text{res}}$, TPR, PPV of inviscid Burgers' equation with different noise levels, $0\%$, $3\%$, $5\%$, $8\%$, $10\%$. Each experiment is repeated for 50 times.}
    \label{fig:inviscid_Burgers_equation_with_a_discontinuity}
    \end{figure}

\begin{figure}[t!]
    \centering
    \begin{tabular}{cccc}
    \hline
    \multicolumn{4}{c}{Inviscid Burgers' equation}\\\hline
        $\sigma_{\text{NSR}}=0\%$ & 3\% & 5\% & 10\% \\
    \hline
    \multicolumn{4}{c}{Given data}\\
    {\includegraphics[width=0.18\linewidth]{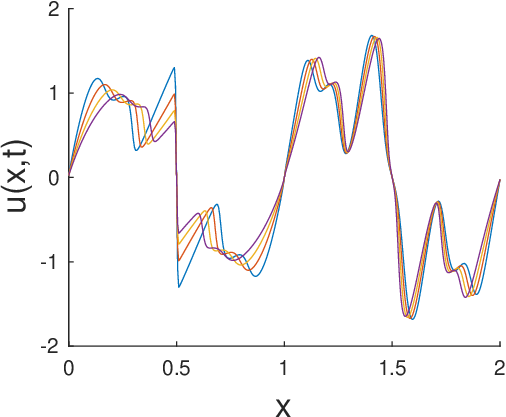}} &
        {\includegraphics[width=0.18\linewidth]{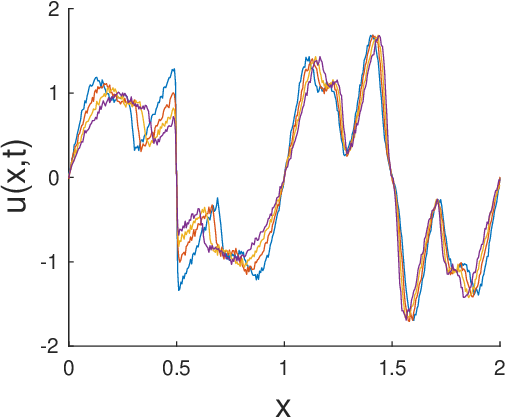}} &
        {\includegraphics[width=0.18\linewidth]{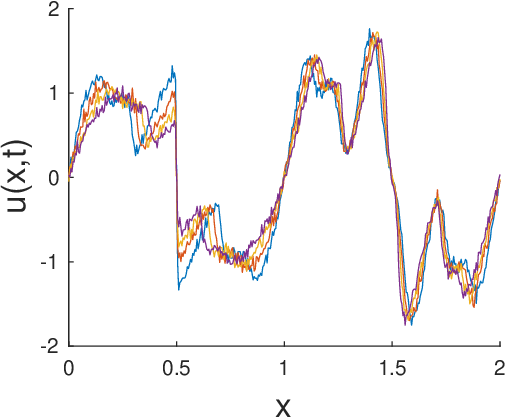}}&
        {\includegraphics[width=0.18\linewidth]{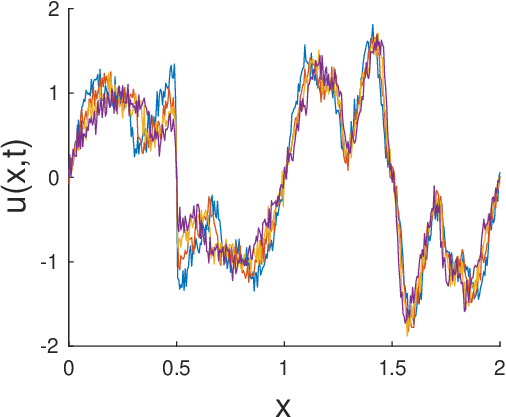}}\\
        \hline
        \multicolumn{4}{c}{Reconstructed varying coefficient (blue = true, red = identified)}\\
        {\includegraphics[width=0.18\linewidth]{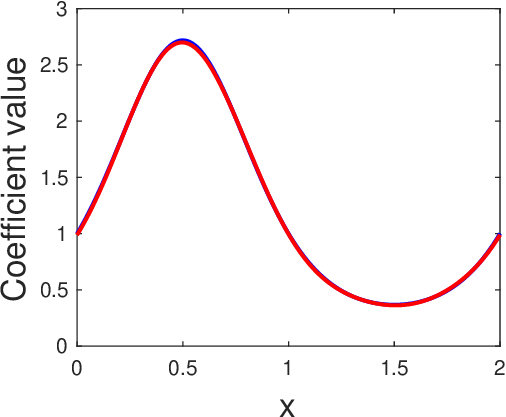}} &
        {\includegraphics[width=0.18\linewidth]{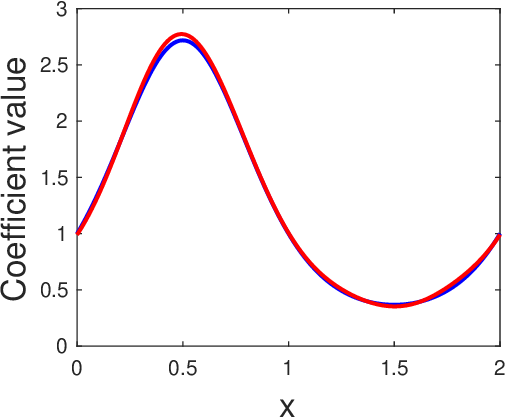}} &
        {\includegraphics[width=0.18\linewidth]{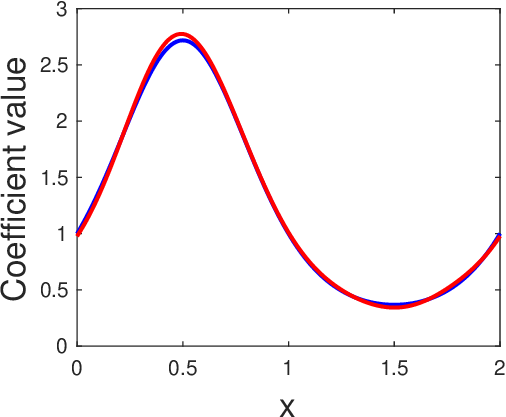}}&
        {\includegraphics[width=0.18\linewidth]{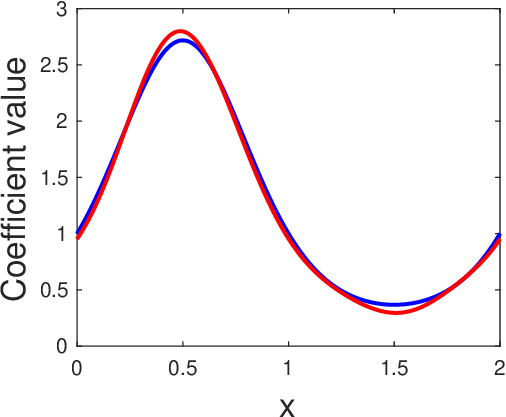}}\\
    \hline
       \multicolumn{4}{c}{Simulated solution from the identified PDE}\\
        {\includegraphics[width=0.18\linewidth]{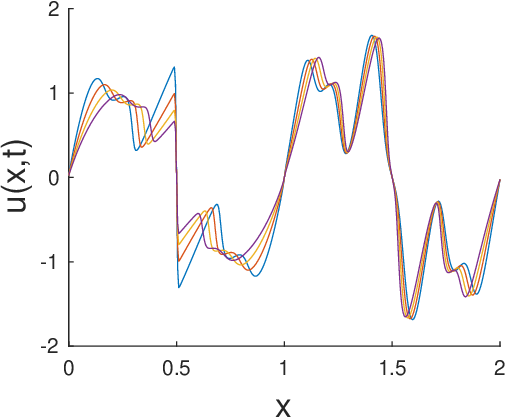}} &
        {\includegraphics[width=0.18\linewidth]{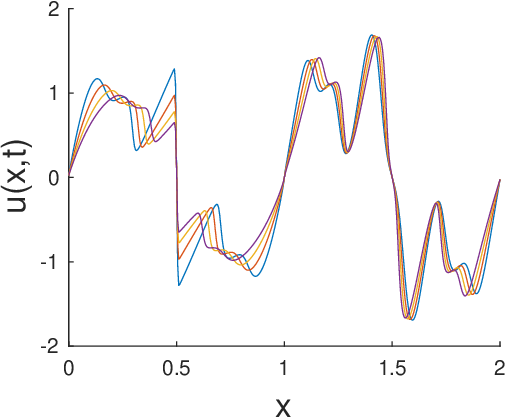}} &
        {\includegraphics[width=0.18\linewidth]{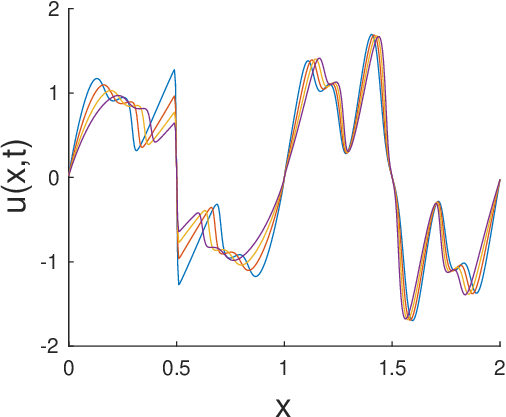}}&
        {\includegraphics[width=0.18\linewidth]{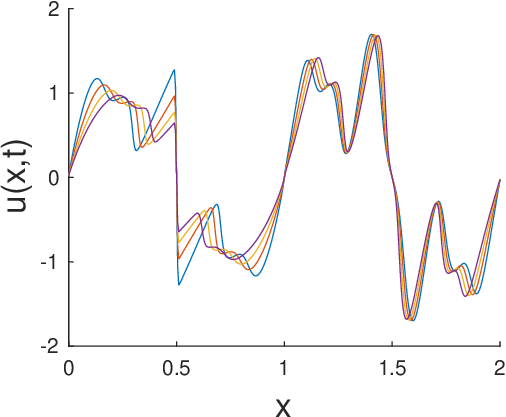}}\\
    \hline
    \multicolumn{4}{c}{Spatial profile at $t=0.06$ (blue = true, red = identified)}\\
    {\includegraphics[width=0.18\linewidth]{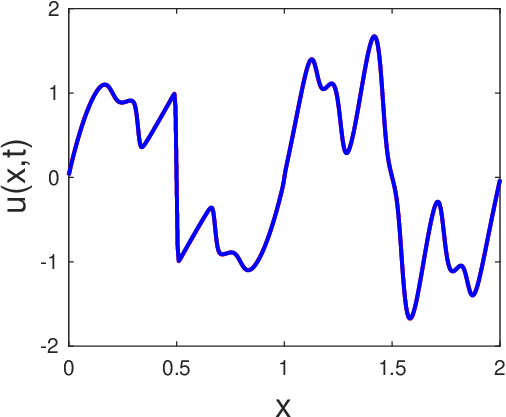}} &
    {\includegraphics[width=0.18\linewidth]{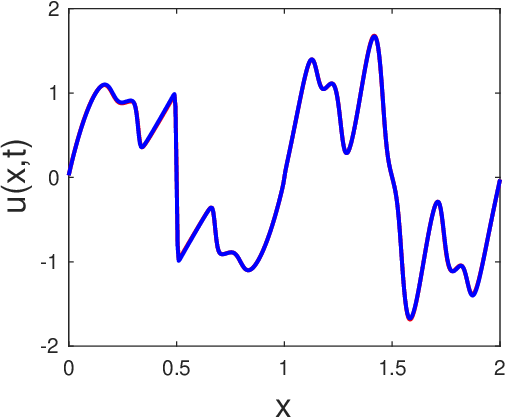}} &
    {\includegraphics[width=0.18\linewidth]{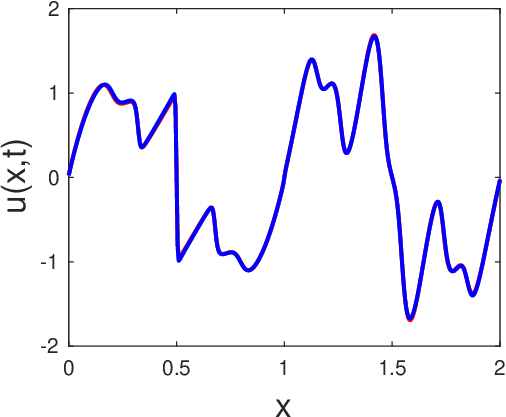}}&
    {\includegraphics[width=0.18\linewidth]{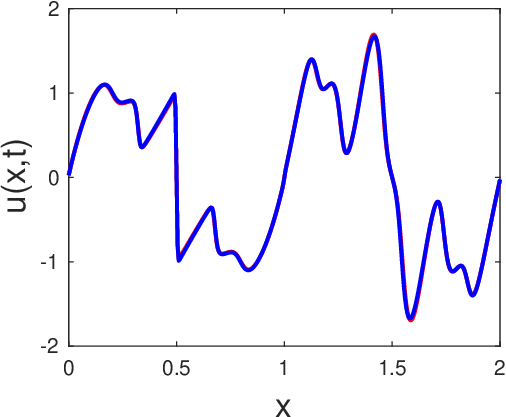}}\\
    \hline
    \multicolumn{4}{c}{Absolute error}\\
        {\includegraphics[width=0.18\linewidth]{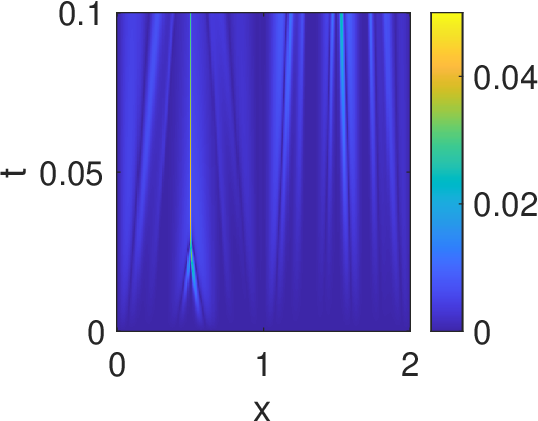}} &
        {\includegraphics[width=0.18\linewidth]{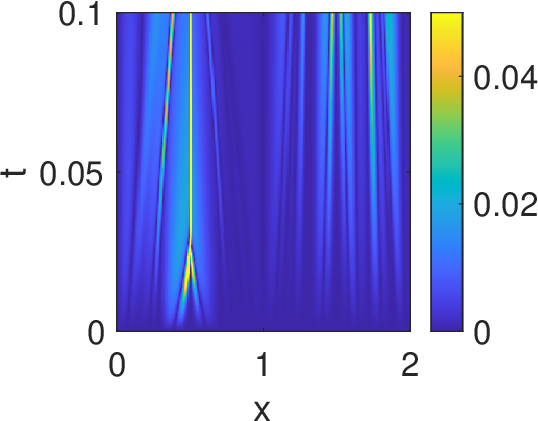}} &
        {\includegraphics[width=0.18\linewidth]{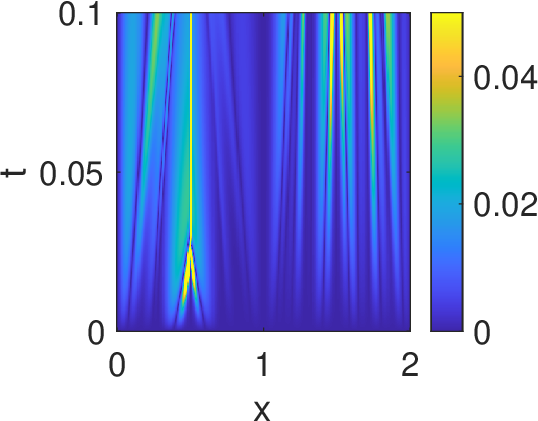}}&
        {\includegraphics[width=0.18\linewidth]{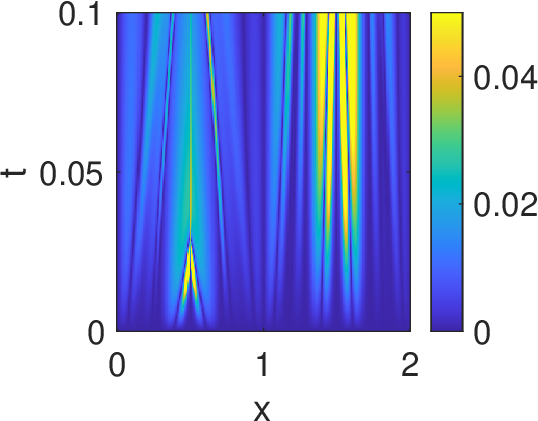}}\\
    \hline
    \end{tabular}
    \caption{Inviscid Burgers' equation under increasing observation noise.
Columns correspond to $\sigma_{\text{NSR}}=0\%, 3\%, 5\%, 10\%$.
Row~1 shows the given (noisy) observations as spatial trajectories at time frames $t\in\{0.04, 0.06, 0.08, 0.1\}$; 
Row~2 shows the reconstructed varying coefficient overlaying the true (blue) and identified (red) trajectories;
Row~3 shows trajectories at the same time frames simulated from the PDE identified by WG-IDENT; 
Row~4 shows a 1D spatial profile at $t=0.06$ overlaying the true (blue) and identified (red) trajectories; 
Row~5 shows the point-wise absolute errors between the simulated and true solutions across the computational domain, $x\in[0,2]$, $t\in[0,0.1]$.}
    \label{fig:2d3d_Inviscidburgerdiff}
\end{figure}

We then evaluate the identification performance of WG-IDENT on this more challenging setting using four quantitative metrics: the relative error \( E_2 \), the residual error \( E_{\text{res}} \), the True Positive Rate (TPR), and the Positive Predictive Value (PPV). Each experiment is repeated 50 times under varying noise levels \( \sigma_{\text{NSR}} \in \{1\%, 3\%, 5\%, 8\%, 10\%\} \). Figures~\ref{fig:inviscid_Burgers_equation_with_a_discontinuity} presents box plots illustrating the distributions of these metrics. For inviscid Burgers' equation, \( E_2 \) and \( E_{\text{res}} \) remain consistently low across noise levels, while TPR and PPV are close to unity. This demonstrates that WG-IDENT can robustly recover the true governing structure even in the presence of shocks, noise, and a coefficient function that is considerably more general than first-degree trigonometric. In particular, Figure~\ref{fig:2d3d_Inviscidburgerdiff} further illustrates that the reconstructed varying coefficients, the identified PDE, and the corresponding solutions remain highly consistent with the underlying truth across all tested noise levels, highlighting the robustness of WG-IDENT in accurately capturing both governing dynamics and spatially varying coefficients under noisy observations.

\subsection{Robustness tests under higher noise levels} \label{appendixd5:highernoiselevel}
To further examine the robustness of WG-IDENT, we conducted additional experiments on two representative PDEs: the viscous Burgers’ equation and the advection--diffusion equation. In these tests, we extended the noise-to-signal ratio (NSR) beyond the levels reported in the main text in order to probe the stability limits of the method.  

For the viscous Burgers’ equation, we increased the NSR up to $12\%$. Figure~\ref{fig:robustness_TPR_PPV} reports the True Positive Rate (TPR) and Positive Predictive Value (PPV) for different noise levels ($0\%$, $3\%$, $5\%$, $8\%$, $10\%$, and $12\%$). The results show that WG-IDENT achieves perfect identification (TPR and PPV close to one) for NSR up to $10\%$, consistently recovering the correct set of terms without spurious selections. When the noise level is further increased to $12\%$, however, the algorithm becomes less stable, occasionally leading to missed or redundant terms.  

For the advection--diffusion equation, we pushed the NSR much higher, up to $40\%$. Specifically, we evaluated TPR and PPV at noise levels $0\%$, $5\%$, $10\%$, $20\%$, $30\%$, and $40\%$. The results confirm that WG-IDENT remains remarkably robust: even with $20\%$ additive noise, both TPR and PPV remain equal to 1 across 50 repeated trials. Performance begins to degrade only under extreme noise contamination ($30\%$--$40\%$ NSR), where occasional missed or spurious terms appear.  

Overall, these findings demonstrate that WG-IDENT is highly reliable under practically relevant noise levels, consistently achieving perfect recovery up to $10\%$ NSR across both equations, and even tolerating significantly higher noise (up to $20\%$) for the advection--diffusion equation before noticeable degradation occurs.

\begin{figure}[t!]
    \centering
    \begin{tabular}{cc}
        \hline
        \multicolumn{2}{c}{Viscous Burgers' equation} \\
        \hline
        \hspace{1.0cm} TPR  & PPV \\
        \includegraphics[width=0.35\linewidth]{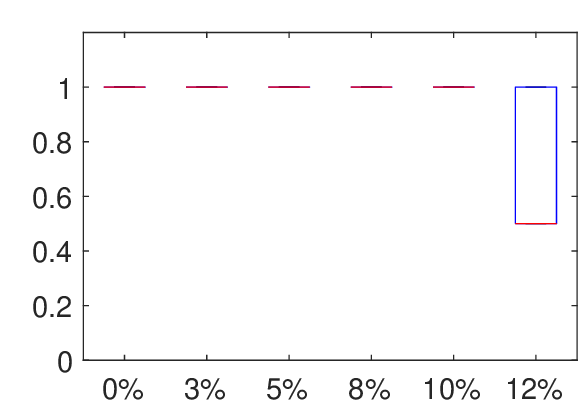} &
        \includegraphics[width=0.35\linewidth]{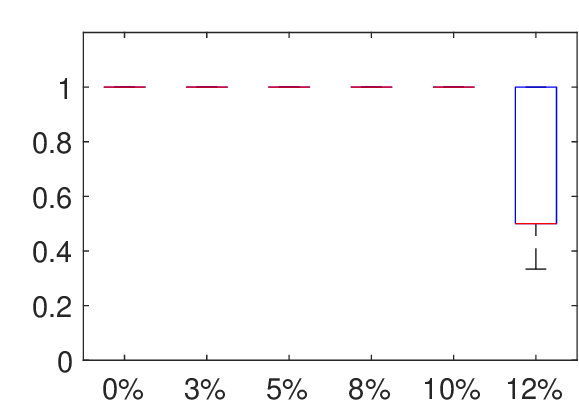} \\
        \hline
        \multicolumn{2}{c}{Advection diffusion equation} \\
        \hline
        \hspace{1.0cm} TPR  & PPV \\
        \includegraphics[width=0.35\linewidth]{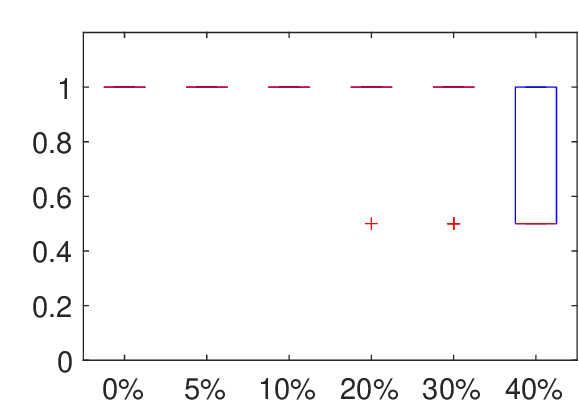} &
        \includegraphics[width=0.35\linewidth]{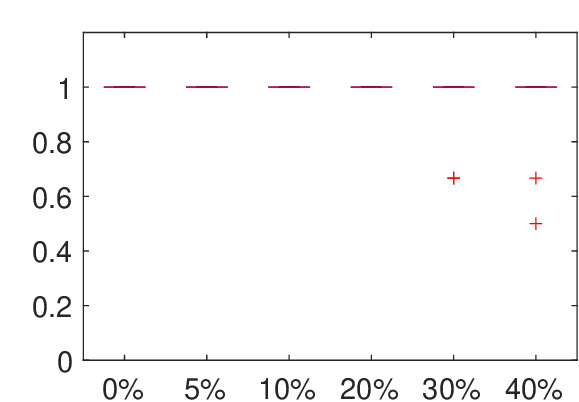}

        \\

        \hline
    \end{tabular}

    \caption{Robustness tests of WG-IDENT under higher noise levels. 
    (First row) Viscous Burgers’ equation: TPR and PPV are reported for NSR levels 
    $0\%$, $3\%$, $5\%$, $8\%$, $10\%$, and $12\%$. WG-IDENT consistently achieves perfect identification up to $10\%$ NSR, with instability observed only at $12\%$. 
    (Second row) Advection--diffusion equation: TPR and PPV are reported for NSR levels 
    $0\%$, $5\%$, $10\%$, $20\%$, $30\%$, and $40\%$. WG-IDENT maintains perfect recovery up to $20\%$ NSR, while performance degrades only at extreme noise levels ($30$--$40\%$).}
    
    \label{fig:robustness_TPR_PPV}
\end{figure}

\subsection{Additional coefficient recovery experiments} \label{appendixd7:coeffrecovery}
We include here more coefficient recovery experiments for additional validations. 
Figures~\ref{fig:coef2_reconstruction_advectiondiff}–\ref{fig:coef2_reconstruction_Sch_vt} present the reconstructed coefficients for the advection–diffusion equation, the viscous Burgers’ equation, and the Schr\"{o}dinger equation (both real and imaginary components) under noise levels of $3\%$, $5\%$, and $10\%$. 
In all cases, the estimated coefficients closely match the ground-truth values, which further demonstrates the robustness and effectiveness of WG-IDENT in recovering varying coefficients across a diverse set of PDEs.

\begin{figure}[t!]
    \centering
    \begin{tabular}{cccc}
    \hline
        \multicolumn{4}{c}{Advection diffusion equation}\\\hline
             &  $\sigma_{\text{NSR}}=3\%$ & $5\%$ & $10\%$ \\
    \hline
        $u_x$ & 
        \raisebox{-0.5\height}{\includegraphics[width=0.25\linewidth]{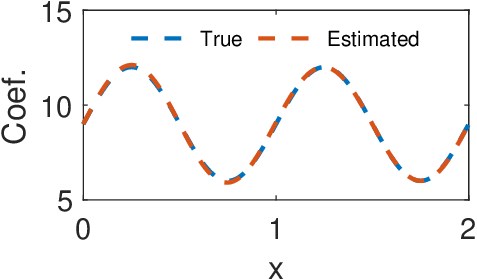}} &
        \raisebox{-0.5\height}{\includegraphics[width=0.25\linewidth]{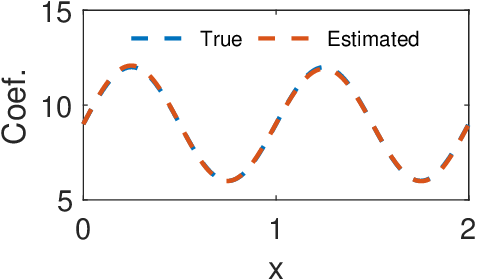}}&
        \raisebox{-0.5\height}{\includegraphics[width=0.25\linewidth]{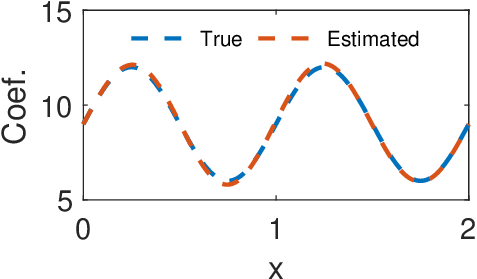}}\\
    \hline 
        $u_{xx}$ &
        \raisebox{-0.5\height}{\includegraphics[width=0.25\linewidth]{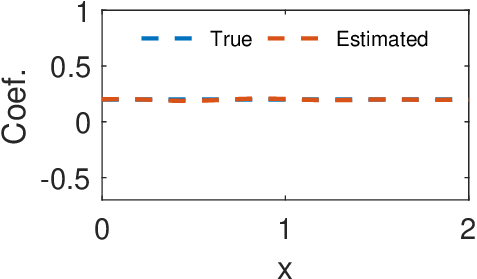}} &
        \raisebox{-0.5\height}{\includegraphics[width=0.25\linewidth]{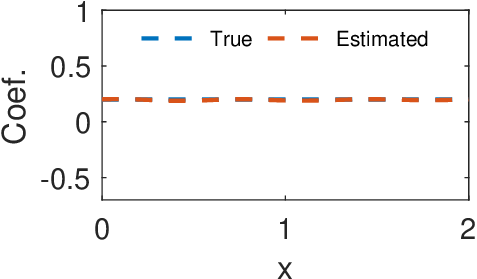}}&
        \raisebox{-0.5\height}{\includegraphics[width=0.25\linewidth]{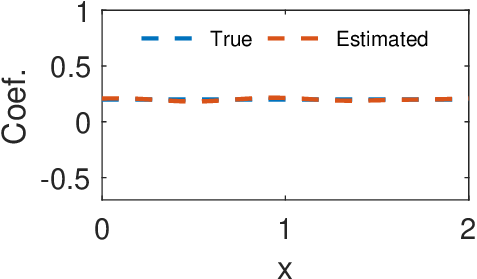}}\\
    \hline
    \end{tabular}
    \caption{Estimated coefficients of different terms in Advection diffusion equation with $3\%$, $5\%$ and $10\%$ noise. The first row shows the estimated coefficients of $u_x$. The second row shows the estimated coefficients of $u_{xx}$. }
    \label{fig:coef2_reconstruction_advectiondiff}
\end{figure}

\begin{figure}[t!]
    \centering
    \begin{tabular}{cccc}
    \hline
        \multicolumn{4}{c}{Viscous Burgers' equation}\\\hline
             &  $\sigma_{\text{NSR}}=3\%$ & $5\%$ & $10\%$ \\
    \hline
        $uu_x$ & 
        \raisebox{-0.5\height}{\includegraphics[width=0.25\linewidth]{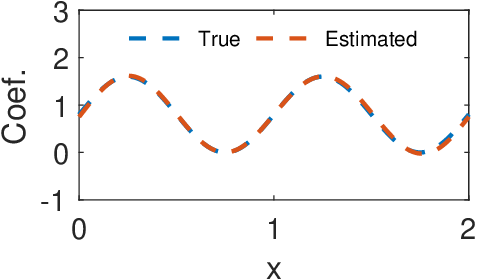}} &
        \raisebox{-0.5\height}{\includegraphics[width=0.25\linewidth]{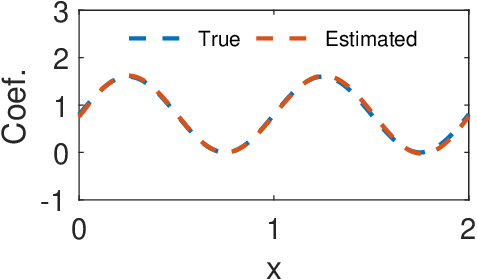}}&
        \raisebox{-0.5\height}{\includegraphics[width=0.25\linewidth]{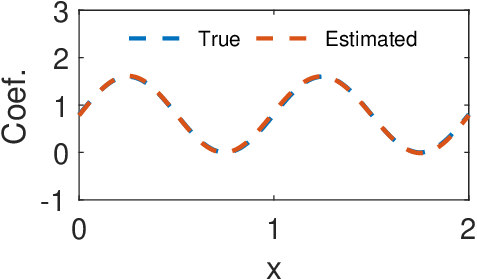}}\\
    \hline 
        $u_{xx}$ &
        \raisebox{-0.5\height}{\includegraphics[width=0.25\linewidth]{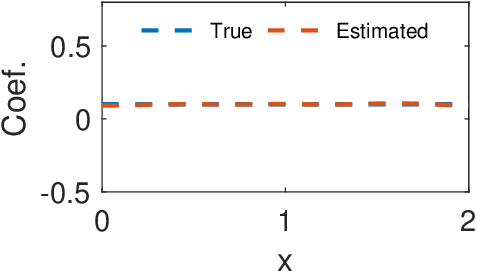}} &
        \raisebox{-0.5\height}{\includegraphics[width=0.25\linewidth]{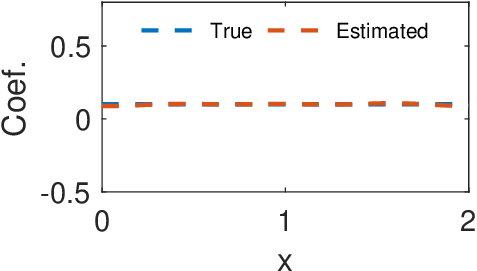}}&
        \raisebox{-0.5\height}{\includegraphics[width=0.25\linewidth]{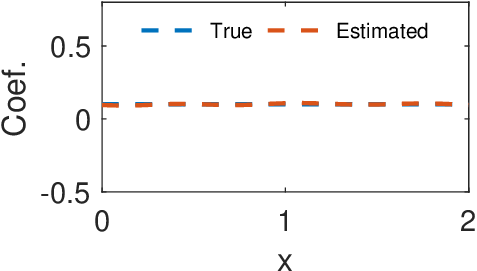}}\\
    \hline
    \end{tabular}
    \caption{Estimated coefficients of different terms in Viscous Burgers' equation with $3\%$, $5\%$ and $10\%$ noise. The first row shows the estimated coefficients of $uu_x$. The second row shows the estimated coefficients of $u_{xx}$. }
    \label{fig:coef2_reconstruction_visburgers}
\end{figure}

\begin{figure}[t!]
    \centering
    \begin{tabular}{cccc}
    \hline
    \multicolumn{4}{c}{Schr\"{o}dinger equation real part}\\\hline
         &  $\sigma_{\text{NSR}}=3\%$ & $5\%$ & $10\%$ \\
    \hline
        $w$ & 
        \raisebox{-0.5\height}{\includegraphics[width=0.25\linewidth]{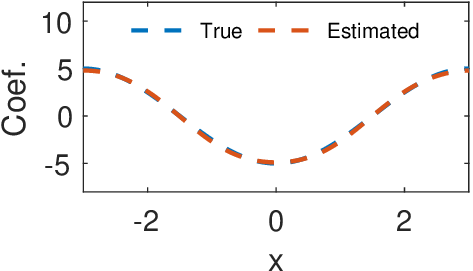}} &
        \raisebox{-0.5\height}{\includegraphics[width=0.25\linewidth]{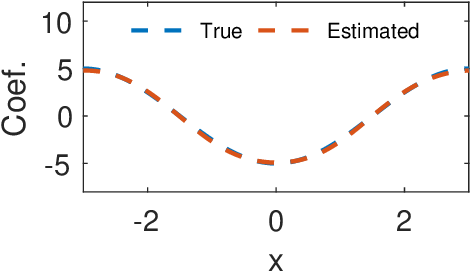}}&
        \raisebox{-0.5\height}{\includegraphics[width=0.25\linewidth]{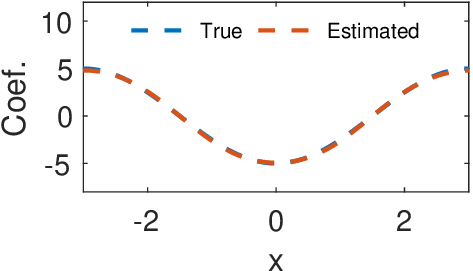}}\\
    \hline 
        $w_{xx}$ &
        \raisebox{-0.5\height}{\includegraphics[width=0.25\linewidth]{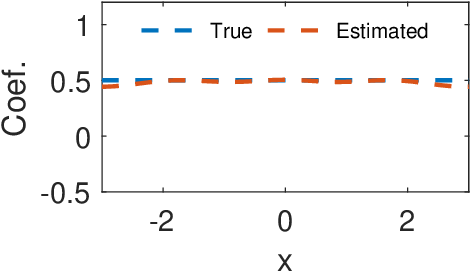}} &
        \raisebox{-0.5\height}{\includegraphics[width=0.25\linewidth]{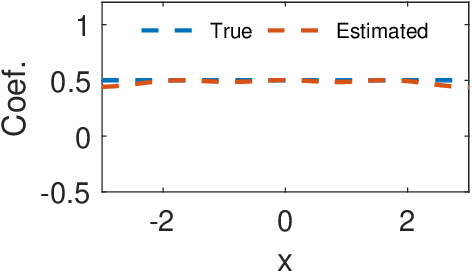}}&
        \raisebox{-0.5\height}{\includegraphics[width=0.25\linewidth]{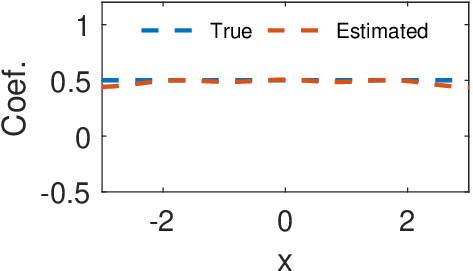}}\\
    \hline
    \end{tabular}
    \caption{Estimated coefficients of the real components of different terms in Sch equation with $3\%$, $5\%$ and $10\%$ noise. The first row shows the estimated coefficients of $w$. The second row shows the estimated coefficients of $w_{xx}$. }
    \label{fig:coef2_reconstruction_Sch_Ut}
\end{figure}

\begin{figure}[t!]
    \centering
    \begin{tabular}{cccc}
    \hline
    \multicolumn{4}{c}{Schr\"{o}dinger equation imaginary part}\\\hline
         &  $\sigma_{\text{NSR}}=3\%$ & $5\%$ & $10\%$ \\
    \hline
        $v$ & 
        \raisebox{-0.5\height}{\includegraphics[width=0.25\linewidth]{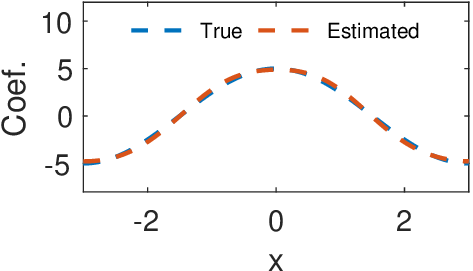}} &
        \raisebox{-0.5\height}{\includegraphics[width=0.25\linewidth]{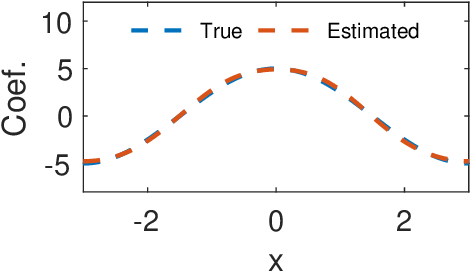}}&
        \raisebox{-0.5\height}{\includegraphics[width=0.25\linewidth]{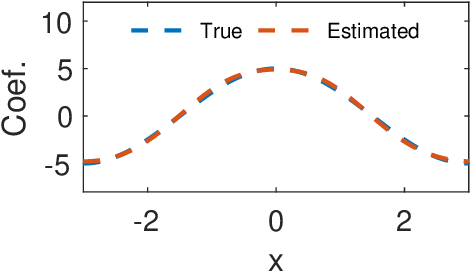}}\\
    \hline 
        $v_{xx}$ &
        \raisebox{-0.5\height}{\includegraphics[width=0.25\linewidth]{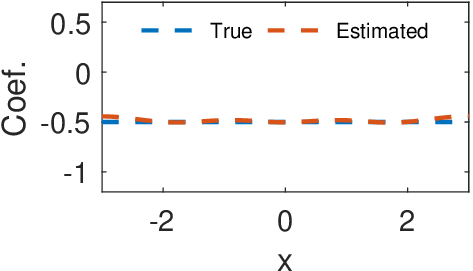}} &
        \raisebox{-0.5\height}{\includegraphics[width=0.25\linewidth]{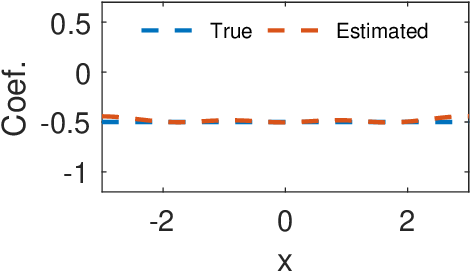}}&
        \raisebox{-0.5\height}{\includegraphics[width=0.25\linewidth]{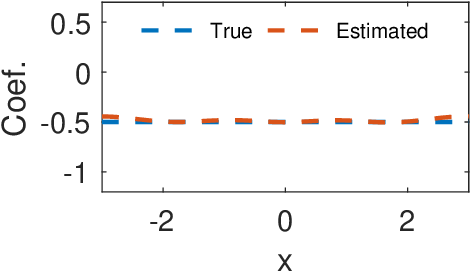}}\\
    \hline
    \end{tabular}
    \caption{Estimated coefficients of the imaginary components of different terms in Sch equation with $3\%$, $5\%$ and $10\%$ noise. The first row shows the estimated coefficients of $v$. The second row shows the estimated coefficients of $v_{xx}$. }
    \label{fig:coef2_reconstruction_Sch_vt}
\end{figure}

\subsection{Example of group feature selection across sparsity levels}\label{appendixd6:gftrim}

To provide further clarification on the behavior of WG-IDENT across different sparsity levels, we include here a representative example using the Kuramoto--Sivashinsky (KS) equation. In classical approaches such as Group Lasso, the selected feature sets are nested, meaning that the set identified at sparsity level $\theta+1$ always contains all groups from $\theta$. In contrast, our framework based on Group Projected Subspace Pursuit (GPSP) combined with Group Feature Trimming (GF-Trim) does not impose such monotonicity. As a result, the active feature groups identified at different $\theta$ may not form a strictly additive sequence.

Table~\ref{tab:KS_sparsity_example} presents the identified groups for the KS equation at sparsity levels $\theta=1$ through $\theta=6$, both before and after GF-Trim. The correct active indices for this equation are $[4,6,8]$. We observe that, while the candidate sets at certain $\theta$ include spurious features (e.g., indices [10,13,15] at $\theta=4$), these are consistently removed after GF-Trim, leading to recovery of the correct feature set once $\theta \geq 3$.

\begin{table}[t!]
\footnotesize
\centering
{\centering \bfseries KS equation \par}
\begin{tabular}{cll}
\toprule
Sparsity level & Before GF-Trim & After GF-Trim \\
\midrule
$\theta = 1$ & [8]              & [8] \\
$\theta = 2$ & [6,8]            & [6,8] \\
$\theta = 3$ & [4,6,8]          & [4,6,8] \\
$\theta = 4$ & [8,10,13,15]     & [8,10,15] \\
$\theta = 5$ & [4,6,8,9,10]     & [4,6,8] \\
$\theta = 6$ & [4,6,8,9,13,14]  & [4,6,8] \\
\bottomrule
\end{tabular}
\caption{Example of feature groups identified for the KS equation at different sparsity levels. The true indices are [4,6,8].}
\label{tab:KS_sparsity_example}
\end{table}

This example highlights two important aspects:  
(i) the identified groups across $\theta$ are not necessarily nested, unlike the behavior in group LASSO; and  
(ii) the GF-Trim procedure plays a crucial role in eliminating low-contribution groups, thereby stabilizing the recovery process. Together, these observations emphasize the importance of the refinement stage in WG-IDENT and clarify the interpretation of feature selection across different sparsity levels.

\subsection{Comparison of weak formulation in space-time domain and only in space} \label{appendix_wfinspace}
In our formulation, the weak formulation is computed by integrating over the space-time domain, so that both spatial and temporal derivatives are transferred to test functions, see (\ref{eq_model_general_weak}). Another strategy is to only utilize the weak formulation in space, and use numerical differentiation to calculate time derivative with the the successively denoised differentiation (SDD) \cite{he2022robust} method. In this subsection, we compare these two strategies and demonstrate the advantages of using space-time weak formulation.

In the presence of noisy data, numerical differentiation can amplify noise, leading to a substantial amplification of errors in the time derivative. To address this issue in the second strategy, we adopt SDD, which computes the time derivative as
$$
\partial_t u_i^n \approx S_{(t)}D_tS_{(\textbf{x})}[U_i^n],
$$
where $S_{(t)},S_{(x)}$ are smoothing operators in time and space, and $D_t$ is a numerical differentiation scheme. By setting $S(t)$ as 
$$
S_{(t)}[U_\textbf{i}^n] = p_\textbf{i}^n(t^n), \text{ where } p_\textbf{i}^n = \argmin_{p\in P_2}\sum_{0\leq k\leq N}(p((t^k)-U_\textbf{i}^n))^2\exp\left (-\frac{||t^n-t^k||^2}{\gamma^2}\right ),
$$ 
where $P_2$ is the set of polynomials with a degree not exceeding 2, and $\gamma$ is a paramter controlling the weight. The operator $S_{(x)}$ can be defined similarly. For $D_t$, we employ a 4-point central difference scheme given by
$$
D_t U_i^n=\frac{-U_i^{n+2}+8U_i^{n+1}-8U_i^{n-1}+U_i^{n-2}}{12\Delta t}.
$$

We compare our method with the two strategies for identifying viscous Burgers' equation and KS equation. The $E_2$ values of the results are shown in Figure \ref{fig:SDD_phix_N_phixt}. For both PDE, the approach using SDD in the time domain and weak formulation in the spatial domain results in higher $E_2$ values. While using weak formulation in space-time domain gives constantly smaller $E_2$, demonstrating the benefits of this strategy.

\begin{figure}[t!]
{\footnotesize
    \centering
    \begin{tabular}{cc|cc}
    \hline
    \multicolumn{2}{c|}{Viscous Burgers' equation} & \multicolumn{2}{c}{KS equation}\\
    \hline
    strategy 1 & strategy 2 & strategy 1 & strategy 2\\
    \includegraphics[width=0.22\linewidth]{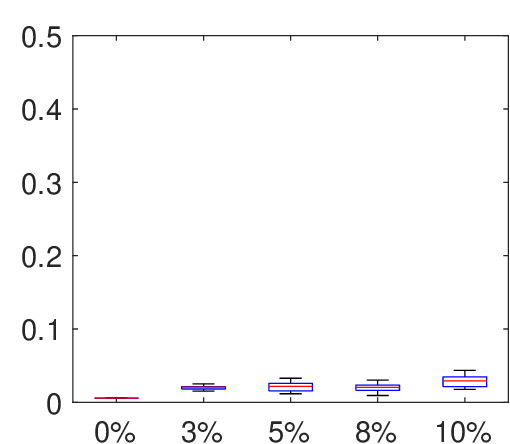}&
    \includegraphics[width=0.22\linewidth]{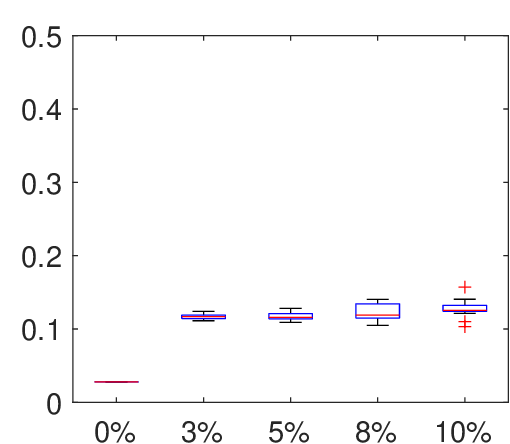}&
    \includegraphics[width=0.22\linewidth]{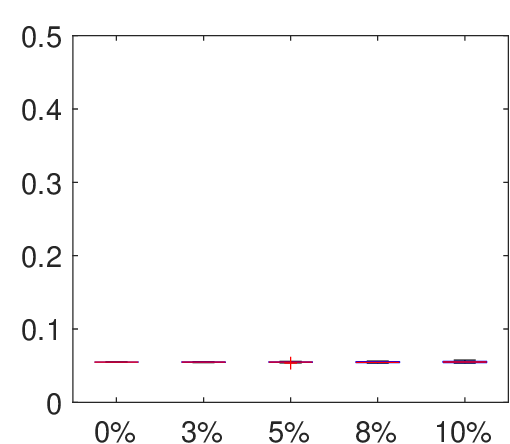}&
    \includegraphics[width=0.22\linewidth]{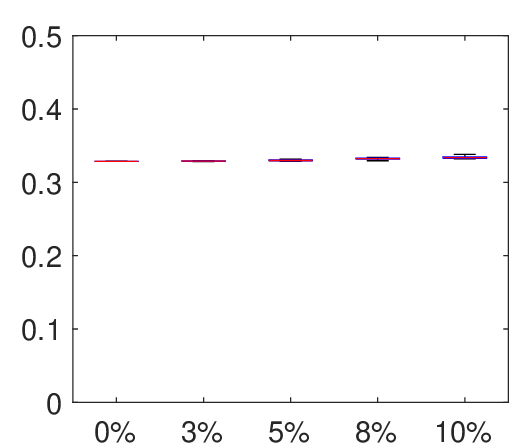}\\
    \hline 
    \end{tabular}
    }
    \caption{Comparison of weak formulation in space-time domain 
    (strategy 1) and only in space (strategy 2). The plots show the $E_2$ value ($y$-axis) of different strategies in identifying viscous Burgers' equation and KS equation with various noise levels ($x$-axis). In strategy 1, weak formulation is used in space-time domain. In strategy 2, weak formulation is only used in space.
    }
    \label{fig:SDD_phix_N_phixt}
\end{figure}

\subsection{More comparisons with the state-of-the-art methods}\label{appendix_comparemethods}
We presents in Table \ref{tab:AdvectionKdVKS_eq} a detailed comparison of WG-IDENT (the proposed method) with GLASSO, SGTR, and rSGTR on the advection-diffusion equation, KdV equation, and KS equation, as a complement for the comparison in Section \ref{sec:comparison_study}. In this comparison, the dictionary consists of 16 features, incorporating partial derivatives of $u$, $u^2$ and $u^3$ up to the fourth order. The results demonstrate that our method consistently identifies the correct features across various noise levels, highlighting its robustness compared to other methods.

\begin{table}[t!]
\centering
\footnotesize
{\centering \bfseries (a) Advection-Diffusion Equation \par}
\begin{tabular}{ccccc}
\toprule
Noise level & WG-IDENT & GLASSO & SGTR & rSGTR \\
\midrule
No Noise & $u_x, u_{xx}$ & $u_x, u_{xx}$ & $u_x, u_{xx}$ & $u_x, u_{xx}$ \\
0.1\%  & $u_x, u_{xx}$ & $u_x, u_{xx}$ & $u_x, u_{xx}$ & $u_x, u_{xx}$ \\
0.5\%  & $u_x, u_{xx}$ & $u_x, u_{xx}$ & $u_x, u_{xx}$ & $u_x, u_{xx}$ \\
1\%    & $u_x, u_{xx}$ & $u_x, u_{xx}$ & $u_x, u_{xx}$ & $u_x, u_{xx}$ \\
5\%    & $u_x, u_{xx}$ & $u_x, u_{xx}$ & $u_x, u_{xx}$ & $u_x, u_{xx}$  \\
8\%    & $u_x, u_{xx}$ & $\geq 4\ \text{terms}$ & $u_x, u_{xxxx}$ & $u_x, u_{xxxx}$  \\
10\%   & $u_x, u_{xx}$ & $\geq 4\ \text{terms}$ & $\geq 4\ \text{terms}$ & $\geq 4\ \text{terms}$  \\
\bottomrule
\end{tabular}
\vspace{0.5cm}

{\centering \bfseries (b) KdV Equation \par}
\begin{tabular}{ccccc}
\toprule
Noise level & WG-IDENT & GLASSO & SGTR & rSGTR \\
\midrule
No Noise & $uu_x, u_{xxx}$ & $3\ \text{terms}$ & $uu_x, u_{xxx}$ & $uu_x, u_{xxx}$ \\
0.1\%  & $uu_x, u_{xxx}$ & $3\ \text{terms}$ & $uu_x, u_{xxx}$ & $uu_x, u_{xxx}$ \\
0.5\%  & $uu_x, u_{xxx}$ & $3\ \text{terms}$ & $\geq 4\ \text{terms}$ & $\geq 4\ \text{terms}$ \\
1\%    & $uu_x, u_{xxx}$ & $3\ \text{terms}$ & $\geq 4\ \text{terms}$ & $\geq 4\ \text{terms}$ \\
5\%    & $uu_x, u_{xxx}$ & $3\ \text{terms}$ & $\geq 4\ \text{terms}$ & $\geq 4\ \text{terms}$  \\
8\%    & $uu_x, u_{xxx}$ & $3\ \text{terms}$ & $\geq 4\ \text{terms}$ & $\geq 4\ \text{terms}$  \\
10\%   & $uu_x, u_{xxx}$ & $4\ \text{terms}$ & $\geq 4\ \text{terms}$ & $\geq 4\ \text{terms}$  \\
\bottomrule
\end{tabular}
\vspace{0.5cm}

{\centering \bfseries (c) KS Equation \par}
\begin{tabular}{ccccc}
\toprule
Noise level & WG-IDENT & GLASSO & SGTR & rSGTR \\
\midrule
No Noise & $uu_x, u_{xx}, u_{xxxx}$ & $\geq 4\ \text{terms}$ & $uu_x, u_{xx}, u_{xxxx}$ & $uu_x, u_{xx}, u_{xxxx}$ \\
0.1\%  & $uu_x, u_{xx}, u_{xxxx}$ & $\geq 4\ \text{terms}$ & $\geq 4\ \text{terms}$ & $\geq 4\ \text{terms}$ \\
0.5\%  & $uu_x, u_{xx}, u_{xxxx}$ & $\geq 4\ \text{terms}$ & $\geq 4\ \text{terms}$ & $\geq 4\ \text{terms}$ \\
1\%    & $uu_x, u_{xx}, u_{xxxx}$ & $\geq 4\ \text{terms}$ & $\geq 4\ \text{terms}$ & $\geq 4\ \text{terms}$ \\
5\%    & $uu_x, u_{xx}, u_{xxxx}$ & $\geq 4\ \text{terms}$ & $\geq 4\ \text{terms}$ & $\geq 4\ \text{terms}$  \\
8\%    & $uu_x, u_{xx}, u_{xxxx}$ & $\geq 4\ \text{terms}$ & $\geq 4\ \text{terms}$ & $\geq 4\ \text{terms}$  \\
10\%   & $uu_x, u_{xx}, u_{xxxx}$ & $\geq 4\ \text{terms}$ & $\geq 4\ \text{terms}$ & $\geq 4\ \text{terms}$  \\
\bottomrule
\end{tabular}
\caption{Comparison of results for the rest of the PDEs: (a) advection-diffusion equation, (b) KdV equation, and (c) KS equation. The table presents the identified features by our method, GLASSO, SGTR, and rSGTR across various noise levels. Our method consistently identifies the correct features at all noise levels.}
\label{tab:AdvectionKdVKS_eq}
\end{table}

\end{document}